\documentclass[preprint,12pt]{elsarticle}

\usepackage{amssymb}
\usepackage{amsthm}
\usepackage{mathrsfs}
\usepackage{enumerate}
\usepackage{amsmath}
\usepackage{amsfonts}
\usepackage{amssymb}
\usepackage{multicol}
\usepackage{bm}
\usepackage[table]{xcolor}
\usepackage{caption}
\usepackage{tikz}
\usepackage{graphicx}
\usepackage{mathtools} 
\usepackage{titlesec}    
\usepackage{etoolbox}
\usepackage[utf8]{inputenc}
\usepackage[english]{babel}
\usepackage{float}
\restylefloat{table}
\usepackage{multirow}
\makeatletter
\patchcmd{\ttlh@hang}{\parindent\z@}{\parindent\z@\leavevmode}{}{}
\patchcmd{\ttlh@hang}{\noindent}{}{}{}
\makeatother
\usepackage{titletoc}
\usepackage{fancyhdr}
\usepackage{indentfirst}
\usepackage[margin=0.8in,vmargin=2.7cm]{geometry}
  
\usepackage{mdframed}
\newmdenv[skipabove=4mm,linewidth=1]{MyFrame}   

\numberwithin{equation}{section}

\journal{arxiv}

\newtheorem{theorem}{Theorem}[section]

\newtheorem{lemma}[theorem]{Lemma}
\newtheorem{remark}[theorem]{Remark}

\begin{document}
\begin{frontmatter}

\title{Mass-conservative and positivity preserving second-order semi-implicit methods for high-order parabolic equations}

\author[addr1,addr2]{Sana Keita\corref{corres}}     \ead{sana.keita@um6p.ma}
\cortext[corres]{Corresponding author.}
\author[addr1,addr2]{Abdelaziz Beljadid} \ead{abdelaziz.beljadid@um6p.ma}
\author[addr2]{Yves Bourgault}     \ead{ybourg@uottawa.ca}
\address[addr1]{International Water Research Institute, Mohammed VI Polytechnic University,  Morocco} 
\address[addr2]{Department of Mathematics and Statistics, University of Ottawa, Canada \qquad\qquad\qquad\qquad}

\begin{abstract}
We consider a class of finite element approximations for fourth-order parabolic equations that can be written as a system of second-order equations by introducing an auxiliary variable. In our approach, we first solve a variational problem and then an optimization problem to satisfy the desired physical properties of the solution such as conservation of mass, positivity (non-negativity) of solution and dissipation of energy. Furthermore, we show existence and uniqueness of the solution to the optimization problem and we prove that the methods converge to the truncation schemes \cite{Berger1975}.  We also propose new conservative truncation methods for high-order parabolic equations. A numerical convergence study is performed and a series of numerical tests are presented to show and compare the efficiency and robustness of the different schemes.
\end{abstract}

\end{frontmatter}

\section{Introduction}
High-order partial differential equations often arise in systems from fluid and solid mechanic problems for industrial applications. For instance, the equations for thin film liquid motion \cite{ehrhard_davis_1991,greenspan1978,Hocking1981,Huppert1982,Moriarty1991}, the Cahn-Hilliard (CH) equations for phase separation modeling in a binary mixture \cite{CAHN1962,CahnHilliard1958}, the fourth-order model for the approximation of crack propagation \cite{AMIRI2016}, models for tumor-growth  \cite{Colli2017,Wise2008}, models for biological entities \cite{Elson2010,Khain2008}, models for ocean and atmosphere circulation \cite{DELHEZ2007}, image processing \cite{Bertozzi4}, models of unstable infiltration through porous media \cite{Beljadid2020,Cueto2009}, etc. 

Discretization of high-order spatial operators is a major issue encountered when numerically solving these equations. As a result, several different procedures have been employed over the years to deal with higher-order equations, for instance, mixed formulations, where the governing equations are split into a coupled system of lower-order differential equations \cite{BoffiBrezziFortin2013,Elliott1989}.

Positivity of solutions of second-order parabolic problems mainly follows from the positivity of the initial data and the form of boundary data \cite{Aronson1967,konkov2018}. Higher-order partial differential equations are different from second-order partial differential equations mainly in the lack of maximum principle, i.e.\ solutions which are initially positive may become negative in certain regions over time \cite{Fbernis1987}. Although there exists positive or bounded solutions in some particular cases (see the discussion in section \ref{SectModelEquation}), the computations of high-order equations require special numerical techniques in order to preserve the boundedness and/or positivity of the solutions. In fact, as mentioned in \cite{Bertozzi1995}, solutions which are initially strictly positive can reach the value of zero in some regions for a finite interval of time and become again strictly positive. There have been many studies on numerical methods for  high-order parabolic equations that inherit the positivity, boundedness, mass conservation or energy dissipation property from the physical equations \cite{Barrett3,Barrett1,Barrett2,Bertozzi1,Bertozzi6,Grun2000,SHIN2011,YE2005,Bertozzi3}. Numerical methods where positivity is weakly imposed by using variational inequalities are proposed in \cite{Barrett3,Barrett1,Barrett2}. However, such schemes turned out to be very slow as we will see in section~\ref{SectNumResults}. Numerical schemes using entropy to preserve positivity of the discrete solution are proposed in \cite{Grun2000,Bertozzi3}. However, these schemes are nonlinear and only first-order accurate in time. Moreover, the iterative procedures to solve the nonlinear systems require good preconditioners, a good initial guess of the solution and a large number of iterations to converge at each time step \cite{Bertozzi1}. 

In this study, we focus on the numerical approximation of a parabolic equation involving the bi-harmonic operator. We propose several numerical methods that preserve the mass and satisfy the boundedness property of the solution, whenever the solution of the continuous problem should have such property. The main idea consists in first solving a variational problem  and then an optimization problem to satisfy the desired physical properties (conservation of mass, positivity or boundedness of solutions and dissipation of energy) on the solution. We prove that the methods converge to the truncation schemes introduced in \cite{Berger1975} for second-order parabolic equations. We also propose new techniques to build conservative truncation methods for high-order parabolic equations.  The proposed schemes maintain second-order accuracy in time by using the temporal approximation techniques proposed in \cite{Keita2021}. These techniques mainly consist in using second-order backward difference formula for the time derivative, an extrapolation formula and a special technique (using Taylor approximation) for the discretization of the nonlinear terms. This makes the proposed schemes efficient in terms of computational cost since the numerical method developed to solve the variational problem only deals with a linear system at each time step. However, as opposed to the present study, positivity and/or boundedness of the computed solution are not guaranteed in \cite{Keita2021}. We perform numerical convergence analysis in space using both a manufactured and analytical solutions to demonstrate optimal convergence rates of the $L^2$- and $H^1$-error norms for sufficiently regular solutions. 

The paper is organized as follows. In section~\ref{SectModelEquation}, the model equation is introduced and the existence of solutions to the problem is discussed. In section~\ref{SectionNumMethods}, an analysis of the optimization problem is performed  and several finite element methods resulting from the analysis are proposed.  In Section~\ref{SectNumResults}, the spatial convergence rates, accuracy and robustness of the proposed schemes are numerically investigated by means of  numerical experiments. Second-order time-accuracy, energy dissipation and mass conservation properties of the proposed schemes are illustrated. Section~\ref{SectConclusion} provides a summary and concluding remarks.

\section{Model equations}
\label{SectModelEquation}
Let $\Omega\subset\mathbb{R}^d$, where $d=2$ or $3$, be an open bounded set with sufficiently smooth boundary $\partial \Omega$, and $T>0$ a fixed time. We focus on a class of fourth-order parabolic differential equation of the form
\begin{equation}
\partial_tu+\gamma\nabla\cdot\big(f(u)\nabla \Delta u\big) - \nabla\cdot\big(f(u)\nabla \varphi'(u)\big)=0 \quad \text{in } \Omega_T:=\Omega\times(0,T),
\label{GenThinFilmEquation}
\end{equation}
with unknown $u:\Omega\times[0,T]\rightarrow\mathbb{R}$, the given functions $f:\mathbb{R}\rightarrow\mathbb{R}_{\geqslant 0}$ and $\varphi:\mathbb{R}\rightarrow\mathbb{R}$ are  smooth and $\gamma>0$ is a constant. The function $f$ is a diffusional or mobility term and $\varphi$ is a free energy function. Depending on the applications, $u$ can describe the height of a liquid phase spreading on a solid surface \cite{greenspan1978,Hocking1981},  the concentration, volume fraction or density of a phase in a mixture \cite{CAHN1962,CahnHilliard1958}, etc. 

By introducing an auxiliary variable $w$, the fourth-order equation \eqref{GenThinFilmEquation} can be written in the following system of second-order equations
\begin{equation}
\begin{aligned}
&\partial_tu-\nabla\cdot\big(f(u)\nabla w\big)=0 \quad \text{in } \Omega_T,\\
&w=-\gamma\Delta u + \varphi'(u) \quad \text{in } \Omega_T.
\end{aligned}
\label{ThinFilmSyst}
\end{equation}
System \eqref{ThinFilmSyst} is to be supplemented by the initial condition
\begin{equation}
u(\bm{x},0)=u_0(\bm{x}) \quad \text{in } \Omega,
\label{ThinFilmSystInitialCondition}
\end{equation}
and the no-flux boundary conditions
\begin{equation}
\nabla u\cdot\bm{n}=\nabla w\cdot\bm{n}=0 \quad \text{on } \partial\Omega\times(0,T),\\
\label{HomoNeumannBoundCond}
\end{equation}
where $\bm{n}$ is the exterior unit normal vector to $\partial\Omega$.

There are two important properties for the  equations \eqref{ThinFilmSyst} with the natural boundary conditions \eqref{HomoNeumannBoundCond}. The first one is the conservation of mass. In fact,
\begin{equation}
\dfrac{d}{dt}\int_{\Omega}u\,d\bm{x}=\int_{\Omega}\nabla\cdot\big(f(u)\nabla w\big)\,d\bm{x}=\int_{\partial\Omega}f(u)\nabla w\cdot\bm{n}\,d\bm{x}=0\text{ }\Longrightarrow\int_{\Omega}u\,d\bm{x}=\int_{\Omega}u_0\,d\bm{x},
\label{MassConservation}
\end{equation}
that is the total amount of the phase in the domain must always be equal to the given initial amount of this phase. The equations \eqref{ThinFilmSyst} are associated with the total free energy 
\begin{equation}
J(u)=\int_{\Omega}\left(\frac{\gamma}{2}\vert\nabla u\vert^2+\varphi(u)\right)\,d\bm{x},
\label{EnergyFunctional}
\end{equation}
which satisfies the energy dissipation property, that is
\begin{equation}
\dfrac{d J}{dt}=\int_{\Omega}\big(\gamma\nabla u \nabla \partial_tu+\varphi'(u)\partial_tu\big)\,d\bm{x}=\int_{\Omega}w \partial_tu\,d\bm{x}=-\int_{\Omega}f(u)\vert\nabla w\vert^2\,d\bm{x} \leqslant 0.
\label{EnergyProperty}
\end{equation}

A weak formulation for problem \eqref{ThinFilmSyst} can be obtained by multiplying the first and second equation in \eqref{ThinFilmSyst} by test functions $v$ and $q$, respectively: Find $\big(u ,w\big) \in H^1(\Omega)\times H^1(\Omega)$ such that
\begin{equation}
\begin{aligned}
&\int_{\Omega}u_t v \,d\bm{x}+\int_{\Omega} f(u)\nabla w\cdot\nabla v\, d\bm{x}=0,\quad\forall v \in H^1(\Omega),\\
&\int_{\Omega} w q\,d\bm{x}-\int_{\Omega}\gamma\nabla u\cdot\nabla q\,d\bm{x}-\int_{\Omega}\varphi'(u) q\,d\bm{x}=0,\quad\forall q \in  H^1(\Omega).
\end{aligned}
\label{lubweakForm}
\end{equation}
The existence of a solution $(u,w)\in H^3(\Omega)\times H^3(\Omega)$ to \eqref{ThinFilmSyst} in the weak sense \eqref{lubweakForm} is shown in \cite{Elliott89} for constant  mobilities. It is proven in \cite{Elliott96} the existence of a pair of functions $(u,w)\in H^1(\Omega)\times H^1(\Omega)$  which satisfies \eqref{ThinFilmSyst} for strictly positive phase-dependent mobility functions. However, one can only expect $w$ to be in $H^1(E)$ when the mobility function goes to zero on a subset of $\Omega\setminus E$ \cite{Dai2016,Elliott96}. The degeneracy of the mobility function to zero has a major impact on the regularity of the auxiliary variable $w$. We refer to \cite{Dai2016} for more discussion on the well-posedness of \eqref{ThinFilmSyst} with degenerate phase-dependent mobility function. Many finite element methods for \eqref{GenThinFilmEquation} are based on the splitting technique \eqref{ThinFilmSyst} which uses the auxiliary variable $w$ as an unknown function and hence only requires continuous finite element approximations \cite{Barrett3,Barrett1,Barrett2,ElliottFrench1989}. In the following, we assume that the auxiliary variable $w$ belongs to $H^1(\Omega)$, for the sake of developing our numerical methods. However, our methods work well for less regular solutions, as illustrated with the first numerical test in section~\ref{SectNumResults}.

While there is no comparison principle for the equation \eqref{GenThinFilmEquation}, the existence of positive solution in the one-dimensional case with the property $0\leqslant u(x,t)\leqslant 1$, whenever the initial data satisfies $0\leqslant u_0(x)\leqslant 1$ and
\begin{equation}
f(u)\geqslant0,\text{ } \forall u\in(0,1)\text{ with }f(0)=f(1)=0 \quad\text{and}\quad \varphi^{\prime}(u)=\int_0^u \upsilon(s)\,ds,\text{ } \upsilon\in\mathcal{C}^0([0,1]),   
\end{equation}
is shown in \cite{Yin1992}. The existence of solutions bounded by $1$ in magnitude to \eqref{GenThinFilmEquation} for a specific form of degenerate mobility and logarithmic  energy is shown in \cite{Elliott96}. See sections \ref{SubsectionNullFreeEnergy} and \ref{SubsectionLogFreeEnergy}, for more details on the existence of solution to \eqref{GenThinFilmEquation} for  particular forms of mobility and  free energy functions.

\subsection{Systems with null free energy function}
\label{SubsectionNullFreeEnergy}
For $\varphi=0$, model \eqref{GenThinFilmEquation} reduces to
\begin{equation}
\partial_tu+\gamma\nabla\cdot\big(f(u)\nabla \Delta u\big)=0 \quad \text{in } \Omega_T.
\label{ThinFilmEquation}
\end{equation}
Equation \eqref{ThinFilmEquation} appears in the lubrication approximation of a thin viscous film, which is driven by surface tension. In this context, $u$ describes the height of the liquid film which spreads on a solid surface \cite{greenspan1978,Hocking1981} and hence $u$ is restricted to be positive. In many studies on the equation \eqref{GenThinFilmEquation}, the mobility function $f(u)$ is assumed to be constant. However, a $u$-dependent mobility appeared in the original derivation of the equation \cite{CAHN1962} which has many practical applications.  The typical functional forms of the mobility depend on the condition which is prescribed on the interface between the solid and the liquid. For \eqref{ThinFilmEquation}, we consider the general mobility function of the form \cite{BernisFreeman1990}
\begin{equation}
f(u)=f_0(u)\vert u\vert^p,\quad p\in (0,\infty),
\label{MobFunc}
\end{equation}
where $f_0$ being positive and sufficiently smooth such that $f(u)\sim\vert u\vert^p$ near $u=0$. Since one assumes, on physical grounds, that $u$ is positive, $\vert u\vert$ could be replaced by $u$. 

In one-dimensional space, the existence of positive solution to \eqref{ThinFilmEquation}-\eqref{MobFunc} under appropriate boundary conditions and positive initial condition is shown in \cite{BernisFreeman1990}. The exponent $p$ plays a very important role in the existence and positivity of solutions to \eqref{ThinFilmEquation}-\eqref{MobFunc}. It has been shown by \cite{BernisFreeman1990} that positivity is preserved for $p>1$ and strict positivity is preserved for $p\geqslant4$. $\mathcal{C}^1$ source-type radial self-similar positive solutions for \eqref{ThinFilmEquation}-\eqref{MobFunc} with $f_0(u)\equiv1$  are constructed in \cite{BernisetAl1992}.  However, such type of self-similar solutions with finite mass exists only for $p\in(0,3)$ \cite{BernisetAl1992,Ferreira1997}. It also has been shown by \cite{BernisFreeman1990} in a study which consists in approximating the mobility function $f$ by
\begin{equation}
f_{\varepsilon}(u)=\dfrac{f(u)u^m}{\varepsilon f(u) + u^m}
\label{RegMobFunc}
\end{equation}
and choosing the initial condition $u_{0\varepsilon}=u_0+\varepsilon^q$ with $\varepsilon>0$ and $0<q<\frac{1}{2}$, that the solution of this modified problem remains strictly positive if $m$ is large enough. It can also be shown \cite{BernisFreeman1990} that any subsequence $\lbrace u_{\varepsilon}\rbrace$ converges to a solution of the original problem as $\varepsilon\rightarrow0$. However, this modified mobility function has the disadvantages that $\varepsilon$ has to be chosen dependent on the mesh size for convergence \cite{Barrett1}.  Moreover, tracking the free boundary becomes harder, since for any fixed $\varepsilon>0$, the solution $u_{\varepsilon}$ of the modified problem is strictly positive.

In multi-dimensional space, the construction of positive self-similar solutions for \eqref{ThinFilmEquation}-\eqref{MobFunc} with $f_0(u)\equiv1$  is conducted in \cite{Ferreira1997}. The existence of positive solutions with finite speed of propagation property for \eqref{ThinFilmEquation}-\eqref{MobFunc} with $f_0(u)\equiv1$ and $p\in(\frac{1}{8},2)$ under appropriate boundary conditions, is shown in \cite{Bertsch1998,DalPasso1998} using energy and entropy estimates.   Furthermore, results on the asymptotic behaviour for these solutions as $t\rightarrow\infty$ is established in \cite{DalPasso1998}.  We also mention the works in \cite{Boutat2008,DalPasso2001,Elliott96,Grun1995}, where it is proved the existence of solutions for variants of the equation \eqref{ThinFilmEquation}.

Few previous numerical studies have introduced positivity preserving schemes for \eqref{ThinFilmEquation} \cite{Barrett1,Bertozzi1995,Grun2000,Bertozzi3}. Finite element approximations that guarantee the positivity of the solution of \eqref{ThinFilmEquation} has been initiated by Barrett et al.\ \cite{Barrett1} by enforcing at each time step the positivity of the approximate solution as a constraint. This leads to solving variational inequality problems. Finite difference methods that guarantee the positivity of the solution of \eqref{ThinFilmEquation} by using entropy are proposed in \cite{Grun2000,Bertozzi3}. 

\subsection{Systems with logarithmic free energy functions}
\label{SubsectionLogFreeEnergy}
Equation \eqref{GenThinFilmEquation}, known as Cahn-Hilliard equation, was introduced to model phase separation in binary mixture \cite{CAHN1962,CahnHilliard1958}. In such applications, the quantity $u$ describes the state of a relative concentration $u=(\frac{2X_1}{X_1+X_2}-1)$, where $X_1, X_2\in[0,1]$ are the local concentrations of the two components \cite{GARCKE2003}. Hence, only values of $u$ in the interval $[-1,1]$ are physically meaningful. In this case, the mobility function $f$ should be zero in the pure component, i.e.\ $f(1)=f(-1)=0$. A thermodynamically reasonable choice is \cite{CAHN1994}
\begin{equation}
f(u)=1-u^2.
\label{BoundedMob}
\end{equation}
A phenomenological theory describing the separation of phase is provided by consideration of a free energy $\varphi(u)$, where $\varphi"(u)>0$ for $\theta>\theta_c$ and $\varphi"(u)<0$ for $\theta<\theta_c$. Here, $\theta$ is the absolute temperature and $\theta_c>0$ is a constant.  A main form for the free energy having the required double-well form is given by the logarithmic potential \cite{Barrett3,Barrett2,Elliott96}
\begin{equation}
\varphi(u)=\frac{\theta}{2}\left((1+u)\log\left(\frac{1+u}{2}\right)+(1-u)\log\left(\frac{1-u}{2}\right)\right)+\varphi_0(u),
\label{FreeEnergyLog}
\end{equation}
where $\varphi_0$ is a smooth function on the interval $[-1,1]$. A typical example of $\varphi_0$ is given by \cite{Barrett3,Barrett2}
\begin{equation}
\varphi_0(u)=\dfrac{\theta_c}{2}(1-u^2).
\label{F0function}
\end{equation}
One important property of equation \eqref{GenThinFilmEquation} with the degenerate  mobility function \eqref{BoundedMob} and the logarithmic free energy \eqref{FreeEnergyLog}-\eqref{F0function} is that solutions with initial data satisfying $\vert u_0(\bm{x})\vert\leqslant 1$ remain in the interval $[-1,1]$ for all later times \cite{Elliott96}. This is not true in general for fourth-order parabolic equations.

The second derivative of the energy function \eqref{FreeEnergyLog}-\eqref{F0function} is given by
\begin{equation}
\varphi^{\prime\prime}(u)=\theta_c\left(\dfrac{\theta/\theta_c}{1-u^2}-1\right)
\end{equation}
and is negative on $\big(-\sqrt{1-\theta/\theta_c},\sqrt{1-\theta/\theta_c}\big)$ for $\theta<\theta_c$, called spinodal interval.\\
If $\theta$ is closed to $\theta_c$ then $\varphi^{\prime\prime}$ could be approximated by $\theta_c u^2$ in the neighbourhood of $u=0$, and one could approximate the free energy function $\varphi$ as a quartic polynomial in $u$ retaining the double-well form as 
\begin{equation}
\varphi(u)=c(u^2-\beta^2)^2,\quad c\in \mathbb{R}_+,
\label{FreeEnergySmooth}
\end{equation}
where $\beta$ is the positive root of $\varphi^{\prime}(u)$. \\
If $\theta/\theta_c \rightarrow 0$ then $\varphi^{\prime\prime}$ tends to $-\theta_c$ and the spinodal interval tends to $\big(-1,1\big)$. It is suggested in this case by \cite{Oono1988} to use an obstacle potential free energy of the form 
\begin{equation}
\varphi(u)=
\left\lbrace
\begin{aligned}
&\dfrac{1}{2}(1-u^2), \quad |u|\leqslant 1,\\
&\infty,\quad |u|>1,
\end{aligned}
\right.
\end{equation}
which is formally the limit of the logarithmic free energy \eqref{FreeEnergyLog}-\eqref{F0function} as $\theta\rightarrow 0$, with $\theta_c=1$. It is shown in \cite{Elliott96} that weak solutions to \eqref{GenThinFilmEquation} and \eqref{BoundedMob} with the  free energy function \eqref{FreeEnergyLog}-\eqref{F0function} converge to weak solutions of \eqref{GenThinFilmEquation} and \eqref{BoundedMob} with the free energy function \eqref{F0function} as $\theta\rightarrow0$.

However, one difficulty in solving \eqref{ThinFilmSyst} with the logarithmic free energy \eqref{FreeEnergyLog} is that $\varphi^{\prime}$ presents singularities  at points  $u=\pm 1$ and thus, the second equation in \eqref{ThinFilmSyst} has no meaning at the limit $|u|\rightarrow 1$.  A finite element approximation for \eqref{ThinFilmSyst} with the mobility \eqref{BoundedMob} and the free energy \eqref{FreeEnergyLog} has been proposed in \cite{Barrett2} with the constraint on the solution $\vert u\vert<1$. Here, we propose a formulation of \eqref{ThinFilmSyst} with the mobility \eqref{BoundedMob} and the  energy \eqref{FreeEnergyLog} that has no singularity at the limit $|u|\rightarrow 1$. The equations \eqref{ThinFilmSyst} can be rewritten  as 
\begin{equation}
\begin{aligned}
&\partial_tu-\nabla\cdot\big(f(u)\nabla w\big)-\nabla\cdot\big(g(u)\nabla u\big)=0 \quad \text{in } \Omega_T,\\
&w=-\gamma\Delta u \quad \text{in } \Omega_T.
\end{aligned}
\label{ThinFilmSystMod}
\end{equation} 
With the  mobility function \eqref{BoundedMob} and  the  free energy function \eqref{FreeEnergyLog}, the function $g$ is given by
\begin{equation}
g(u)=f(u)\varphi^{\prime\prime}(u)=\theta+(1-u^2)\varphi_0^{\prime\prime}(u).
\label{GfunctionExpl}
\end{equation}
There is no more singularity at points $u=\pm 1$ in the formulation \eqref{ThinFilmSystMod}. Therefore, the numerical solutions could be computed at  $u=\pm 1$, which correspond to the pure component values.

\section{Numerical methods}
\label{SectionNumMethods}
The time interval $[0,T]$ is discretized as
\begin{equation}
\Delta t=\dfrac{T}{N},\quad t_n=n \Delta t,\quad  n=0,1,2,...,N,
\end{equation} 
where $\Delta t$ is the time step used. For the time discretization of \eqref{ThinFilmSyst}, we use a second-order semi-implicit method, so that $u^n\simeq u(t_n)$ for $n=1,...,N-1$, and 
\begin{eqnarray}
\begin{aligned}
&\dfrac{3u^{n+1}-4u^n+u^{n-1}}{2\Delta t}-\nabla\cdot\Big(\big(2f(u^n)-f(u^{n-1})\big)\nabla w^{n+1}\Big)=0,\\ 
&w^{n+1}+\gamma\Delta u^{n+1}-\big(\varphi^{\prime}(u^n)+\varphi^{\prime\prime}(u^n)(u^{n+1}-u^n)\big)=0.
\end{aligned}
\label{TimeDiscretisation}
\end{eqnarray}
To allow for a linear scheme in the temporal discretization \eqref{TimeDiscretisation} of the equations \eqref{ThinFilmSyst}, the time derivative $\partial_tu$ is approximated using a second-order backward differentiation formula, a second-order extrapolation formula is employed for the mobility function $f(u)$ while the energy $\varphi^{\prime}(u)$ is evaluated using a Taylor approximation with second-order truncation error. We refer to \cite{Keita2021} for more details on the numerical investigation of these techniques for fourth-order nonlinear diffusion equations.

Let $\mathcal{T}_h$ be a regular triangulation of $\Omega$ into disjoint triangles $\kappa$, with $h_{\kappa}:=diam(\kappa)$ and $h:=\max_{\kappa\in \mathcal{T}_h}h_{\kappa}$, so that 
\begin{equation}
\Omega=\bigcup_{\kappa\in\mathcal{T}_h}\kappa.
\end{equation}
Let $\mathbf{J}$ be the set of nodes of $\mathcal{T}_h$, with coordinates $\lbrace \bm{x}_j\rbrace_{j\in\mathbf{J}}$. Associated with  $\mathcal{T}_h$  are the finite element spaces 
\begin{equation}
{\mathcal{V}_h}=\Big\lbrace v_h\in \mathcal{C}^0(\Omega,\mathbb{R}): {v_h}_{|_{\kappa}}\text{ is linear }\forall \kappa\in \mathcal{T}_h \Big\rbrace
\label{Vhspace}
\end{equation}
and
\begin{equation}
{\mathcal{S}_h^a}=\Big\lbrace v_h\in {\mathcal{V}_h}: v_h(\bm{x}_j)\geqslant a,\quad\forall \bm{x}_j\in\mathcal{T}_h \Big\rbrace,
\label{Shspace}
\end{equation}
where $a\in\mathbb{R}$ is given. For any element $v_h\in\mathcal{V}_h$ and point $\bm{x}_j\in\mathcal{T}_h$, we denote $v_h(\bm{x}_j)$ by $v_{h,j}$.

We consider the following finite element approximation for \eqref{TimeDiscretisation}: Given a suitable approximation of the initial solution $u_h^{0}=\Pi_hu_0\in {\mathcal{S}}_h$ and a proper initialization for $u_h^1\in {\mathcal{S}}_h$, find $\big(\widehat{u}_h^{n+1},w_h^{n+1}\big)\in {\mathcal{V}}_h \times {\mathcal{V}}_h$ such that
\begin{eqnarray}
\begin{aligned}
&\int_{\Omega}\left(\dfrac{3\widehat{u}_h^{n+1}-4u_h^{n}+u_h^{n-1}}{2\Delta t}\right)v_h\,d\bm{x}+\int_{\Omega}[2f(u_h^{n})-f(u_h^{n-1})]\nabla w_h^{n+1}\cdot\nabla v_h\,d\bm{x}=0,\forall v_h \in {\mathcal{V}_h},\\
&\int_{\Omega} w_h^{n+1}q_h \,d\bm{x}-\int_{\Omega}\gamma\nabla \widehat{u}_h^{n+1}\cdot\nabla q_h\,d\bm{x}-\int_{\Omega}\big(\varphi^{\prime}(u_h^{n})+\varphi^{\prime\prime}(u_h^{n})(\widehat{u}_h^{n+1}-u_h^{n}) \big)q_h\,d\bm{x}=0,\forall q_h \in {\mathcal{V}_h}.
\end{aligned}
\label{VariationalEqualityScheme}
\end{eqnarray}
Then find $u_h^{n+1}\in{\mathcal{S}}_h$ such that 
\begin{equation*}
(\mathbf{P})\left\lbrace
\begin{aligned}
&u_h^{n+1}=\text{arg} \min_{u_h\in {\mathcal{V}_h}}F(u_h), \text{ where } F(u_h):=\dfrac{1}{2}\Vert u_h-\widehat{u}_h^{n+1} \Vert^2_{L^2(\Omega)},\\
&\text{ subject to }\quad c_1(u_h):=-(u_h-a)\leqslant 0,\quad\qquad\qquad(i) \\
&\qquad\qquad\qquad c_2(u_h):=\int_{\Omega}u_h\,d\bm{x}-\int_{\Omega}u_h^{0}\,d\bm{x}=0.\quad (ii) \qquad\qquad\qquad\qquad\qquad\qquad\qquad\qquad
\end{aligned}
\right.
\label{P}
\end{equation*}
$\Pi_h$ is an interpolation or projection operator on $\mathcal{V}_h$. The numerical algorithm \eqref{VariationalEqualityScheme} is a two steps scheme and therefore, it requires the use of a starting procedure to obtain $u_h^1$. Here, we use a semi-implicit backward Euler method
\begin{equation}
\begin{aligned}
&\int_{\Omega}\left(\dfrac{\widehat{u}_h^{1}-u_h^{0}}{\Delta t}\right)v_h\,d\bm{x}+\int_{\Omega}f(u_h^{0})\nabla w_h^{1}\cdot\nabla v_h\,d\bm{x}=0,\quad\forall v_h \in {\mathcal{V}_h},\\
&\int_{\Omega} w_h^{1}q_h \,d\bm{x}-\int_{\Omega}\gamma\nabla \widehat{u}_h^{1}\cdot\nabla q_h\,d\bm{x}-\int_{\Omega}\varphi'(u_h^{0}) q_h\,d\bm{x}=0, \quad\forall q_h \in {\mathcal{V}_h},\\
\end{aligned}
\label{VariationalEqualitySchemeInit}
\end{equation}
to compute $\widehat{u}_h^1$ from the initial solution $u_h^{0}$ and then find $u_h^1$ from $(\mathbf{P})$.

\subsection{Analysis of problem $(\mathbf{P})$}
The feasible set 
\begin{equation}
\mathcal{K}_h=\lbrace u_h\in\mathcal{V}_h: c_1(u_h)\leqslant 0\text{ and }c_2(u_h)=0 \rbrace\subset\mathcal{S}_h^a
\end{equation}
is a non-empty (since it contains the initial condition $u_h^{0}$) closed and convex subset of the Hilbert space $\mathcal{V}_h$ equipped with the $L^2$-inner product, noted
$\langle u_h,v_h\rangle$. The functional $F$  is Fr\'echet differentiable with derivative $\nabla F(u_h)=u_h-\widehat{u}_h^{n+1}$, which belongs to the finite element space $\mathcal{V}_h$ and is a continuous function of $u_h$. We also have
\begin{equation}
\begin{aligned}
\langle\nabla F(u_h)-\nabla F(v_h),u_h-v_h\rangle=\Vert u_h-v_h\Vert_{L^2(\Omega)}^2.
\end{aligned}
\end{equation}
Hence, $F$ is elliptic. By Theorem 8.4-1 in \cite{Ciarlet1982}, problem $(\mathbf{P})$ has an unique solution $u_h^{n+1}\in\mathcal{K}_h\subset\mathcal{S}_h^a$ characterized by
\begin{equation}
\langle \nabla F(u_h^{n+1}),v_h-u_h^{n+1}\rangle=\int_{\Omega}(u_h^{n+1}-\widehat{u}_h^{n+1})(v_h-u_h^{n+1})\,d\bm{x}\geqslant 0,\quad \forall v_h\in\mathcal{K}_h.
\end{equation}

Let the Lagrangian function of $(\mathbf{P})$ be
\begin{equation}
\mathcal{L}(u_h,\lambda_h,\mu)=F(u_h)+\langle\lambda_h,c_1(u_h)\rangle+\mu c_2(u_h),
\end{equation}
where $\lambda_h\in\mathcal{V}_h$ and $\mu\in\mathbb{R}$ are Lagrange multipliers.  $\mathcal{L}$ is Fr\'echet differentiable with respect to $u_h$ at the point $(u_h,\lambda_h,\mu)$ and its derivative is given by 
\begin{equation}
\dfrac{\partial \mathcal{L}}{\partial u_h}(u_h,\lambda_h,\mu)=u_h-\widehat{u}_h^{n+1}+\mu-\lambda_h,
\end{equation}
which also belongs to $\mathcal{V}_h$  and is a continuous function of $u_h$. The Karush-Kuhn-Tucker (KKT) necessary conditions for the problem $(\mathbf{P})$ state as: 
\begin{equation}
\begin{aligned}
&\dfrac{\partial \mathcal{L}}{\partial u_h}(u_h^{n+1},\lambda_h^*,\mu^*)=u_h^{n+1}-\widehat{u}_h^{n+1}+\mu^*-\lambda_h^*=0,\\
&\langle\lambda_h^*, c_1(u_h^{n+1})\rangle=-\int_{\Omega}\lambda_h^*(u_h^{n+1}-a)\,d\bm{x}=0,\\
& c_1(u_h^{n+1})=-(u_h^{n+1}-a)\leqslant 0,\quad c_2(u_h^{n+1})=\int_{\Omega}u_h^{n+1}\,d\bm{x}-\int_{\Omega}u_h^{0}\,d\bm{x}=0,\\
&\lambda_h^*\in\mathcal{S}_h^0,\quad \mu^*\in\mathbb{R}.
\end{aligned}
\label{KKTCond}
\end{equation}
The functional $F$ and the constraint functions $c_i$, $i=1,2$, are convex and differentiable at all points $u_h\in\mathcal{V}_h$. The existence of $u_h^{n+1}$, solution of $(\mathbf{P})$, implies the existence of $(u_h^{n+1}, \lambda_h^*,\mu^*)$, solution of the KKT conditions \eqref{KKTCond}, which conditions become sufficient since the functional $F$ is convex (see Theorems 9.2-3 and 9.2-4 in \cite{Ciarlet1982}).

An Uzawa algorithm for solving $(\mathbf{P})$ is given by:\\
Take $u_{h}^{n+1,0}\in \mathcal{S}_h^a$, $\lambda_{h}^0\in\mathcal{S}_h^0$ and $\mu^0\in\mathbb{R}$.\\
For $k=1,2,3,...,$ compute until convergence:
\begin{align}
&\lambda_{h}^k=\left[\lambda_{h}^{k-1}+\beta_kc_1(u_h^{n+1,k-1})\right]_+=\left[\lambda_{h}^{k-1}-\beta_k(u_h^{n+1,k-1}-a)\right]_+,\label{lambdak}\\
&\mu^k=\mu^{k-1}+\rho_kc_2(u_h^{n+1,k-1})=\mu^{k-1}+\rho_k\left(\int_{\Omega}u_h^{n+1,k-1}\,d\bm{x}-\int_{\Omega}u_h^{0}\,d\bm{x}\right),\label{muk}\\
&u_h^{n+1,k}=u_h^{n+1,k-1}-\alpha_k\dfrac{\partial \mathcal{L}}{\partial u_h}(u_h^{n+1,k-1},\lambda_h^k,\mu^k)=u_h^{n+1,k-1}-\alpha_k(u_h^{n+1,k-1}-\widehat{u}_h^{n+1}+\mu^k-\lambda_{h}^k),
\label{uk}
\end{align}
where $\beta_k$, $\rho_k$ and $\alpha_k$  are positive step sizes to be chosen, and $\left[\cdot\right]_+$ designs the positive part of the expression in the brackets, taken at each node $\bm{x}_j$ of the mesh. The stopping criteria is chosen as: 
\begin{equation}
\quad\Vert \lambda_h^k-\lambda_h^{k-1}\Vert^2_{L^2(\Omega)}\leqslant\epsilon,\quad\vert \mu^k-\mu^{k-1}\vert\leqslant\epsilon \quad\text{and}\quad\Vert u_h^{n+1,k-1}-u_h^{n+1,k-2} \Vert^2_{L^2(\Omega)}\leqslant\epsilon,
\label{toleranceUzawa}
\end{equation}
where $\epsilon>0$ is chosen small enough. The Uzawa method is a classical method for solving constrained optimization problems. It consists in replacing a constrained minimization problem by a sequence of unconstrained minimization problems \cite{Allaire2007,Ciarlet1989,Uzawa1958}.\\
If the algorithm \eqref{lambdak}-\eqref{toleranceUzawa} converges to $(u_h^\star,\lambda_h^\star,\mu^\star)$ as $\epsilon\rightarrow 0$ then $(u_h^\star,\lambda_h^\star,\mu^\star)$ satisfies the KKT conditions \eqref{KKTCond} and thus, $u_h^\star$ is solution of $(\mathbf{P})$. The uniqueness of the solution implies that $(u_h^\star,\lambda_h^\star,\mu^\star)=(u_h^{n+1},\lambda_h^{*},\mu^{*})$. This justifies the use of the above Uzawa algorithm to solve the optimization problem $(\mathbf{P})$. In the following, we assume that
\begin{align}
&\beta_k=\beta>0\quad\text{and}\quad\alpha_k=\alpha>0, \quad\forall k\geqslant 1.\label{Condalphabeta}\\
&\exists N_0=\ N_0(\epsilon)\in\mathbb{N}\quad s.t. \quad \forall k>N_0,\text{ }\eqref{toleranceUzawa}\text{ is satisfied.}\label{CondConv}
\end{align}
The existence of $ N_0$ follows from the convergence of the algorithm \eqref{lambdak}-\eqref{toleranceUzawa}.

Define $\overline{u}_h^{n+1,k}\in\mathcal{V}_h$ by
\begin{equation}
\overline{u}_h^{n+1,k}=\mathcal{P}_{\mathcal{S}_h^a}[u_h^{n+1,k-1}-\alpha(u_h^{n+1,k-1}-\widehat{u}_h^{n+1}+\mu^k)],\quad\forall k\geqslant1,
\label{ukProj}
\end{equation}
where the projection operator $\mathcal{P}_{\mathcal{S}_h^a}$ is defined at each node $\bm{x}_j$ as
\begin{equation}
\mathcal{P}_{\mathcal{S}_h^a}[u_{h,j}]=
\left\lbrace
\begin{aligned}
&u_{h,j},\text{ if } u_{h,j}\geqslant a,\\
&a,\text{ if } u_{h,j}<a,
\end{aligned}
\right.
\quad\forall j\in\mathbf{J}. 
\label{ProjOper}
\end{equation}
The aim in what follows is to show that if the solution obtained from the Uzawa algorithm \eqref{lambdak}-\eqref{toleranceUzawa} converges as $\epsilon\rightarrow 0$, then this limit solution is nothing but the limit of $\overline{u}_h^{n+1,k}$ as $k\rightarrow\infty$. This suggests that the solution from the Uzawa algorithm \eqref{lambdak}-\eqref{toleranceUzawa} can be replaced by the solution from the projection operator \eqref{ProjOper}.\\
In one hand, $\overline{u}_h^{n+1,k}$ can be written by using the definition of the positive part of a function as
\begin{equation}
\overline{u}_h^{n+1,k}=u_h^{n+1,k-1}-\alpha(u_h^{n+1,k-1}-\widehat{u}_h^{n+1}+\mu^k)+ \left[-\big(u_h^{n+1,k-1}-\alpha(u_h^{n+1,k-1}-\widehat{u}_h^{n+1}+\mu^k)-a\big)\right]_+.
\label{ukProjMod}
\end{equation}
In the other hand, equation \eqref{uk} implies that
\begin{equation}
u_h^{n+1,k}=u_h^{n+1,k-1}-\alpha(u_h^{n+1,k-1}-\widehat{u}_h^{n+1}+\mu^k)+\alpha\lambda_{h}^k,
\end{equation}
We next express the Lagrangian multiplier $\lambda_{h}^k$ in terms of $-\big(u_h^{n+1,k-1}-\alpha(u_h^{n+1,k-1}-\widehat{u}_h^{n+1}+\mu^k)-a\big)$. In fact, we have from \eqref{uk} that
\begin{equation}
u_h^{n+1,k-1}=u_h^{n+1,k-2}-\alpha(u_h^{n+1,k-2}-\widehat{u}_h^{n+1}+\mu^{k-1}-\lambda_{h}^{k-1}),\quad\forall k\geqslant 2.
\label{uk-1}
\end{equation}
Substituting \eqref{uk-1} into \eqref{lambdak}, we get
\begin{eqnarray}
\begin{aligned}
\lambda_{h}^{k}=\left[(1-\alpha\beta)\lambda_{h}^{k}+\theta_{h}^k\right]_+,\quad \forall k\geqslant2,
\end{aligned}
\label{lambdaMod2}
\end{eqnarray} 
where  
\begin{equation}
\tau_{h}^k=\lambda_{h}^k-\lambda_{h}^{k-1},\quad \xi^k=\mu^k-\mu^{k-1} \quad\text{and}\quad \eta_{h}^k=u_h^{n+1,k-1}-u_h^{n+1,k-2},\quad \forall k\geqslant1,
\label{epsilons}
\end{equation}
and
\begin{equation}
\theta_{h}^k=-\beta\big(u_h^{n+1,k-1}-\alpha(u_h^{n+1,k-1}-\widehat{u}_h^{n+1}+\mu^k)-a\big)+(\alpha\beta-1)\tau_{h}^k+\beta(1-\alpha)\eta_{h}^k-\alpha\beta\xi^k.
\label{theta}
\end{equation}
We can now show that
\begin{lemma}
For all $k\geqslant2$, we have
\begin{equation}
\lambda_{h}^k=\left[-\frac{1}{\alpha}\big(u_h^{n+1,k-1}-\alpha(u_h^{n+1,k-1}-\widehat{u}_h^{n+1}+\mu^k)-a\big)+(1-\frac{1}{\alpha\beta})\tau_{h}^k+(\frac{1}{\alpha}-1)\eta_{h}^k-\xi^k\right]_+.
\end{equation}
\label{Lemmalambdak}
\end{lemma}

\begin{proof}
We  have the equivalence
\begin{equation}
(1-\alpha\beta)\lambda_{h}^{k}+\theta_{h}^k \geqslant 0 \Longleftrightarrow \lambda_{h}^{k}=\dfrac{\theta_{h}^k}{\alpha\beta}\geqslant0.
\label{Equiv}
\end{equation}
In fact, if $(1-\alpha\beta)\lambda_{h}^{k}+\theta_{h}^k \geqslant 0$ then 
$\lambda_{h}^{k}=(1-\alpha\beta)\lambda_{h}^{k}+\theta_{h}^k\Longrightarrow\lambda_{h}^{k}=\dfrac{\theta_{h}^k}{\alpha\beta}\geqslant0.$\\
Conversely, if $\lambda_{h}^{k}=\dfrac{\theta_{h}^k}{\alpha\beta}\geqslant0$ then 
$(1-\alpha\beta)\lambda_{h}^{k}+\theta_{h}^k=(1-\alpha\beta)\dfrac{\theta_{h}^k}{\alpha\beta}+\theta_{h}^k=\dfrac{\theta_{h}^k}{\alpha\beta}\geqslant 0.$\\
From \eqref{lambdaMod2} and \eqref{Equiv}, we get
\begin{equation}
\lambda_{h}^{k}=\left[\dfrac{\theta_{h}^k}{\alpha\beta}\right]_+,\quad \forall k\geqslant2.
\label{lambdaMod4}
\end{equation}
This completes the proof.
\end{proof}
\noindent
The following inequality also holds:
\begin{lemma}
For all $p_h, r_h, s_h\in\mathcal{V}_h$ such that $p_{h}=\left[-r_{h}+s_{h}\right]_+-\left[-r_{h}\right]_+$ then
$\vert p_{h}\vert\leqslant \vert s_{h}\vert$.
\label{phbound}
\end{lemma}

\begin{proof}
\begin{equation}
p_{h}=\dfrac{1}{2}\Big(\vert s_{h}-r_{h}\vert+s_{h}-r_{h}-\big(\vert-r_{h}\vert-r_{h}\big)\Big)=\dfrac{1}{2}\Big(\vert s_{h}-r_{h}\vert-\vert r_{h}\vert+s_{h}\Big).
\end{equation}
From the reverse triangle inequality, we get $-\vert s_{h}\vert\leqslant\vert s_{h}-r_{h}\vert-\vert r_{h}\vert\leqslant\vert s_{h}\vert$
which implies that 
$-2\vert s_{h}\vert\leqslant \vert s_{h}-r_{h}\vert-\vert r_{h}\vert+s_{h}=2p_{h}\leqslant2\vert s_{h}\vert$. This completes the proof.
\end{proof}
\noindent
We now state and prove the main result:
\begin{theorem}
\label{TheoDiffSol}
There exist two positive constants $C_1$ and $C_2$ depending only on $\alpha$, $\beta$, $\rho_k$ and the domain $\Omega$ such that for all $k\geqslant N_0$, we have
\begin{align}
&\Vert u_h^{n+1,k}-\overline{u}_h^{n+1,k}\Vert_{L^2(\Omega)}\leqslant C_1\epsilon,\label{NormDiff}\\
&\Big\vert \int_{\Omega}\overline{u}_h^{n+1,k}\,d\bm{x}-\int_{\Omega}u_h^{0}\,d\bm{x}\Big\vert\leqslant C_2\epsilon.\label{NormMassDiff}
\end{align}
\end{theorem}

\begin{proof}
Using Lemma \ref{Lemmalambdak}, equation \eqref{uk} can be written as
\begin{equation}
\begin{aligned}
u_h^{n+1,k}=&u_h^{n+1,k-1}-\alpha(u_h^{n+1,k-1}-\widehat{u}_h^{n+1}+\mu^k)\\
&+\left[-\big(u_h^{n+1,k-1}-\alpha(u_h^{n+1,k-1}-\widehat{u}_h^{n+1}+\mu^k)-a\big)+(\alpha-\frac{1}{\beta})\tau_{h}^k+(1-\alpha)\eta_{h}^k-\alpha\xi^k\right]_+.
\end{aligned}
\label{ukGeneral}
\end{equation}
Subtracting \eqref{ukProjMod} from \eqref{ukGeneral}, we obtain
\begin{equation}
\begin{aligned}
u_h^{n+1,k}-\overline{u}_h^{n+1,k}=&[-\big(u_h^{n+1,k-1}-\alpha(u_h^{n+1,k-1}-\widehat{u}_h^{n+1}+\mu^k)-a\big)+(\alpha-\frac{1}{\beta})\tau_{h}^k+(1-\alpha)\eta_{h}^k-\alpha\xi^k]_+\\
&-\left[-\big(u_h^{n+1,k-1}-\alpha(u_h^{n+1,k-1}-\widehat{u}_h^{n+1}+\mu^k)-a\big)\right]_+.
\end{aligned}
\end{equation}
Using Lemma~\ref{phbound}, we get
\begin{equation}
\begin{aligned}
\vert u_h^{n+1,k}-\overline{u}_h^{n+1,k}\vert\leqslant\vert(\alpha-\frac{1}{\beta})\tau_{h}^k+(1-\alpha)\eta_{h}^k-\alpha\xi^k\vert\leqslant\vert\alpha-\frac{1}{\beta}\vert\vert\tau_{h}^k\vert+\vert1-\alpha\vert\vert\eta_{h}^k\vert+\alpha\vert\xi^k\vert.
\end{aligned}
\label{NormDiff1}
\end{equation}
Taking the square in \eqref{NormDiff1},  integrating over $\Omega$ and using  Holder inequality yield 
\begin{equation}
\begin{aligned}
\Vert u_h^{n+1,k}-\overline{u}_h^{n+1,k}\Vert_{L^2(\Omega)}^2\leqslant& \vert\alpha-\frac{1}{\beta}\vert^2\Vert\tau_{h}^k\Vert_{L^2(\Omega)}^2+\vert1-\alpha\vert^2\Vert\eta_{h}^k\Vert_{L^2(\Omega)}^2\\
&+2\vert\alpha-\frac{1}{\beta}\vert\vert1-\alpha\vert\Vert\tau_{h}^k\Vert_{L^2(\Omega)}\Vert\eta_{h}^k\Vert_{L^2(\Omega)}+\alpha^2\vert\xi^k\vert^2\vert\Omega\vert\\
&+2\alpha\vert\xi^k\vert\Big(\vert\alpha-\frac{1}{\beta}\vert\Vert\tau_{h}^k\Vert_{L^2(\Omega)}\vert\Omega\vert^{1/2}+\vert1-\alpha\vert\Vert\eta_{h}^k\Vert_{L^2(\Omega)}\vert\Omega\vert^{1/2}\Big),
\end{aligned}
\label{NormDiff4}
\end{equation}
where $\vert\Omega\vert$ is the Lebesgue measure of the bounded domain $\Omega$. Combining \eqref{toleranceUzawa}, \eqref{CondConv} and \eqref{epsilons} gives rise to
\begin{equation}
\begin{aligned}
\Vert u_h^{n+1,k}-\overline{u}_h^{n+1,k}\Vert_{L^2(\Omega)}^2\leqslant \Big(\big(\vert\alpha-\frac{1}{\beta}\vert+\vert1-\alpha\vert\big)^2+\alpha^2\vert\Omega\vert+2\alpha\vert\Omega\vert^{1/2}\big(\vert\alpha-\frac{1}{\beta}\vert+\vert1-\alpha\vert\big)\Big)\epsilon^2,
\end{aligned}
\label{NormDiff5}
\end{equation}
for all $k>N_0$. Hence, \eqref{NormDiff} holds with 
\begin{equation}
C_1=\vert\alpha-\frac{1}{\beta}\vert+\vert1-\alpha\vert+\alpha\vert\Omega\vert^{1/2}.
\label{C1}
\end{equation}
We also have
\begin{equation}
\begin{aligned}
\Big\vert \int_{\Omega}\overline{u}_h^{n+1,k}\,d\bm{x}-\int_{\Omega}u_h^{0}\,d\bm{x}\Big\vert&=\Big\vert \int_{\Omega}(\overline{u}_h^{n+1,k}-u_h^{n+1,k}+u_h^{n+1,k})\,d\bm{x}-\int_{\Omega}u_h^{0}\,d\bm{x}\Big\vert\\
&\leqslant\left(C_1\vert\Omega\vert^{1/2}+\dfrac{1}{\rho_k}\right)\epsilon.
\end{aligned}
\end{equation}
Let $\rho_k$ be such that $0<\underline{\rho}<\rho_k<\overline{\rho}$. Then $\frac{1}{\rho_k}<\frac{1}{\underline{\rho}}$. Hence, \eqref{NormMassDiff} holds with $C_2=C_1\vert\Omega\vert^{1/2}+\frac{1}{\underline{\rho}}$.  This completes the proof.
\end{proof}
\begin{remark}
\textbf{1)} The Cauchy-Schwarz inequality could be applied to the inequality \eqref{NormDiff1}, but this leads to a constant $C_1=\sqrt{(\vert\alpha-\frac{1}{\beta}\vert^2+\vert1-\alpha\vert^2+\alpha^2)(2+\vert\Omega\vert})$ which is greater or equal to the one we found in \eqref{C1}.\\
\textbf{2)} The above analysis also holds for solutions $u_h\geqslant a$.\\
\textbf{3)} We have taken $a$ as constant for sake of simplicity. There is no difficulty in taking a function $a=a(\bm{x})$.
\end{remark}

\subsection{An Uzawa algorithm for \eqref{ThinFilmSyst}}
The Uzawa algorithm \eqref{lambdak}-\eqref{uk} is used to approximate the solution $u_h^{n+1}$ of the optimization problem $\textbf{(P)}$. Hence, an approximate solution $(u_h^{n+1},w_h^{n+1})$ at time $t_{n+1}$ to \eqref{ThinFilmSyst} could be obtained by using  the scheme, referred to as \textit{Scheme1}  and summarized in the following box: 
\begin{MyFrame}
\quad \textbf{\textit{Scheme1}: An Uzawa method for \eqref{ThinFilmSyst}}

\qquad Compute $\big(\widehat{u}_h^{n+1},w_h^{n+1}\big)$ with the FEM  \eqref{VariationalEqualityScheme}.

\qquad Then, for a given $\epsilon$, iterate \eqref{lambdak}-\eqref{uk} until \eqref{toleranceUzawa} is satisfied to get  $u_h^{n+1}$.
\end{MyFrame}

\subsection{Classical truncation method for \eqref{ThinFilmSyst}}
\label{SubSectionClassTrunctionMethod}
The truncation method was introduced in \cite{Berger1975} for the approximation of the solution of a parabolic problem with free boundary. Later, the method was used for solving parabolic variational inequalities \cite{Berger1977,Berger1976}. The  truncation method for \eqref{ThinFilmSyst} consists in using an appropriate method (here, we use the finite element approximation  \eqref{VariationalEqualityScheme}) to obtain the solution $\widehat{u}_h^{n+1}$ at time $t_{n+1}$ of \eqref{TimeDiscretisation}.  Then at each node point $\bm{x}_j$ of the mesh, the value of the approximate solution $u_h^{n+1}$ at time $t_{n+1}$  of \eqref{ThinFilmSyst} is obtained by a direct truncation.
This algorithm is referred to as \textit{Scheme2} and is summarized in the following box:
\begin{MyFrame}
\quad \textbf{\textit{Scheme2}: A classical truncation method for \eqref{ThinFilmSyst}}

\qquad Compute $\big(\widehat{u}_h^{n+1},w_h^{n+1}\big)$ with the algorithm  \eqref{VariationalEqualityScheme}.

\qquad Then, compute $u_h^{n+1}=\mathcal{P}_{\mathcal{S}_h^a}[\widehat{u}_h^{n+1}]$ at each nodal point $\bm{x}_j$ of the mesh.
\end{MyFrame}

If one removes the mass conservation constraint $(ii)$ in $(\mathbf{P})$, then a similar reasoning as in the proof of Theorem~\ref{TheoDiffSol} shows that
\begin{equation}
\Vert u_h^{n+1,k}-\overline{u}_h^{n+1,k}\Vert_{L^2(\Omega)}\leqslant\big(\vert\alpha-\frac{1}{\beta}\vert+\vert1-\alpha\vert\big)\epsilon, \quad \forall k\geqslant N_0.
\label{C1Simple}
\end{equation}
Taking the step sizes $\alpha=\beta=1$ in \eqref{C1Simple}, we get
\begin{equation}
u_h^{n+1,k}=\overline{u}_h^{n+1,k}=\mathcal{P}_{\mathcal{S}_h^a}[\widehat{u}_h^{n+1}],\quad\forall k>N_0.
\label{TruncationNonConss}
\end{equation}
Equality \eqref{TruncationNonConss} means that the solution of \textit{Scheme1} converges to the solution of \textit{Scheme2}  in a finite number of iterations. One drawback of the classical truncation method is that total mass is not preserved.

\subsection{A conservative truncation method for \eqref{ThinFilmSyst}}
\label{SubSectionConsTruncationMethod}
At the limit $\epsilon\rightarrow0$, Theorem~\ref{TheoDiffSol} implies that
\begin{align}
&u_h^{n+1}=\overline{u}_h^{n+1}=\mathcal{P}_{\mathcal{S}_h^a}[u_h^{n+1}-\alpha(u_h^{n+1}-\widehat{u}_h^{n+1}+\mu^*)]\quad a.e.,\label{TruncationCons1} \\
&\int_{\Omega}\mathcal{P}_{\mathcal{S}_h^a}[u_h^{n+1}-\alpha(u_h^{n+1}-\widehat{u}_h^{n+1}+\mu^*)]\,d\bm{x}=\int_{\Omega}u_h^{0}\,d\bm{x}. \label{TruncationCons2}
\end{align}
Equality \eqref{TruncationCons1} means that if convergence of the Uzawa algorithm \eqref{lambdak}-\eqref{toleranceUzawa}  occurs for $\alpha=1$, then the approximate solution $u_h^{n+1}$ at time $t_{n+1}$ of \eqref{ThinFilmSyst} computed with \textit{Scheme1} is given by
\begin{equation}
u_h^{n+1}=\mathcal{P}_{\mathcal{S}_h^a}[\widehat{u}_h^{n+1}-\mu^*],
\label{TruncationScheme}
\end{equation}
which is nothing else but a truncation of a translation of the solution $\widehat{u}_h^{n+1}$ computed with \eqref{VariationalEqualityScheme}.\\
Equality \eqref{TruncationCons2} means that the discrete total mass is preserved. The conservation of the discrete total mass is guaranteed by the translation of the solution $\widehat{u}_h^{n+1}$  by $\mu^*$. Hence, a truncation method for \eqref{ThinFilmSyst} that preserves the discrete total mass could be obtained by first finding $\mu^*$ such that 
\begin{equation}
\int_{\Omega}\mathcal{P}_{\mathcal{S}_h^a}[\widehat{u}_h^{n+1}-\mu^*]\,d\bm{x}=\int_{\Omega}u_h^{0}\,d\bm{x}
\label{muCons}
\end{equation}
and then truncating the translation of the solution $\widehat{u}_h^{n+1}$  by $\mu^*$. 

Define
\begin{equation}
\underline{\mu}:=\min_{j\in\mathbf{J}}\widehat{u}_h^{n+1}(\bm{x}_j)-a\quad\text{and}\quad\overline{\mu}:=\max_{j\in\mathbf{J}}\widehat{u}_h^{n+1}(\bm{x}_j)-a.
\end{equation}
We assume that $\underline{\mu}<0$. Otherwise, $\widehat{u}_h^{n+1}\in\mathcal{S}_h^a$ and thus, projection is not needed.  For all $j\in\mathbf{J}$, we have
\begin{align}
\widehat{u}_h^{n+1}(\bm{x}_j)-\mu\geqslant a,\quad\forall \mu\leqslant\underline{\mu},\\
\widehat{u}_h^{n+1}(\bm{x}_j)-\mu\leqslant a,\quad\forall \mu\geqslant\overline{\mu}.
\end{align}
Consider the nonlinear equation
\begin{equation}
\Theta(\mu):=\int_{\Omega}\mathcal{P}_{\mathcal{S}_h^a}[\widehat{u}_h^{n+1}-\mu]\,d\bm{x}-\int_{\Omega}u_h^{0}\,d\bm{x}=0.
\label{OptimalmuEqu}
\end{equation}
Let $\mu_1, \mu_2\in\mathbb{R}$. We have  
\begin{equation}
\begin{aligned}
\vert\Theta(\mu_1)-\Theta(\mu_2)\vert&\leqslant\int_{\Omega}\big\vert \mathcal{P}_{\mathcal{S}_h^a}[\widehat{u}_h^{n+1}-\mu_1]-\mathcal{P}_{\mathcal{S}_h^a}[\widehat{u}_h^{n+1}-\mu_2]\big\vert\,d\bm{x},\\
&\leqslant\int_{\Omega}\left\vert \frac{\vert \widehat{u}_h^{n+1}-\mu_1-a\vert+\widehat{u}_h^{n+1}-\mu_1+a}{2}-\frac{\vert \widehat{u}_h^{n+1}-\mu_2-a\vert+\widehat{u}_h^{n+1}-\mu_2+a}{2}\right\vert\,d\bm{x},\\
&\leqslant\vert\Omega\vert\vert \mu_1-\mu_2\vert.
\end{aligned}
\end{equation}
Hence, the function $\Theta$ is continuous. From the fact that \eqref{VariationalEqualityScheme} preserves the total mass \cite{Keita2021}, we also have
\begin{align}
&\Theta(\underline{\mu})=\int_{\Omega}(\widehat{u}_h^{n+1}-\underline{\mu})\,d\bm{x}-\int_{\Omega}u_h^{0}\,d\bm{x}=-\underline{\mu}\vert\Omega\vert>0,\\
&\Theta(\overline{\mu})=\int_{\Omega}a\,d\bm{x}-\int_{\Omega}u_h^{0}\,d\bm{x}=\int_{\Omega}(a-u_h^0)\,d\bm{x}\leqslant0.
\end{align}
By the Intermediate Value Theorem, there is $\mu^*\in[\underline{\mu},\overline{\mu}]$ such that $\Theta(\mu^*)=0$. 

For solving this equation, we use the secant method \cite{McNamee2013,Phillips1996}. The secant method is faster than the bisection method and does not require the derivative of $\Theta$. The algorithm of the secant method for approximating the unique solution of \eqref{OptimalmuEqu} in the interval $[\underline{\mu},\overline{u}]$ works as follows:

Take $\mu^0,\mu^1\in[\underline{\mu},\overline{u}]$.

For $k=1,2,3,...,$ compute until convergence:
\begin{equation}
\mu^{k+1}=\mu^k-\dfrac{\Theta(\mu^k)(\mu^k-\mu^{k-1})}{\Theta(\mu^k)-\Theta(\mu^{k-1})}.
\label{SecantAlgorithm}
\end{equation} 
The stopping criteria is chosen as
\begin{equation}
\vert\mu^{k+1}-\mu^k\vert\leqslant\epsilon.
\label{toleranceSecant}
\end{equation}

The resulting numerical scheme is referred to as \textit{Scheme3} and is summarized in the following box:
\begin{MyFrame}
\quad \textbf{\textit{Scheme3}: A conservative truncation method for \eqref{ThinFilmSyst}}

\qquad Compute $\big(\widehat{u}_h^{n+1},w_h^{n+1}\big)$ with the algorithm  \eqref{VariationalEqualityScheme}.

\qquad Then, for a given $\epsilon$, iterate \eqref{SecantAlgorithm} until \eqref{toleranceSecant} is satisfied to get $\mu^*$.

\qquad Finally, compute $u_h^{n+1}=\mathcal{P}_{\mathcal{S}_h^a}[\widehat{u}_h^{n+1}-\mu^*]$ at each node $\bm{x}_j$ of the mesh.
\end{MyFrame}

\subsection{A second-order time accurate version of Barrett et al.\ scheme for \eqref{ThinFilmEquation}}
Recall that  Barrett et al.\ \cite{Barrett1} proposed a first-order time stepping finite element method for \eqref{ThinFilmEquation} that preserves the positivity of the computed solution by solving a variational inequality. For purpose of comparison, we consider the  version of their scheme that corresponds to a second-order approximation in time as in the scheme \eqref{VariationalEqualityScheme}. The scheme is written as follows: Find $(u_h^{n+1,k},w_h^{n+1,k},r_h^{n+1,k})\in {\mathcal{V}_h}\times {\mathcal{V}_h}\times {\mathcal{V}_h}$ such that
\begin{equation}
\begin{aligned}
&\int_{\Omega}\left(\dfrac{3u_h^{n+1,k}-4u_h^{n}+u_h^{n-1}}{2\Delta t}\right)v_h\,d\bm{x}+f_{max}^n\int_{\Omega}\nabla w_h^{n+1,k}\cdot\nabla v_h\,d\bm{x}\\
&\qquad -\int_{\Omega}[f_{max}^n-(2f(u_h^n)-f(u_h^{n-1}))]\nabla w_h^{n+1,k-1}\cdot\nabla v_h\,d\bm{x}=0,\quad\forall v_h \in {\mathcal{V}_h},\\
&\int_{\Omega} w_h^{n+1,k}q_h\,d\bm{x}+\int_{\Omega} r_h^{n+1,k-1}q_h \,d\bm{x}-\int_{\Omega}\gamma\nabla u_h^{n+1,k}\cdot\nabla q_h\,d\bm{x}=0, \quad\forall q_h \in {\mathcal{V}_h},\\
&r_h^{n+1,k}=\left[r_h^{n+1,k-1}-\varrho u_h^{n+1,k}\right]_+,
\end{aligned}
\label{BarrettScheme}
\end{equation}
where $f_{max}^n:=\Vert 2f(u_h^{n})-f(u_h^{n-1})\Vert_{\infty}$ and $\varrho>0$ is a constant. The stopping criteria for the iterative procedure is chosen as:
\begin{equation}
\Vert u_h^{n+1,k}-u_h^{n+1,k-1}\Vert_{L^2(\Omega)}\leqslant \epsilon.
\label{toleranceBarrett}
\end{equation}
We use their first-order scheme to compute the other starting value $u_h^1$ from the initial condition $u_h^0$. The scheme \eqref{BarrettScheme}-\eqref{toleranceBarrett} is referred to as \textit{Scheme4} and is summarized in the following box:
\begin{MyFrame}
\quad \textbf{\textit{Scheme4}: A second-order semi implicit version of Barrett et al.\ scheme for \eqref{ThinFilmEquation}}

\qquad For a given $\epsilon$,  iterate \eqref{BarrettScheme} until \eqref{toleranceBarrett} is satisfied to get the approximate solution.
\end{MyFrame}

\section{Numerical results}
\label{SectNumResults}
A series of numerical problems in two dimensions is given in this section to test the performance of the different schemes. All the algorithms for the simulations are implemented using FreeFem++ software \cite{FreeFem,FreeFemm}. We use the direct solver  Unsymmetric MultiFrontal method (UMFPACK) available in FreeFem++. All simulations are done on a Dell XPS 13 with Intel Core i7-6500U CPU 2.50 GHz processors  and 16GB of RAM running under linux.

\subsection{Self-similar solution}
\label{SubSectSelfSimSol}
For $\gamma=1$ and $f(u)=u$, equation \eqref{ThinFilmEquation} has a positive closed-form compactly supported, $2$-dimensional self-similar solution \cite{Ferreira1997} given by
\begin{equation}
u(r,t)=
\left\lbrace
\begin{aligned}
&\dfrac{t^{-1/3}}{192}(L^2-r^2)^2,\quad r<L,\\
&0,\quad r\geqslant L,
\end{aligned}
\right.
\label{ExactSol1}
\end{equation}
where $r=r(\bm{x},t):=\vert \bm{x}\vert /t^{1/6}$. Following the mixed formulation \eqref{ThinFilmSyst}, we calculate the auxiliary variable
\begin{equation}
w(r,t)=
\left\lbrace
\begin{aligned}
&\dfrac{t^{-2/3}}{24}(L^2-2r^2),\quad r<L,\\
&0,\quad r\geqslant L.
\end{aligned}
\right.
\label{ExactSol2}
\end{equation}

In our numerical investigation of the error and convergence analysis in space of the different numerical schemes discussed in section \ref{SectionNumMethods}, we compute the error between the exact solution \eqref{ExactSol1}-\eqref{ExactSol2} and the numerical ones. We choose \eqref{ExactSol1} at time $t=0.001$ as initial condition  for the regularity of the exact solution \eqref{ExactSol1}-\eqref{ExactSol2}.  We set $a=0$ in \eqref{Shspace} to impose the positivity of the discrete solution. The computations are conducted in the unit square centered at the origin. We take  $L=1$, a time step $\Delta t= 10^{-6}$, a final computational time $T=0.0012$ and a tolerance $\epsilon=10^{-10}$ for the stopping criteria \eqref{toleranceUzawa}, \eqref{toleranceSecant} and \eqref{toleranceBarrett}. The Dirichlet boundary conditions $u_h=\Pi_hu$ and $w_h=\Pi_hw$ are used. 

Tables~\ref{tab:ErrorUzawa}, \ref{tab:ErrorClassicalTruncation}, \ref{tab:ErrorConsTruncation} and \ref{tab:ErrorBarrett} show the order of convergence, computational time, $L^2$- and $H^1$-error for different grids for each of the schemes, respectively. We can see that all the schemes have the same order of convergence $\mathcal{O}(h^{1})$ for the $H^1$-error on the variable $u_h$. For the $L^2$-error on the variable $u_h$, a convergence  in  $\mathcal{O}(h^{2})$ is obtained with \textit{Scheme4} while a convergence in $\mathcal{O}(h^{1.2})$ is obtained for the three other schemes which produce almost the same error. The auxiliary variable \eqref{ExactSol2} calculated from the self-similar solution \eqref{ExactSol1} is not continuous at the points $r=L$. Hence, the $H^1$-error on $w$ is not defined. The reduce of the order of the convergence of the auxiliary variable $w_h$ is due to its lack of regularity. 

We have noticed from numerical computations that the choice $\alpha_k=\beta_k=\rho_k=1$ always leads to the convergence of the Uzawa algorithm \eqref{lambdak}-\eqref{uk} with the smallest number of iterations. For this test case, we also noticed that values for the parameter $\varrho$ around $3500$ lead to the convergence of \textit{Scheme4} with the smallest number of iterations. Table~\ref{tab:NBIterDiffSchemes} shows the parameters and the average number of iterations for the convergence of each iterative algorithm over time steps. Note that for \textit{Scheme4} and this test case, it takes about 600-700 iterations for the first steps to converge but the number of iterations  decreases as time evolves. Note also that the number of iterations appears to be independent from the mesh size and the chosen time step for \textit{Scheme1} and \textit{Scheme3}. Our numerical investigations also show that the secant iteration \eqref{SecantAlgorithm} always converges to a same value. In terms of computational time, \textit{Scheme2} is slightly faster than \textit{Scheme3} which is bit faster than \textit{Scheme1} which is significantly faster than \textit{Scheme4}. As opposed to the other three schemes, no iterative solver is needed in \textit{Scheme2}, which makes it the winner in terms of computational time. We also have noticed in practice that the iterative methods in \textit{Scheme1} and \textit{Scheme3} converge  more easily than \textit{Scheme4}.

\begin{table}[!h]
\centering
\begin{scriptsize}
\begin{tabular}{l|c|c|c|c|c|c|c|c|c|l|}
\cline{2-7}
&\multicolumn{4}{|c|}{$L^2$-error and order of convergence } & \multicolumn{2}{|c|}{$H^1$-error and order of convergence}  \\ \cline{2-7}
&\multicolumn{2}{|c|}{Error}&\multicolumn{2}{|c|}{Order} &\multicolumn{1}{|c|}{Error}&\multicolumn{1}{|c|}{Order} \\ \cline{1-5}\cline{8-8}
\multicolumn{1}{ |c| }{Mesh} &on $u_h$ & on $w_h$ &for $u_h$ &for $w_h$&on $u_h$&for $u_h$& CPU(s)\\ \cline{1-8}
\multicolumn{1}{ |c| }{$25\times 25$}&$10.8964\times 10^{-5}$&$0.1939$&$1.33$&$0.02$&$14.6660\times 10^{-3}$&$0.92$&$023$ \\ \cline{1-8}
\multicolumn{1}{ |c| }{$50\times 50$}&$4.31507\times 10^{-5}$&$0.1913$&$1.27$&$0.28$&$7.70985\times 10^{-3}$&$0.97$&$069$\\ \cline{1-8}
\multicolumn{1}{ |c| }{$100\times 100$}&$1.77780\times 10^{-5}$&$0.1569$&$1.21$&$0.34$&$3.92392\times 10^{-3}$&$0.97$&$218$ \\ \cline{1-8}
\multicolumn{1}{ |c| }{$200\times 200$}&$0.76670\times 10^{-5}$&$0.1237$&---&---&$1.99102\times 10^{-3}$&---&$753$\\ \cline{1-8}
\end{tabular}
\caption{Exact solution: Errors, order of convergence and CPU time for the solution computed with \textit{Scheme1}. $\alpha_k=1.0$, $\beta_k=1.0$ and $\rho_k=1.0$.}
\label{tab:ErrorUzawa}
\end{scriptsize}
\end{table}

\begin{table}[!h]
\centering
\begin{scriptsize}
\begin{tabular}{l|c|c|c|c|c|c|c|c|c|l|}
\cline{2-7}
&\multicolumn{4}{|c|}{$L^2$-error and order of convergence } & \multicolumn{2}{|c|}{$H^1$-error and order of convergence}  \\ \cline{2-7}
&\multicolumn{2}{|c|}{Error}&\multicolumn{2}{|c|}{Order} &\multicolumn{1}{|c|}{Error}&\multicolumn{1}{|c|}{Order} \\ \cline{1-5}\cline{8-8}
\multicolumn{1}{ |c| }{Mesh} &on $u_h$ & on $w_h$ &for $u_h$ &for $w_h$&on $u_h$ &for $u_h$ & CPU(s)\\ \cline{1-8}
\multicolumn{1}{ |c| }{$25\times 25$}&$11.1658\times 10^{-5}$&$0.2333$&$1.32$&$0.25$&$14.6717\times 10^{-3}$&$0.92$&$016$ \\ \cline{1-8}
\multicolumn{1}{ |c| }{$50\times 50$}&$4.46522\times 10^{-5}$&$0.1956$&$1.24$&$0.25$&$7.72307\times 10^{-3}$&$0.97$&$031$\\ \cline{1-8}
\multicolumn{1}{ |c| }{$100\times 100$}&$1.88082\times 10^{-5}$&$0.1635$&$1.19$&$0.35$&$3.93529\times 10^{-3}$&$0.97$&$125$ \\ \cline{1-8}
\multicolumn{1}{ |c| }{$200\times 200$}&$0.82037\times 10^{-5}$&$0.1280$&---&---&$1.99845\times 10^{-3}$&---&$612$\\ \cline{1-8}
\end{tabular}
\caption{Exact solution: Errors, order of convergence and CPU time for the solution computed with \textit{Scheme2}.}
\label{tab:ErrorClassicalTruncation}
\end{scriptsize}
\end{table}

\begin{table}[!h]
\centering
\begin{scriptsize}
\begin{tabular}{l|c|c|c|c|c|c|c|c|c|l|}
\cline{2-7}
&\multicolumn{4}{|c|}{$L^2$-error and order of convergence } & \multicolumn{2}{|c|}{$H^1$-error and order of convergence}  \\ \cline{2-7}
&\multicolumn{2}{|c|}{Error}&\multicolumn{2}{|c|}{Order} &\multicolumn{1}{|c|}{Error}&\multicolumn{1}{|c|}{Order} \\ \cline{1-5}\cline{8-8}
\multicolumn{1}{ |c| }{Mesh} &on $u_h$ & on $w_h$ &for $u_h$ &for $w_h$&on $u_h$ &for $u_h$ & CPU(s)\\ \cline{1-8}
\multicolumn{1}{ |c| }{$25\times 25$}&$10.8964\times 10^{-5}$&$0.1939$&$1.33$&$0.02$&$14.6660\times 10^{-3}$&$0.92$&$017$ \\ \cline{1-8}
\multicolumn{1}{ |c| }{$50\times 50$}&$4.31506\times 10^{-5}$&$0.1913$&$1.27$&$0.28$&$7.70985\times 10^{-3}$&$0.97$&$049$\\ \cline{1-8}
\multicolumn{1}{ |c| }{$100\times 100$}&$1.77780\times 10^{-5}$&$0.1569$&$1.21$&$0.34$&$3.92392\times 10^{-3}$&$0.97$&$180$ \\ \cline{1-8}
\multicolumn{1}{ |c| }{$200\times 200$}&$0.76670\times 10^{-5}$&$0.1237$&---&---&$1.99102\times 10^{-3}$&---&$635$\\ \cline{1-8}
\end{tabular}
\caption{Exact solution: Errors, order of convergence and CPU time for the solution computed with \textit{Scheme3}.}
\label{tab:ErrorConsTruncation}
\end{scriptsize}
\end{table}

\begin{table}[!h]
\centering
\begin{scriptsize}
\begin{tabular}{l|c|c|c|c|c|c|c|c|c|l|}
\cline{2-7}
&\multicolumn{4}{|c|}{$L^2$-error and order of convergence } & \multicolumn{2}{|c|}{$H^1$-error and order of convergence}  \\ \cline{2-7}
&\multicolumn{2}{|c|}{Error}&\multicolumn{2}{|c|}{Order} &\multicolumn{1}{|c|}{Error}&\multicolumn{1}{|c|}{Order} \\ \cline{1-5}\cline{8-8}
\multicolumn{1}{ |c| }{Mesh} &on $u_h$ & on $w_h$ &for $u_h$ &for $w_h$&on $u_h$&for $u_h$& CPU(s)\\ \cline{1-8}
\multicolumn{1}{ |c| }{$25\times 25$}&$3.59243\times 10^{-5}$&$1.7411$&$1.46$&$\times$&$14.9739\times 10^{-3}$&$0.95$&$00260$ \\ \cline{1-8}
\multicolumn{1}{ |c| }{$50\times 50$}&$1.30124\times 10^{-5}$&$0.8716$&$2.02$&$\times$&$7.73444\times 10^{-3}$&$0.98$&$01325$\\ \cline{1-8}
\multicolumn{1}{ |c| }{$100\times 100$}&$0.31919\times 10^{-5}$&$1.0040$&$1.97$&$\times$&$3.91320\times 10^{-3}$&$0.98$&$06605$ \\ \cline{1-8}
\multicolumn{1}{ |c|}{$200\times 200$}&$0.08126\times 10^{-5}$&$0.9427$&---&---&$1.97163\times 10^{-3}$&---&$34556$\\ \cline{1-8}
\end{tabular}
\caption{Exact solution: Errors, order of convergence and CPU time for the solution computed with \textit{Scheme4}. $\varrho=3500$.}
\label{tab:ErrorBarrett}
\end{scriptsize}
\end{table}

\begin{table}[!h]
\centering
\begin{tabular}{|c|c|c|}
\cline{1-3}
Methods & Parameters & Number of iterations  \\ \cline{1-3}
Uzawa iteration \eqref{lambdak}-\eqref{uk}  &  $\alpha_k=1$, $\beta_k=1$, $\rho_k=1$ &3-4\\ \cline{1-3}
Secant iteration \eqref{SecantAlgorithm}  & NA & 2-3   \\ \cline{1-3}
\textit{Scheme4}  & $\varrho=3500$ & 60-70\\ \cline{1-3}
\end{tabular}
\caption{ Parameters and average number of iterations for the convergence of the iterative algorithms.}
\label{tab:NBIterDiffSchemes}
\end{table}

\subsubsection{Conservation of mass}
Note that for $L=1$, the solution on the boundary of the computational domain is null for all $t\in[0.001,0.0012]$. Hence, the natural boundary conditions \eqref{HomoNeumannBoundCond} are satisfied and thus, \eqref{MassConservation} should  be satisfied by the numerical solutions. To numerically check the conservation of mass, we compute the relative error 
\begin{equation}
E_h(u_h^n):=\left\vert\dfrac{\int_{\Omega}u_h^n\,d\bm{x}-\int_{\Omega}u_0\,d\bm{x}}{\int_{\Omega}u_0\,d\bm{x}}\right\vert
\label{RelativeErrorMassCons}
\end{equation}
on the total mass computed at each time step for different meshes with each numerical scheme. The numerical results are displayed in Figure~\ref{Fig:MassConsSelfSimSol}.  We can see that the error \eqref{RelativeErrorMassCons} produced by each scheme converges to zero as the mesh is refined. \textit{Scheme1}, \textit{Scheme3} and \textit{Scheme4} perfectly preserve the discrete total mass. However, there is a defect of conservation of  mass with \textit{Scheme2} but the error decreases with the refinement of the mesh. 
\begin{figure}[!h]
\centering
\begin{tabular}{cc}
\includegraphics[scale=0.5]{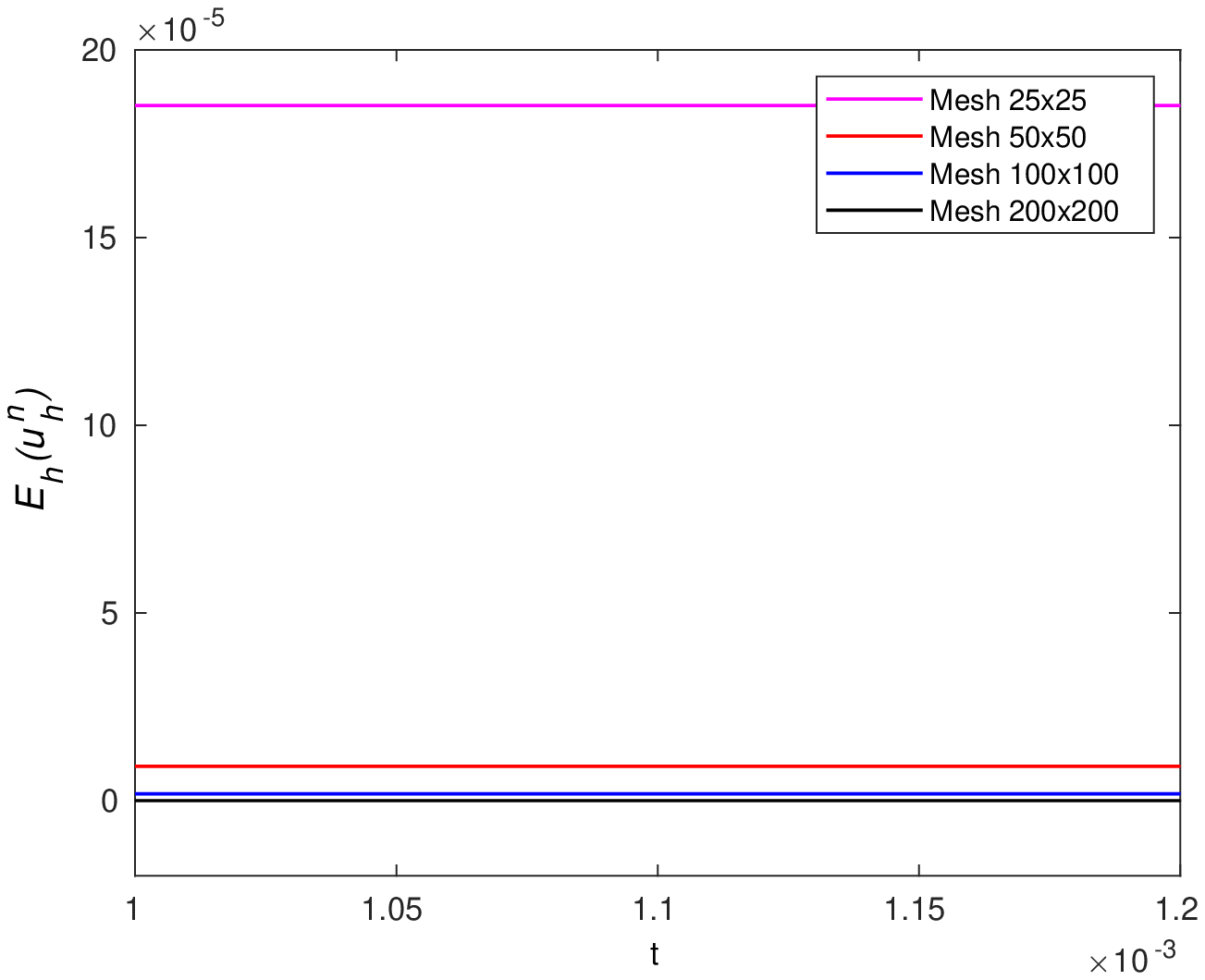}& \includegraphics[scale=0.5]{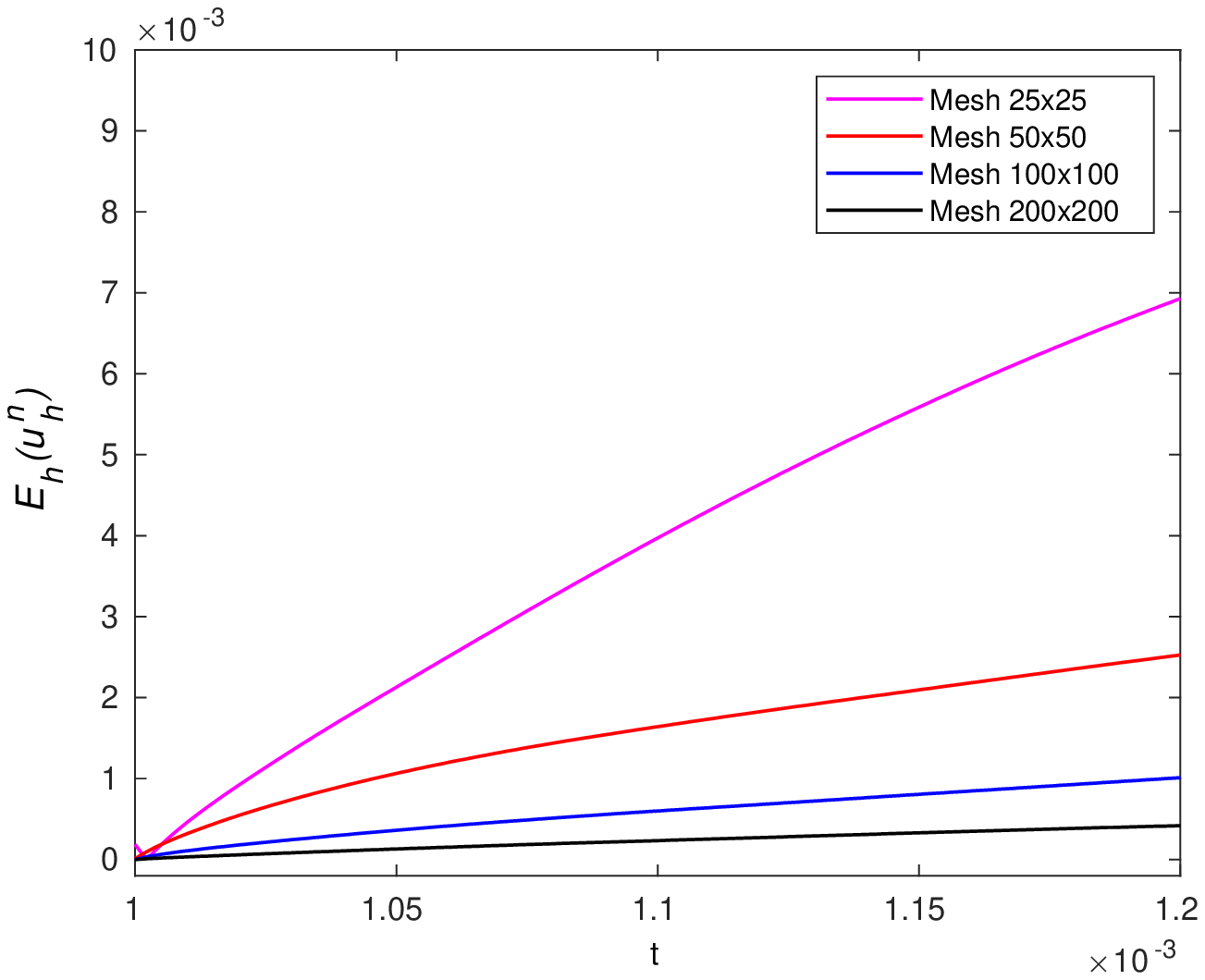}\\
\textit{Scheme1} & \textit{Scheme2}\\
\includegraphics[scale=0.5]{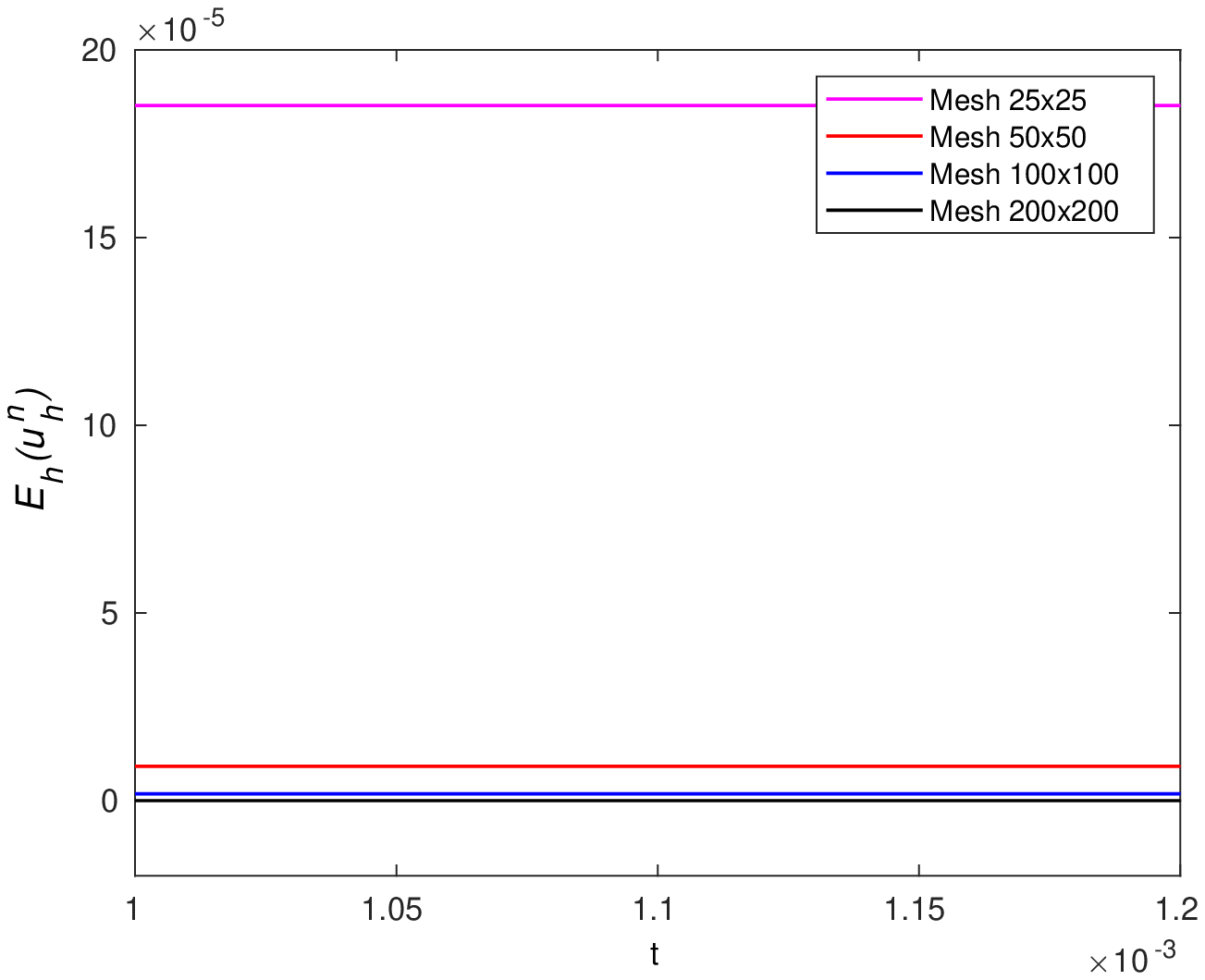}&\includegraphics[scale=0.5]{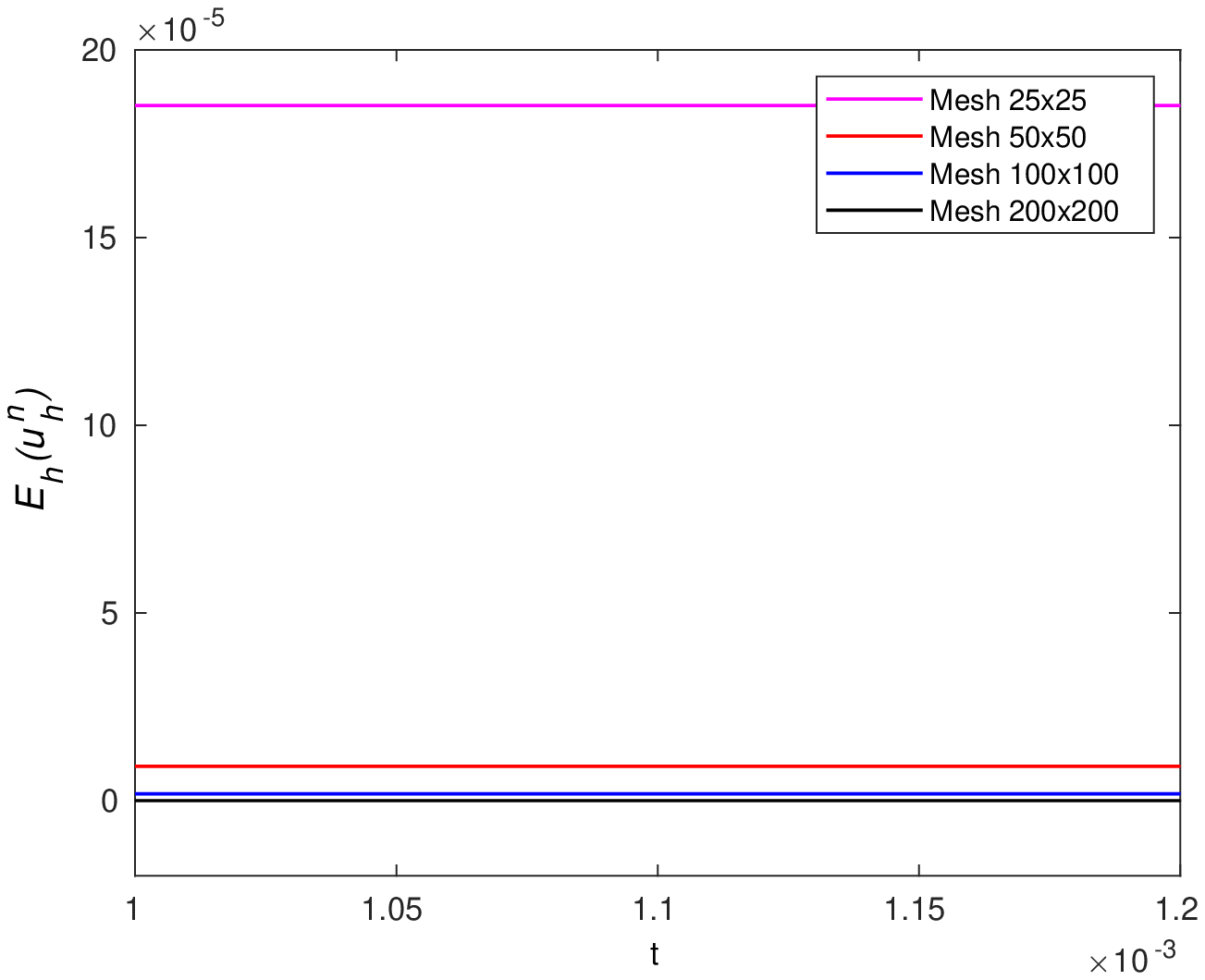}\\
\textit{Scheme3} & \textit{Scheme4} 
\end{tabular}
\caption{Relative error on the discrete total mass on different grids for each scheme.}
\label{Fig:MassConsSelfSimSol}
\end{figure}

\textit{Scheme3} has some important advantages with respect to the other three schemes. In fact, it is almost as fast and accurate as \textit{Scheme2} but better preserves the discrete mass than the latter. In addition to being faster and conservative, positivity of the computed solution is always satisfied as opposed to \textit{Scheme1} and \textit{Scheme4}, for which the positivity of the computed solution is only guaranteed at the limit $\epsilon\rightarrow0$, a condition that one cannot reach in practice.

\subsubsection{Free boundary propagation}
Here, we test the robustness of \textit{Scheme3} for tracking the free boundary (interface position) using the exact solution \eqref{ExactSol1}. The interface position corresponds to the boundary of the region where the solution $u$ is strictly positive. For this analysis, we consider the solution computed with the mesh $200\times200$ and compute the contour $u_h=\delta_{u_h}$, $0<\delta_{u_h}\ll1$ of the numerical solution at different times. Since the exact solution $u$ vanishes in the whole region $r\geqslant 1$, the interface position of the solution is approximated as the contour $u=10^{-12}$. Figure~\ref{Fig:ContoursDiffTimesSelfSimSol} shows the contours of the exact and numerical solutions at different times.  For the contours of the numerical solution, the value $\delta_{u_h}=10^{-5}$ has been chosen since the magnitude of the $L^2$-error between the exact and numerical solution for the mesh $200\times200$ is around $10^{-5}$ as shown in Table~\ref{tab:ErrorConsTruncation}. Note also that for this solution, the maximum occurs at the origin for all time $t$ and is given by $\frac{t^{-1/3}L^4}{192}$ which is far from $10^{-5}$  in the time interval $[0.001,0.02]$, confirming that $\delta_{u_h}=10^{-5}$ is a reasonable choice. We observe a good agreement between the exact and numerical interface positions.
\begin{figure}[!h]
\centering
\begin{tabular}{cccc}
\includegraphics[scale=0.2]{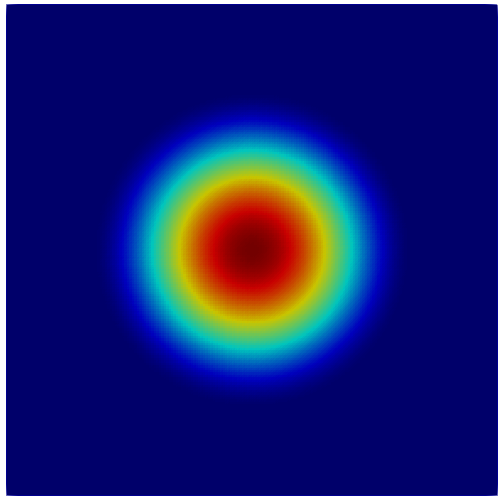}&\includegraphics[scale=0.2]{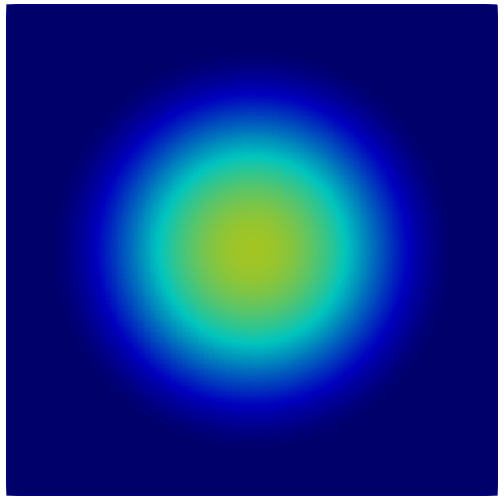}&\includegraphics[scale=0.2]{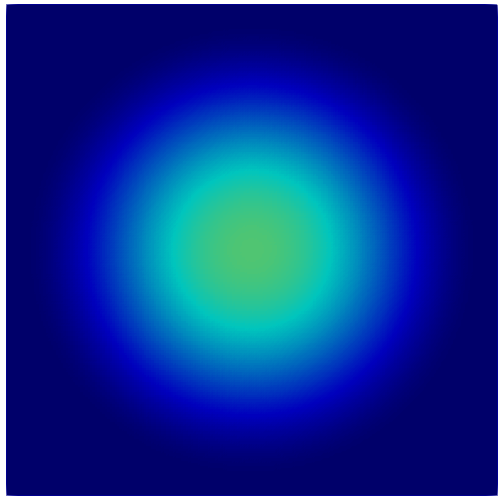}&\includegraphics[scale=0.2]{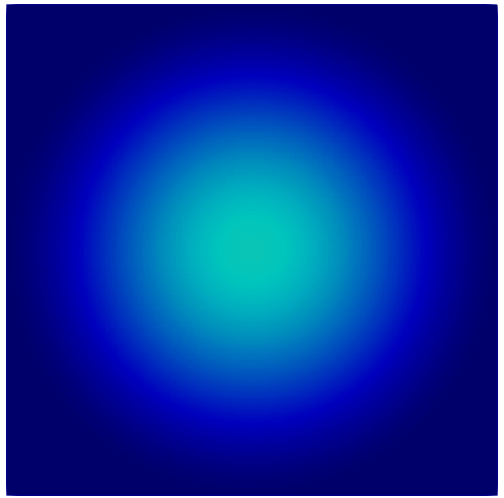}\\
\includegraphics[scale=0.2]{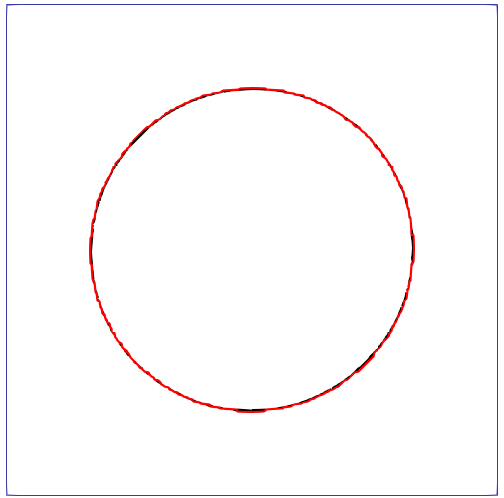}&\includegraphics[scale=0.2]{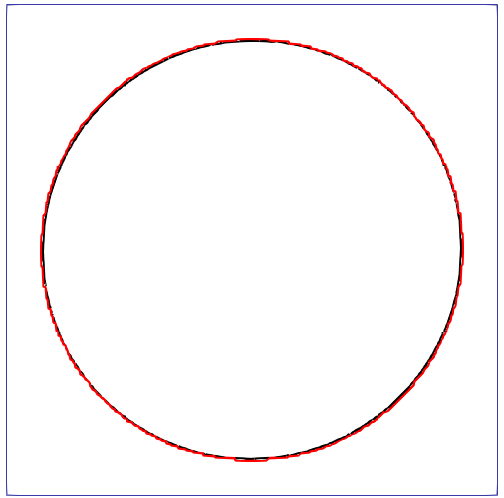}&\includegraphics[scale=0.2]{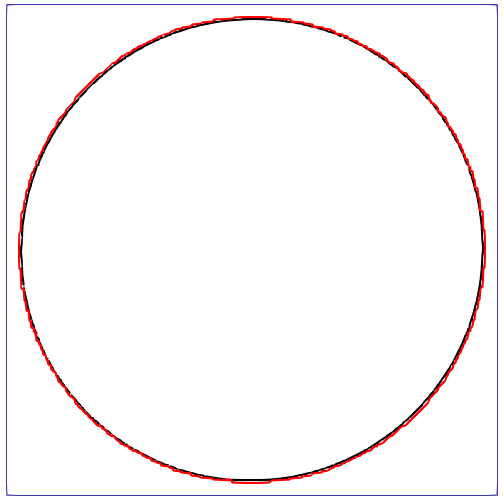}&\includegraphics[scale=0.2]{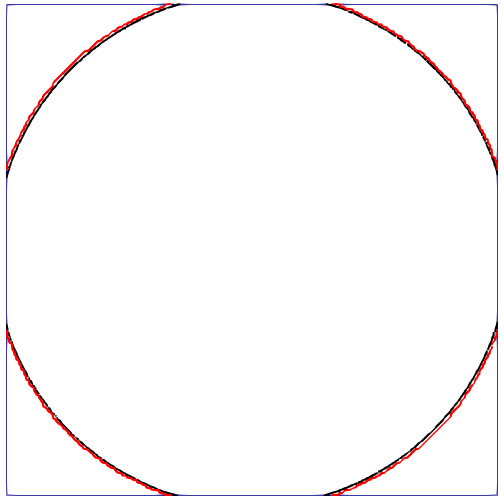}\\
$t=0.0002$&$t=0.005$&$t=0.01$& $t=0.02$
\end{tabular}
\caption{Evolution of the numerical solution (top) and the interface position (bottom) of the exact and numerical solutions at different times.}
\label{Fig:ContoursDiffTimesSelfSimSol}
\end{figure}

To test whether the chosen time step  has an impact or not on the evolution of the free boundary, we compute the interface position of the numerical solution for different values of $\Delta t$. Figure~\ref{Fig:ContoursDiffTimeStepSelfSimSol} shows the free boundary at time $t=0.01$ for different values of the time step. The numerical experiments show that the propagation of the free boundary is not sensitive to the chosen time step as long as the latter does not impact the stability of the solution.
\begin{figure}[!h]
\centering
\includegraphics[scale=0.2]{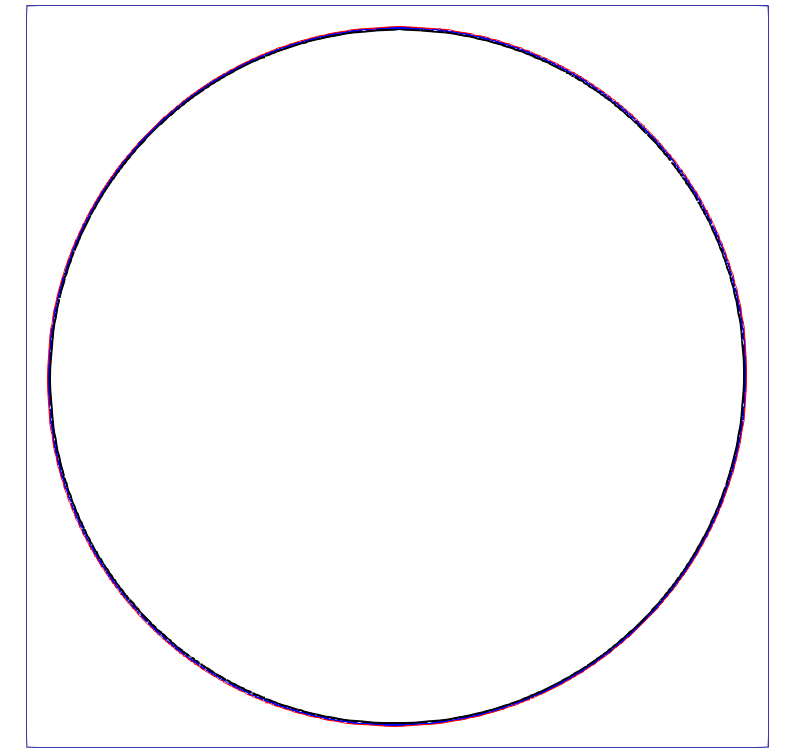}
\caption{Interface position of the numerical solution at time $t=0.001$ for different values of the time step: $\Delta t=10^{-5}$ (blue), $10^{-6}$ (black) and $ 10^{-7}$ (red).}
\label{Fig:ContoursDiffTimeStepSelfSimSol}
\end{figure}

\subsection{Manufactured solution}
\label{SubsectionManSol}
We consider a regular manufactured solution for \eqref{ThinFilmEquation} of the form 
\begin{equation}
u(r,t)=
\left\lbrace
\begin{aligned}
&C\exp\left(-\dfrac{\sigma(t)}{L^2-r^2}\right), \quad r<L,\\
&0,\quad r\geqslant L,
\end{aligned}
\right.
\label{uExactManSol}
\end{equation}
where $r=r(\bm{x},t):=|\bm{x}|/\beta(t)$, the given functions $\sigma:[0,T]\rightarrow\mathbb{R}$ and $\beta:[0,T]\rightarrow\mathbb{R}$ are smooth, $L>0$ and $C>0$ are constants. The manufactured solution \eqref{uExactManSol} is a $\mathcal{C}^{\infty}$-function on the whole plane, as opposed to the previous test case, where we have piecewise regularity. We set the auxiliary variable $w=-\gamma\Delta u$.  Since the function \eqref{uExactManSol} does not satisfy the partial differential equation \eqref{ThinFilmEquation}, one has to find the residual
\begin{equation}
\begin{aligned}
S=u_t-\nabla\cdot(f(u)\nabla w)&=u_t-\nabla\cdot\left(f(u)\Big(\frac{\partial w}{\partial r}\frac{x}{r\beta(t)^2},\frac{\partial w}{\partial r}\frac{y}{r\beta(t)^2}\Big)^T\right),\\
&=u_t-\dfrac{1}{\beta(t)^2}\left(\dfrac{1}{r}\dfrac{\partial w}{\partial r}\Big(f(u)+rf^{\prime}(u)\dfrac{\partial u}{\partial r}\Big)+f(u)\dfrac{\partial^2 w}{\partial r^2}\right).
\end{aligned}
\label{ResidualGeneral}
\end{equation} 

We provide the main terms in \eqref{ResidualGeneral} to avoid lengthy calculations to readers interested in using this  manufactured solution. We  calculate 
\begin{align}
\dfrac{\partial u}{\partial r}&=-\dfrac{2\sigma ru }{(L^2-r^2)^2}, \label{duExactManSol}\\
\dfrac{\partial^2 u}{\partial r^2}&=-\dfrac{2\sigma u(L^2+3r^2)}{(L^2-r^2)^3}+\dfrac{4\sigma^2 r^2 u}{(L^2-r^2)^4}.\label{dduExactManSol}
\end{align}
From \eqref{duExactManSol}-\eqref{dduExactManSol}, we obtain 
\begin{equation}
w=-\gamma\left(\frac{1}{r}\dfrac{\partial u}{\partial r}+\dfrac{\partial^2 u}{\partial r^2}\right)=2\gamma\sigma u\left(\dfrac{1}{(L^2-r^2)^2}+\dfrac{L^2+3r^2}{(L^2-r^2)^3}-\dfrac{2\sigma r^2}{(L^2-r^2)^4}\right).
\label{wExactManSol}
\end{equation}
Equation \eqref{wExactManSol} leads to the following expression
\begin{equation}
\begin{aligned}
\dfrac{\partial w}{\partial r}=&2\gamma\sigma\Big\lbrace\dfrac{\partial u}{\partial r}\Big(\dfrac{1}{(L^2-r^2)^2}+\dfrac{L^2+3r^2}{(L^2-r^2)^3}-\dfrac{2\sigma r^2}{(L^2-r^2)^4}\Big)\\
&+u\Big(\dfrac{4r}{(L^2-r^2)^3}+\dfrac{12r(L^2+r^2)}{(L^2-r^2)^4}-\dfrac{4\sigma r(L^2+3r^2)}{(L^2-r^2)^5}\Big)\Big\rbrace.
\end{aligned}
\label{dwExactManSol}
\end{equation}
From \eqref{dwExactManSol}, we calculate 
\begin{equation}
\begin{aligned}
\dfrac{\partial^2 w}{\partial r^2}=&2\gamma\sigma\Big\lbrace\dfrac{\partial^2 u}{\partial r^2}\Big(\dfrac{1}{(L^2-r^2)^2}+\dfrac{L^2+3r^2}{(L^2-r^2)^3}-\dfrac{2\sigma r^2}{(L^2-r^2)^4}\Big)\\
& + 2 \dfrac{\partial u}{\partial r}\Big(\dfrac{4r}{(L^2-r^2)^3}+\dfrac{12r(L^2+r^2)}{(L^2-r^2)^4} - \dfrac{4\sigma r(L^2+3r^2)}{(L^2-r^2)^5}\Big)\\
& + u\Big(\dfrac{4(L^2+5r^2)}{(L^2-r^2)^4}+\dfrac{12(L^4+10L^2r^2+5r^4)}{(L^2-r^2)^5}-\dfrac{4\sigma(L^4+18L^2r^2+21r^4)}{(L^2-r^2)^6}\Big)\Big\rbrace.
\end{aligned}
\label{ddwExactManSol}
\end{equation}
We end up with the system 
\begin{equation}
\begin{aligned}
&\partial_tu-\nabla\cdot\big(f(u)\nabla w\big)=S \quad \text{in } \Omega_T,\\
&w=-\gamma\Delta u \quad \text{in } \Omega_T.
\end{aligned}
\label{ThinFilmManufSolSyst}
\end{equation}
The manufactured solution \eqref{uExactManSol} does not satisfy the mass conservation property \eqref{MassConservation}. Therefore, one does not need to impose this property on the discrete solution. In such case, \textit{Scheme3}  reduces to \textit{Scheme2} and if convergence occurs for \textit{Scheme1} then the solution converges to the solution of \textit{Scheme2} after a finite number of iterations (see the last paragraph in section \ref{SubSectionClassTrunctionMethod}).

\subsubsection{A constant mobility function}
Assume that $f(u)=M\in \mathbb{R}$ is constant. From \eqref{ResidualGeneral}, we get
\begin{equation}
S=u_t-\dfrac{M}{\beta(t)^2}\left(\frac{1}{r}\dfrac{\partial w}{\partial r} + \dfrac{\partial^2 w}{\partial r^2}\right).
\label{ResConsMob}
\end{equation}

In our numerical investigation of the order of convergence in space, we compute the $L^2$-error between the manufactured solution  and the numerical solutions computed with \textit{Scheme2} and \textit{Scheme4}, respectively.  We choose a time step $\Delta t= 0.001$, a final computational time  $T=0.5$, a tolerance $\epsilon=10^{-8}$ and set
\begin{equation}
\gamma=1,\quad M=1, \quad C=1,\quad L=0.5,\quad \sigma(t)= 1,\quad \beta(t)=t+1.
\end{equation}
The Dirichlet  boundary conditions $u_h=\Pi_hu$ and $w_h=\Pi_hw$ are used. Tables~\ref{tab:ManSolConsMobErrorTrunc} and \ref{tab:ManSolConsMobErrorBarrett} show the errors, rate of convergence and computational cost for different grids for \textit{Scheme2} and \textit{Scheme4}, respectively. Optimal convergence rates, i.e.\  $\mathcal{O}(h^{2})$ and $\mathcal{O}(h)$,  are obtained for the $L^2$- and $H^1$-error, respectively, with both schemes. This is consistent with the fact that $u$ is regular on the whole plane. However, the errors on $w_h$ are larger with both norms. 
\begin{table}[!h]
\centering
\begin{scriptsize}
\begin{tabular}{l|c|c|c|c|c|c|c|c|c|l|}
\cline{2-9}
&\multicolumn{4}{|c|}{$L^2$-error and order of convergence } & \multicolumn{4}{|c|}{$H^1$-error and order of convergence}  \\ \cline{2-9}
&\multicolumn{2}{|c|}{Error}&\multicolumn{2}{|c|}{Order} &\multicolumn{2}{|c|}{Error}&\multicolumn{2}{|c|}{Order} \\ \cline{1-10}
\multicolumn{1}{ |c| }{Mesh} &on $u_h$ & on $w_h$ &for $u_h$ &for $w_h$ &on $u_h$ &on $w_h$ &for $u_h$ & for $w_h$ & CPU(s) \\ \cline{1-10}
\multicolumn{1}{ |c| }{$25\times 25$}&$14.7431\times 10^{-5}$& $13.8107\times 10^{-3}$&$2.17$&$1.77$ &$4.5420\times 10^{-3}$& $1.0408$&$0.99$&$0.87$&$0045$  \\ \cline{1-10}
\multicolumn{1}{ |c| }{$50\times 50$}&$3.26144\times 10^{-5}$&$4.04323\times 10^{-3}$&$1.97$&$1.91$ &$2.2728\times 10^{-3}$&$0.5656$&$0.99$&$0.96$&$0170$   \\ \cline{1-10}
\multicolumn{1}{ |c| }{$100\times 100$}&$0.82728\times 10^{-5}$&$1.07304\times 10^{-3}$&$1.99$&$1.97$ &$1.1376\times 10^{-3}$&$0.2907$&$0.99$&$0.98$&$0677$ \\ \cline{1-10}
\multicolumn{1}{ |c| }{$200\times 200$}&$0.20761\times 10^{-5}$&$0.27272\times 10^{-3}$ & --- & ---&$0.5689\times 10^{-3}$&$0.1464$  & --- & ---&$2174$ \\ \cline{1-10}
\end{tabular}
\caption{Manufactured Solution: Errors, order of convergence and CPU time of the solution computed with \textit{Scheme2} for a constant mobility function.}
\label{tab:ManSolConsMobErrorTrunc}
\end{scriptsize}
\end{table}

\begin{table}[!h]
\centering
\begin{scriptsize}
\begin{tabular}{l|c|c|c|c|c|c|c|c|c|c|c|l|}
\cline{2-9}
&\multicolumn{4}{|c|}{$L^2$-error and order of convergence } & \multicolumn{4}{|c|}{$H^1$-error and order of convergence}  \\ \cline{2-9}
&\multicolumn{2}{|c|}{Error}&\multicolumn{2}{|c|}{Order} &\multicolumn{2}{|c|}{Error}&\multicolumn{2}{|c|}{Order} \\ \cline{1-10}
\multicolumn{1}{ |c| }{Mesh} &on $u_h$ & on $w_h$ &for $u_h$ &for $w_h$&on $u_h$ &on $w_h$ &for $u_h$ & for $w_h$  & CPU(s)\\ \cline{1-10}
\multicolumn{1}{ |c| }{$25\times 25$}&$14.6657\times 10^{-5}$&$13.8067\times 10^{-3}$&$2.14$&$1.77$&$4.5414\times 10^{-3}$&$1.0408$&$0.99$&$0.87$ & $0081$ \\ \cline{1-10}
\multicolumn{1}{ |c| }{$50\times 50$}&$3.31382\times 10^{-5}$&$4.04353\times 10^{-3}$&$1.98$&$1.91$&$2.2727\times 10^{-3}$&$0.5656$&$0.99$&$0.96$ & $0301$\\ \cline{1-10}
\multicolumn{1}{ |c| }{$100\times 100$}&$0.83959\times 10^{-5}$&$1.07317\times 10^{-3}$&$1.99$&$1.97$&$1.1375\times 10^{-3}$&$0.2907$&$0.99$&$0.99$&$1524$ \\ \cline{1-10}
\multicolumn{1}{ |c| }{$200\times 200$}&$0.21064\times 10^{-5}$&$0.27275\times 10^{-3}$&---&---&$0.5689\times 10^{-3}$&$0.1454$& --- & --- &$8739$\\ \cline{1-10}
\end{tabular}
\caption{Manufactured Solution: Errors, order of convergence and CPU time for the solution computed with \textit{Scheme4} for a constant mobility function. $\varrho=900$.}
\label{tab:ManSolConsMobErrorBarrett}
\end{scriptsize}
\end{table}

\subsubsection{An $u$-dependent mobility function}
Now we assume that $f(u)=u$. From \eqref{ResidualGeneral}, we get
\begin{equation}
S=u_t-\dfrac{1}{\beta(t)^2}\left(\dfrac{1}{r}\dfrac{\partial w}{\partial r}\big(u+r\dfrac{\partial u}{\partial r}\big)+u\dfrac{\partial^2 w}{\partial r^2}\right).
\end{equation} 
In our convergence analysis for non-constant mobility, we proceed as in the previous subsection and take the same parameter and boundary conditions except for the time step $\Delta t= 10^{-5}$ and the final computational time $t=0.005$. Tables~\ref{tab:ManSolNonConsMobErrorTrunc} and \ref{tab:ManSolNonConsMobErrorBarrett} show the errors, rate of convergence and computational cost for different grids for \textit{Scheme2} and \textit{Scheme4}, respectively. Optimal convergence rates in the $L^2$- and $H^1$-error are also obtained for a $u$-dependent mobility function.  The difference of order of convergence  between the test cases in sections \ref{SubSectSelfSimSol} and \ref{SubsectionManSol} confirms the impact of the regularity of the solution on the spatial order of convergence. 
\begin{table}[!h]
\centering
\begin{scriptsize}
\begin{tabular}{l|c|c|c|c|c|c|c|c|c|l|}
\cline{2-9}
&\multicolumn{4}{|c|}{$L^2$-error and order of convergence} & \multicolumn{4}{|c|}{$H^1$-error and order of convergence}  \\ \cline{2-9}
&\multicolumn{2}{|c|}{Error}&\multicolumn{2}{|c|}{order}&\multicolumn{2}{|c|}{Error}&\multicolumn{2}{|c|}{order}  \\ \cline{1-10}
\multicolumn{1}{ |c| }{Mesh} & on $u_h$ & on $w_h$ & for $u_h$ &for $w_h$ & on $u_h$ & on $w_h$ & for $u_h$ &for $w_h$   & CPU(s)  \\ \cline{1-10}
\multicolumn{1}{ |c| }{$25\times 25$}&$4.13823\times 10^{-5}$& $3.59119\times 10^{-2}$&$1.98$&$1.50$ &$4.5494\times 10^{-3}$& $2.5509$&$0.99$&$1.15$&$0046$ \\ \cline{1-10}
\multicolumn{1}{ |c| }{$50\times 50$}&$1.04799\times 10^{-5}$&$1.26161\times 10^{-2}$&$1.97$&$1.92$ &$2.2837\times 10^{-3}$& $1.1486$&$1.00$&$1.46$&$0190$ \\ \cline{1-10}
\multicolumn{1}{ |c| }{$100\times 100$}&$0.26651\times 10^{-5}$&$0.33111\times 10^{-2}$&$1.99$&$1.97$  &$1.1389\times 10^{-3}$& $0.4161$&$1.00$&$1.30$&$0677$ \\ \cline{1-10}
\multicolumn{1}{ |c| }{$200\times 200$}&$0.06695\times 10^{-5}$&$0.08401\times 10^{-2}$  & --- & --- &$0.5690\times 10^{-3}$&$0.1683$  & --- & ---&$2200$  \\ \cline{1-10}
\end{tabular}
\caption{Manufactured Solution: Errors, order of convergence and CPU time of the solution computed with \textit{Scheme2} for a $u$-dependent mobility function.}
\label{tab:ManSolNonConsMobErrorTrunc}
\end{scriptsize}
\end{table}
\begin{table}[!h]
\centering
\begin{scriptsize}
\begin{tabular}{l|c|c|c|c|c|c|c|c|c|c|c|l|}
\cline{2-9}
&\multicolumn{4}{|c|}{$L^2$-error and order of convergence } & \multicolumn{4}{|c|}{$H^1$-error and order of convergence}  \\ \cline{2-9}
&\multicolumn{2}{|c|}{Error}&\multicolumn{2}{|c|}{Order} &\multicolumn{2}{|c|}{Error}&\multicolumn{2}{|c|}{Order} \\ \cline{1-10}
\multicolumn{1}{ |c| }{Mesh} &on $u_h$ & on $w_h$ &for $u_h$ &for $w_h$&on $u_h$ &on $w_h$ &for $u_h$ & for $w_h$  & CPU(s)\\ \cline{1-10}
\multicolumn{1}{ |c| }{$25\times 25$}&$3.44889\times 10^{-5}$&$5.43728\times 10^{-2}$&$1.67$&$2.36$&$4.5609\times 10^{-3}$&$2.3486$&$0.99$&$1.42$ & $0077$ \\ \cline{1-10}
\multicolumn{1}{ |c| }{$50\times 50$}&$1.08011\times 10^{-5}$&$1.05393\times 10^{-2}$&$1.97$&$1.98$&$2.2826\times 10^{-3}$&$0.8731$&$1.00$&$1.37$ & $0292$\\ \cline{1-10}
\multicolumn{1}{ |c| }{$100\times 100$}&$0.27431\times 10^{-5}$&$0.26535\times 10^{-2}$&$2.01$&$1.98$&$1.1387\times 10^{-3}$&$0.3372$&$1.00$&$1.13$&$1517$ \\ \cline{1-10}
\multicolumn{1}{ |c| }{$200\times 200$}&$0.06785\times 10^{-5}$&$0.06704\times 10^{-2}$&---&---&$0.5690\times 10^{-3}$&$0.1540$& --- & --- &$8564$\\ \cline{1-10}
\end{tabular}
\caption{Manufactured Solution: Errors, order of convergence and CPU time for the solution computed with \textit{Scheme4} for a $u$-dependent mobility function. $\varrho=900$.}
\label{tab:ManSolNonConsMobErrorBarrett}
\end{scriptsize}
\end{table}

With the manufactured solution \eqref{uExactManSol}, we numerically observed that \textit{Scheme4} converges for values of the parameter $\varrho$ around $900$, and it takes in average 4 iterations to converge for each time step. The first steps take more iterations ( around 15-20 iterations) for the convergence. \textit{Scheme2} is faster and as accurate as \textit{Scheme4}.

In the error and convergence analysis in time for \textit{Scheme2}, we compute a reference solution $(u_h^{\text{ref}}, w_h^{\text{ref}})$ on a grid $100\times 100$  with a small time step $\Delta t=10^{-7}$ to get a good approximation of the semi-discretized solution of the problem. We then vary the time step systematically between $5\times 10^{-5}$ and $4\times 10^{-4}$ using the same grid. For each time step, the $L^2$-error on $u_h$ and $w_h$, computed as 
\begin{equation}
\left(\int_{\Omega}\vert u_h^{\text{ref}} -u_h^{n}\vert^2\,d\bm{x}\right)^{1/2}\quad \text{ and }\quad\left(\int_{\Omega}\vert w_h^{\text{ref}} -w_h^{n}\vert^2\,d\bm{x}\right)^{1/2},
\end{equation}
respectively, are recorded at the final time $T=0.01$. Table~\ref{tab:L2errorManSolNonConsMobTimeOrderConv} shows the $L^2$-error, computational time and order of convergence of the numerical solution with respect to the reference solution. We obtain a second-order convergence rate in time for the main variable $u_h$. Only a first-order convergence rate is observed for the auxiliary variable $w_h$. A second-order convergence rate is obtained for both variables when no projection is done \cite{Keita2021}. However, positivity is not guaranteed in \cite{Keita2021}, i.e.\ only the linear system \eqref{VariationalEqualityScheme} is solved at each time step. We think that it is the projection, which guarantees the positivity, that leads to a loss of order for $w_h$.
\begin{table}[h!]
\centering
\begin{tabular}{c|c|c|c|c|c|c|c}
\cline{2-5}
\multicolumn{1}{c|}{}&\multicolumn{2}{|c|}{$L^2$-error}&\multicolumn{2}{|c|}{Order of convergence}&\multicolumn{1}{|c}{} \\ \cline{1-6}
\multicolumn{1}{|c|}{$\Delta t$}& Error on $u_h$&Error on $w_h$ & Order for $u_h$ & Order for $w_h$ & CPU(s)\\ \cline{1-6}
\multicolumn{1}{|c|}{$4\times 10^{-4}$}&$31.0712\times 10^{-9}$&$2.51871\times 10^{-5}$&$2.01$&$1.22$ &$024$\\ \cline{1-6}
\multicolumn{1}{|c|}{$2\times 10^{-4}$}& $7.67188\times 10^{-9}$&$1.07441\times 10^{-5}$&$2.00$&$1.05$ &$049$ \\ \cline{1-6}
\multicolumn{1}{|c|}{$1\times 10^{-4}$}& $1.90737\times 10^{-9}$& $0.51627\times 10^{-5}$&$1.99$&$1.02$ &$100$\\ \cline{1-6}
\multicolumn{1}{|c|}{$5\times 10^{-5}$}& $0.47929\times 10^{-9}$& $0.25445\times 10^{-5}$& --- & ---  & $188$\\ \cline{1-6}
\end{tabular}
\caption{$L^2$-error as a function of time step $\Delta t$, order of convergence in time and CPU time.}
\label{tab:L2errorManSolNonConsMobTimeOrderConv}
\end{table}

 \subsection{Lubrication}
 \subsubsection{Energy decreasing and mass conservation properties}
Model \eqref{ThinFilmEquation}-\eqref{MobFunc} describes the height of a liquid film  which spreads on a solid surface \cite{greenspan1978,Hocking1981}. Here, we solve \eqref{ThinFilmEquation} with \textit{Scheme3} for the initial data
\begin{equation}
u(x,y,0)=\delta+Ce^{-\sigma(x^2+y^2)},
\label{LubInitCond}
\end{equation}
where $\delta>0$ represents the thickness of an ultra-thin precursor film under the Gaussian fluid droplet centered at the origin.  We use the mobility function \eqref{MobFunc}, where $f_0(u)\equiv1$ and $p=1$, instead of the regularized mobility function \eqref{RegMobFunc} which is used in previous studies \cite{Bertozzi1,WITELSKI2003} to perform numerical simulations using the initial data \eqref{LubInitCond}. The computational domain is the rectangle $\Omega=[-0.5,0.5]\times [-2,2]$. Homogeneous Neumann boundary conditions \eqref{HomoNeumannBoundCond} are used. As in \cite{Bertozzi1},  we use a uniform unstructured mesh of size $70\times 280$, i.e.\ a mesh size $h=0.02$, take $\gamma=1$, $\sigma=80$ and $C=2$. We set $\delta=0$ instead of $\delta=10^{-2}$ which is used in \cite{Bertozzi1,WITELSKI2003}. Figure~\ref{Fig:Contours3dViewDropletSpreading} shows the contours of the numerical solutions and a 3D view of the droplet at different times. No oscillation is observed in the numerical solution. 
\begin{figure}[!h]
\centering
\begin{tabular}{cccc}
\includegraphics[scale=0.2]{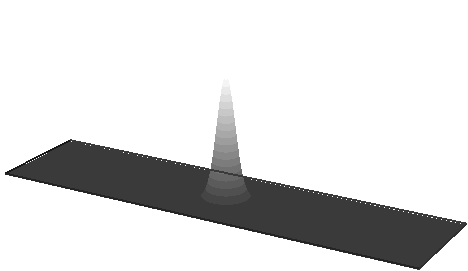}&\includegraphics[scale=0.2]{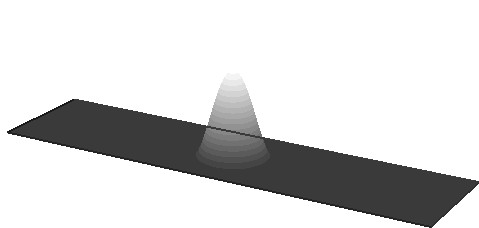}&\includegraphics[scale=0.2]{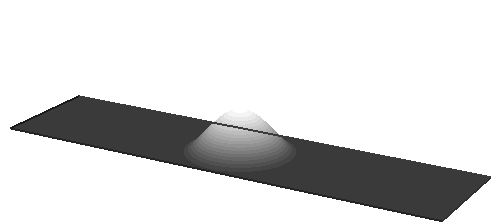}&\includegraphics[scale=0.2]{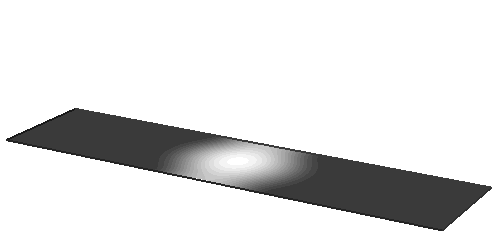}\\
\includegraphics[scale=0.25]{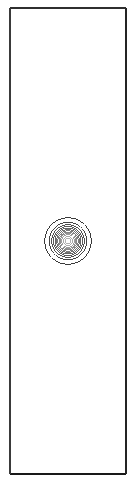}&\includegraphics[scale=0.25]{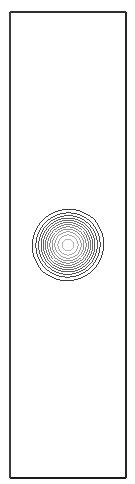}&\includegraphics[scale=0.25]{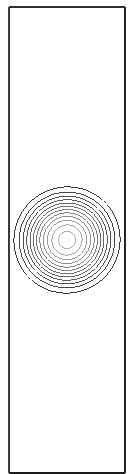}&\includegraphics[scale=0.25]{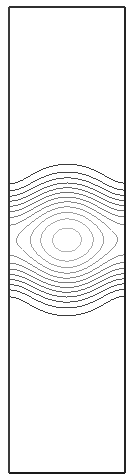}\\
$t=0$&$t=0.0001$&$t=0.0012$&$t=0.01$
\end{tabular}
\caption{Top: A 3D view of the droplet spreading on a solid surface. Bottom: Contours at various times of the numerical solution. $\Delta t=10^{-5}$.}
\label{Fig:Contours3dViewDropletSpreading}
\end{figure}

We now investigate the energy decreasing property \eqref{EnergyProperty} and the total mass conservation.  The discrete total free energy of \eqref{ThinFilmEquation} is defined as
\begin{equation}
J(u_h^n)=\int_{\Omega}\frac{\gamma}{2}\vert\nabla u_h^n\vert^2\,d\bm{x}.
\label{LubDiscreteEnergy}
\end{equation}
Figure~\ref{Fig:FreeEnergyHeightAndMassCons}$(a)$ shows the time-evolution of the discrete total energy \eqref{LubDiscreteEnergy} of the system. The numerical simulations show that the system dissipates the total energy, which agrees well with the energy dissipation property of the equation. The maximum of the solution of \eqref{ThinFilmEquation} with the initial condition \eqref{LubInitCond} occurs at the origin for all times.  Figure~\ref{Fig:FreeEnergyHeightAndMassCons}$(b)$ shows the height of the droplet at the origin as time evolves. After an initial transient phase, the solution evolves like the self-similar source solution \eqref{ExactSol1} for some times, as one can see with the maximal height $u_h^n(0)$ evolving in $t^{-1/3}$. For longer times, the wave reaches the boundary which now dominates the evolution of the solution that becomes no longer radially symmetric. The gradients across the $x$-direction approach zero and the solution evolves like a source type similarity solution  for the one-dimensional case \cite{BernisetAl1992}, as reflected by the maximal height $u_h^n(0)$ evolving in $t^{-1/5}$, before approaching a uniform flat state for large times. These three stages of the evolution of the solution have a slight impact on the  dissipation of the total energy as one can see in Figure~\ref{Fig:FreeEnergyHeightAndMassCons}$(a)$ with the corresponding intervals of times of the three stages in the evolution of the discrete total energy. Figure~\ref{Fig:FreeEnergyHeightAndMassCons}$(c)$ shows that the discrete total mass is preserved over time. The same evolutions are observed in \cite{Bertozzi1,WITELSKI2003} where a regularized mobility \eqref{RegMobFunc} and a thickness $\delta=10^{-2}$ are used by the authors in addition to preserve the positivity of the solution computed with their schemes. This provides a convenient way for validating and showing the robustness of \textit{Scheme3} for preserving the mass, the total energy decreasing property and the positivity of the solution.
\begin{figure}[!h]
\centering
\begin{tabular}{ccc}
\includegraphics[scale=0.38]{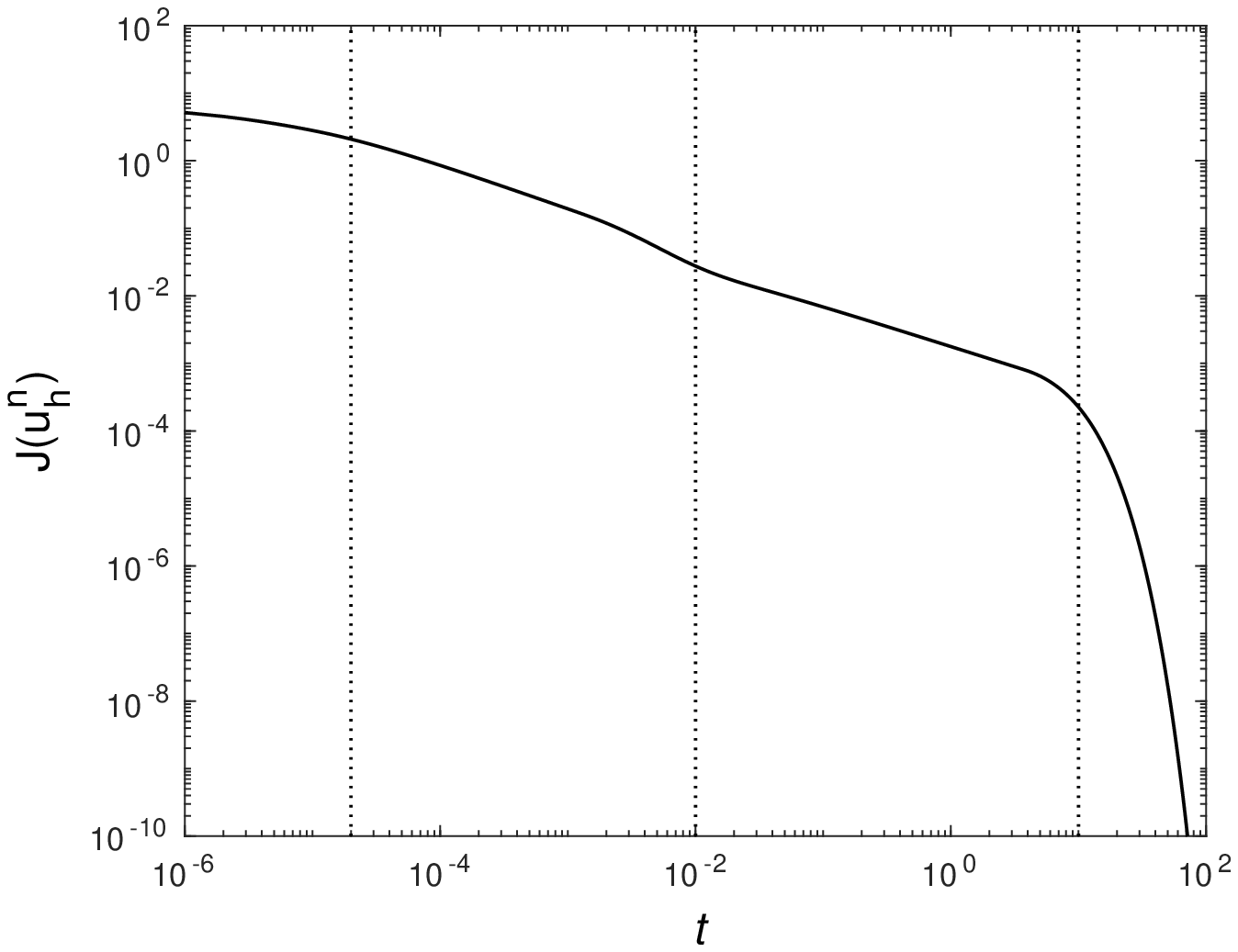}&\includegraphics[scale=0.38]{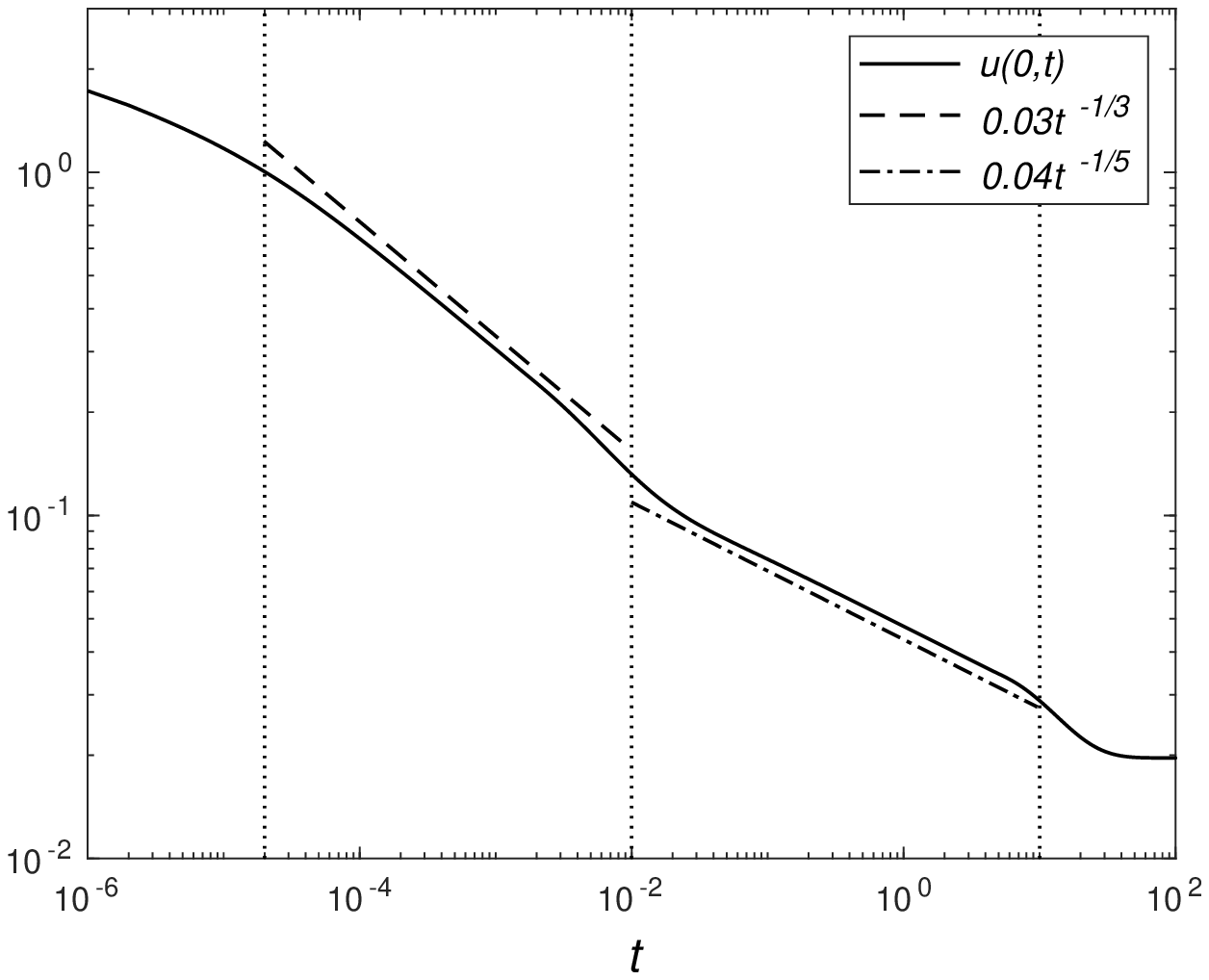}&\includegraphics[scale=0.38]{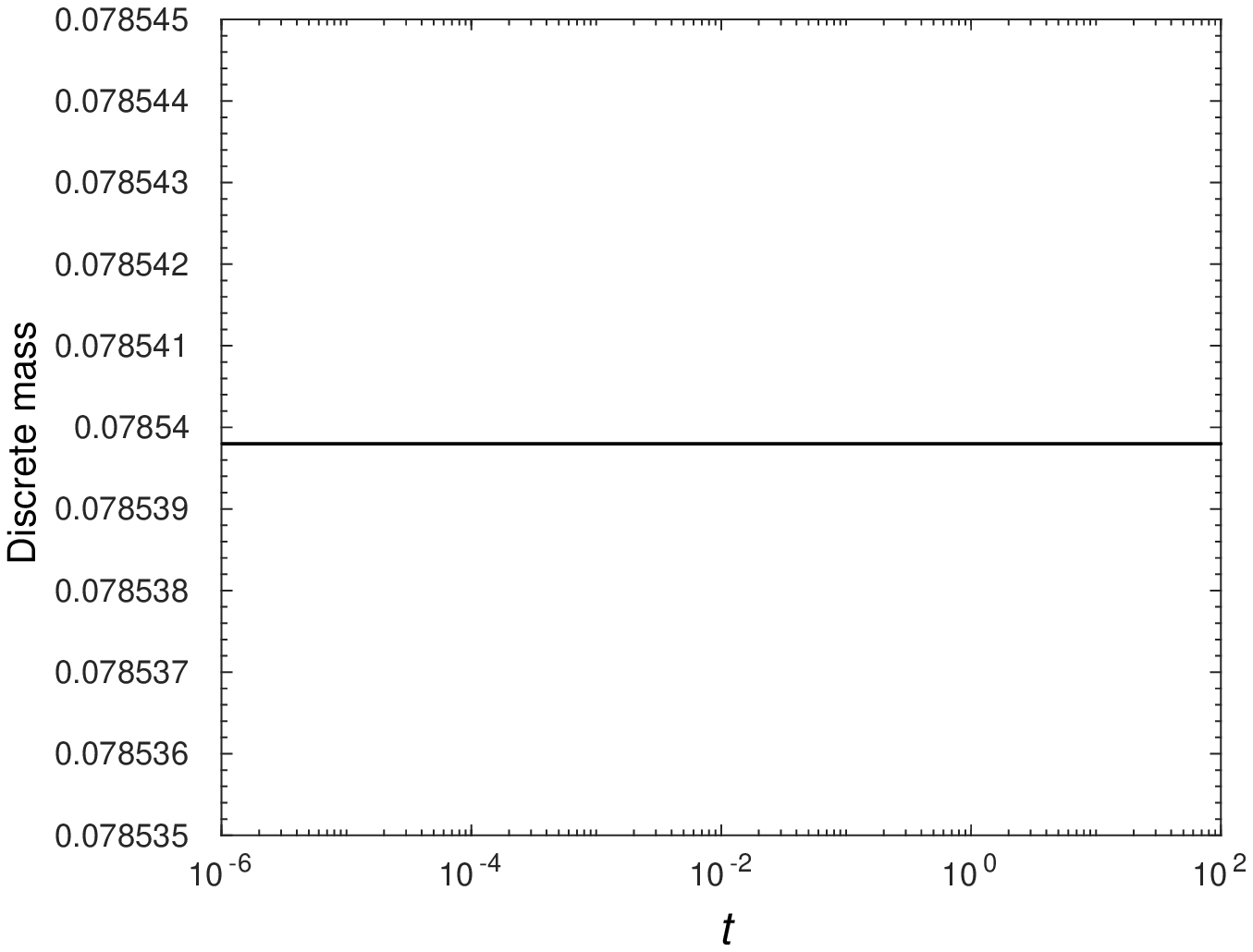}\\
$(a)$ & $(b)$ & $(c)$
\end{tabular}
\caption{\textit{Scheme2}: $(a)$ Dissipation of the total free energy over time. $(b)$ Time-evolution of the droplet maximal height showing the intermediate asymptotics. $(c)$ Time-evolution of the discrete total mass. $\Delta t=10^{-6}$.}
\label{Fig:FreeEnergyHeightAndMassCons}
\end{figure}

\subsubsection{No spreading of support}
The exponent $p$ plays a crucial role in the  qualitative behaviour of solutions to \eqref{ThinFilmEquation}-\eqref{MobFunc}. It is shown in \cite{Bereta1995} in the one-dimensional case that the support of a solution to \eqref{ThinFilmEquation}-\eqref{MobFunc}, where $f_0(u)\equiv1$ and $\gamma=1$, remains constant for values of $p\geqslant4$ as time evolves. We want to test the performance of \textit{scheme3} using this known theoretical result.  To do so, we further consider the one-dimensional test case presented in \cite{Grun2000} (see Figure~4). Here, we model the problem in the two-dimensional rectangular domain $[0,1]\times[0,0.3]$.  The initial condition is given by 
\begin{equation}
u(\bm{x},0)=
\left\lbrace
\begin{aligned}
&0.2(x-0.25), \text{ if } 0.25\leqslant x\leqslant 0.5,\\
&0.2(0.75-x), \text{ if } 0.5<x\leqslant 0.75,\\
&0,\qquad\qquad\quad \text{ otherwise}.
\end{aligned}
\right.
\label{InitSolNoSpreading}
\end{equation}
We use an unstructured mesh of size $h=0.014$ ($100\times 33$ elements). Homogeneous Neumann conditions  are imposed on the boundary. We compute the numerical solution for different values of $p\geqslant 4$. Figure~\ref{Fig:NoSpreadingOfSupport} shows the time-evolution of the numerical solution for $p=4$. Our experiments show that numerical solutions converge for values of $p\geqslant4$ and $t\rightarrow\infty$ to a parabolic profile, and no spreading of support occurs as predicted by the theoretical results in \cite{Bereta1995} and numerical simulations in \cite{Grun2000}. 
\begin{figure}[!h]
\centering
\begin{tabular}{cccc}
\includegraphics[scale=0.27]{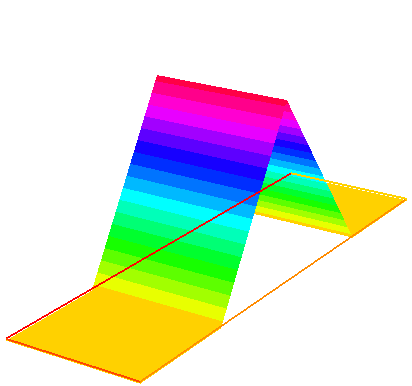}&\includegraphics[scale=0.27]{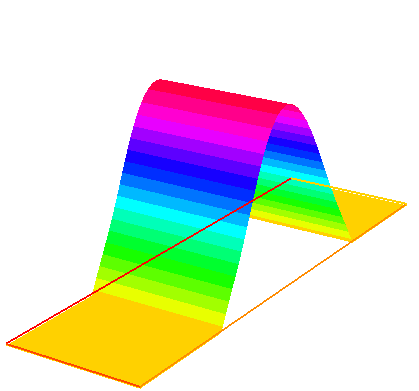}&\includegraphics[scale=0.27]{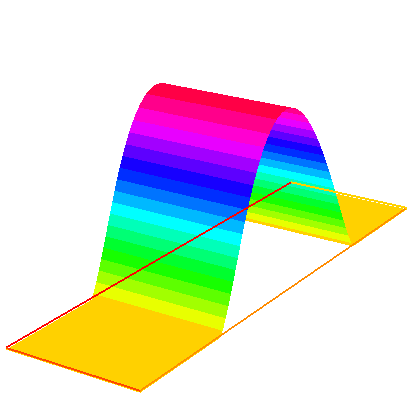}&\includegraphics[scale=0.27]{2DNoSpreading_t20.png}\\
\includegraphics[scale=0.27]{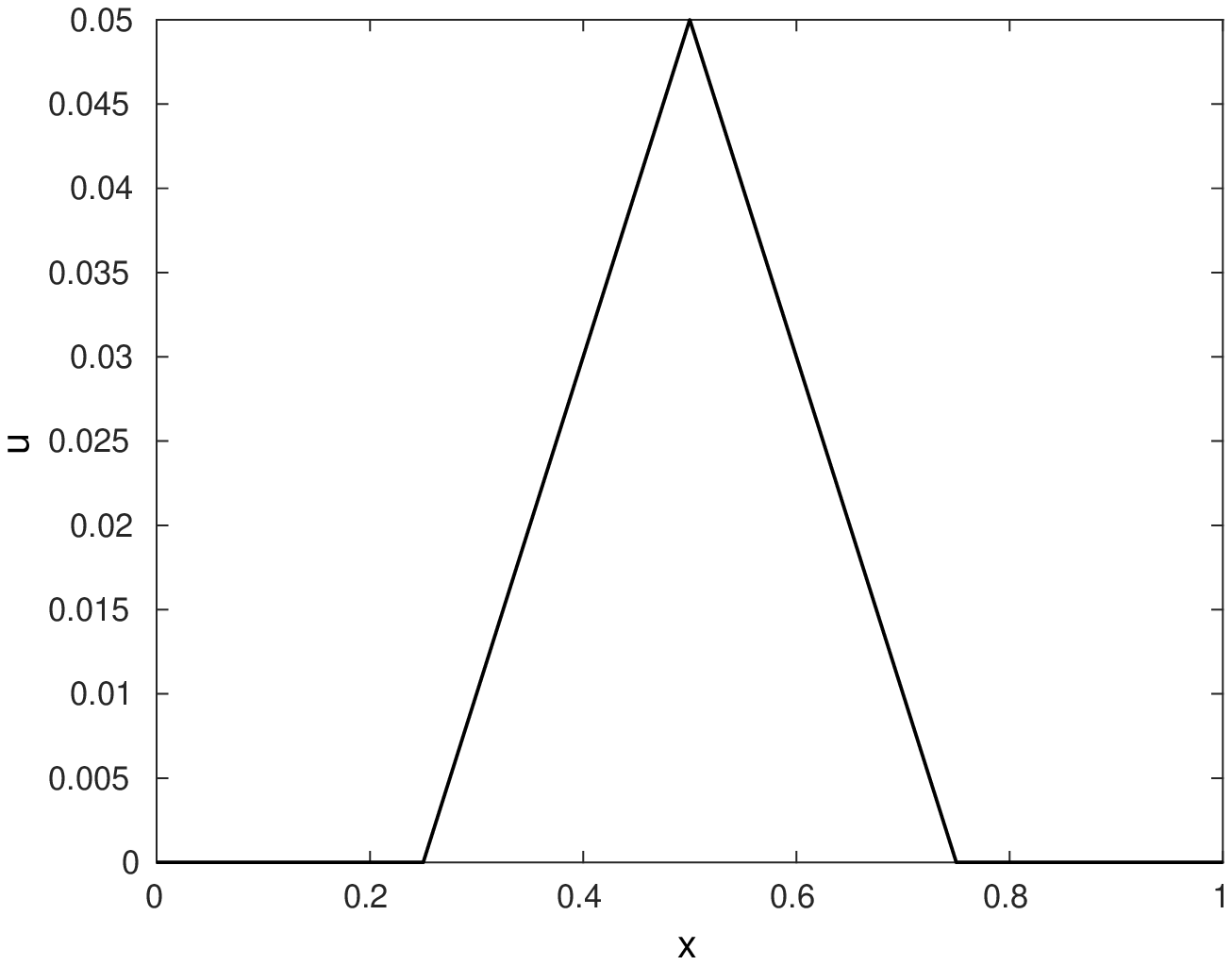}&\includegraphics[scale=0.27]{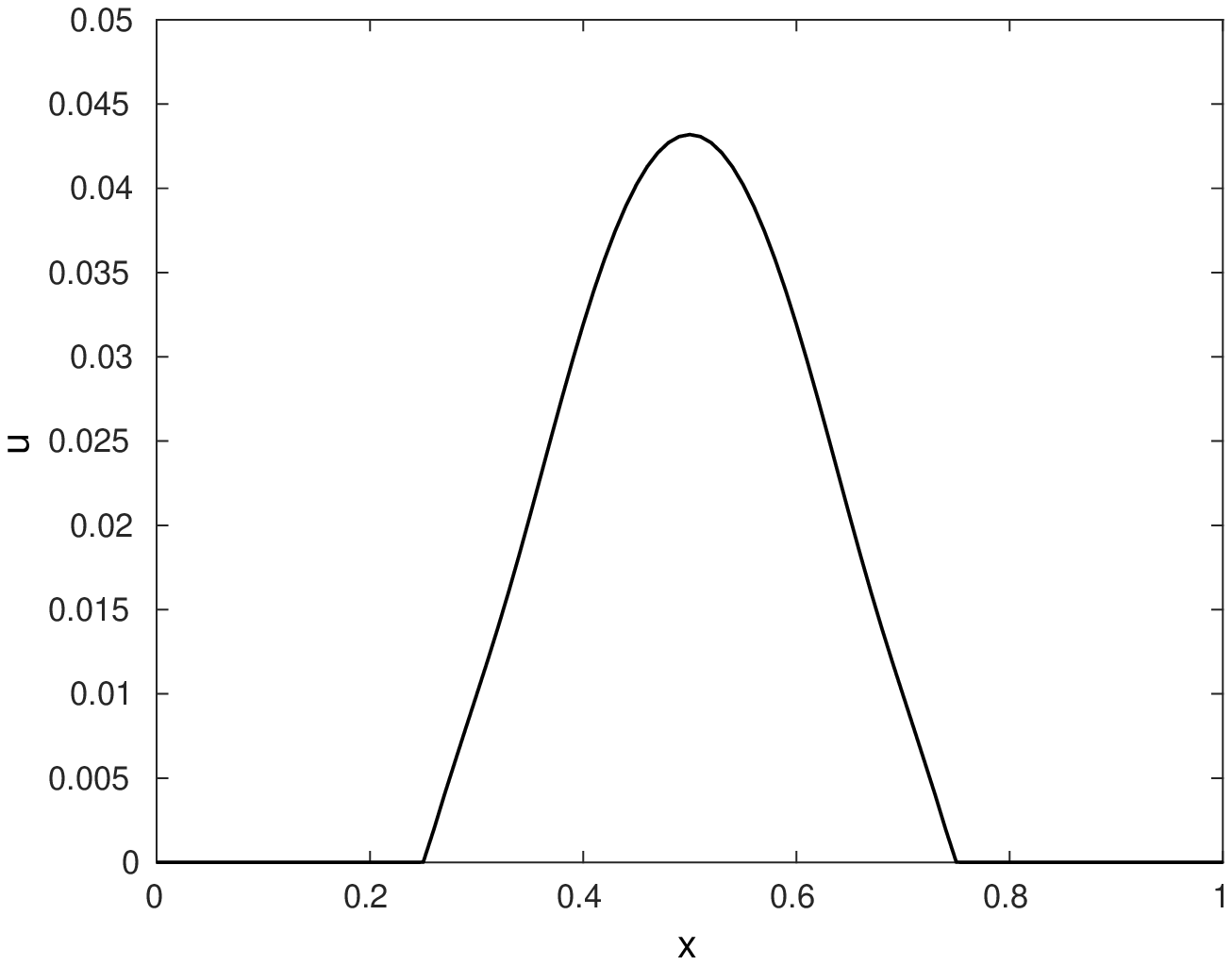}&\includegraphics[scale=0.27]{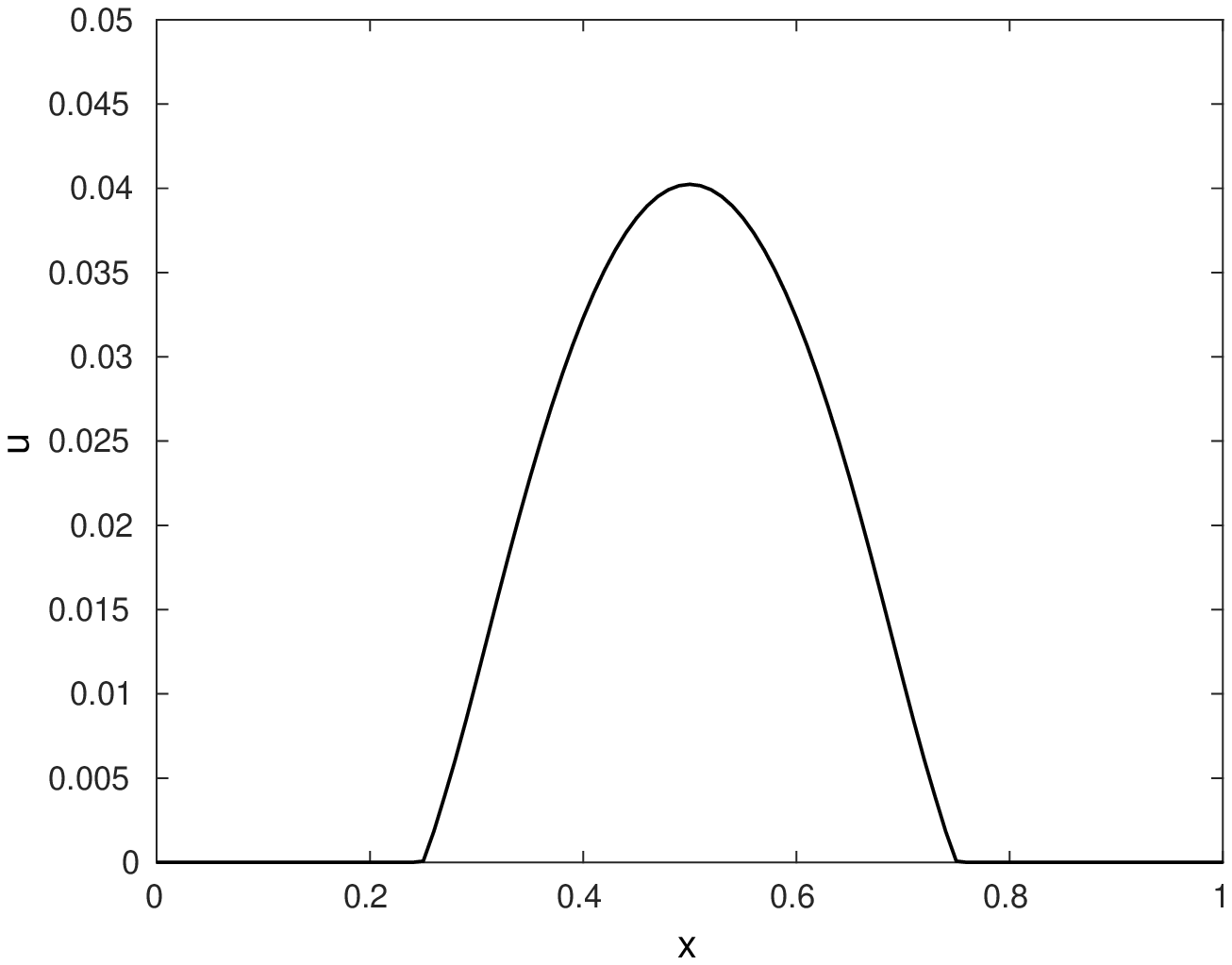}&\includegraphics[scale=0.27]{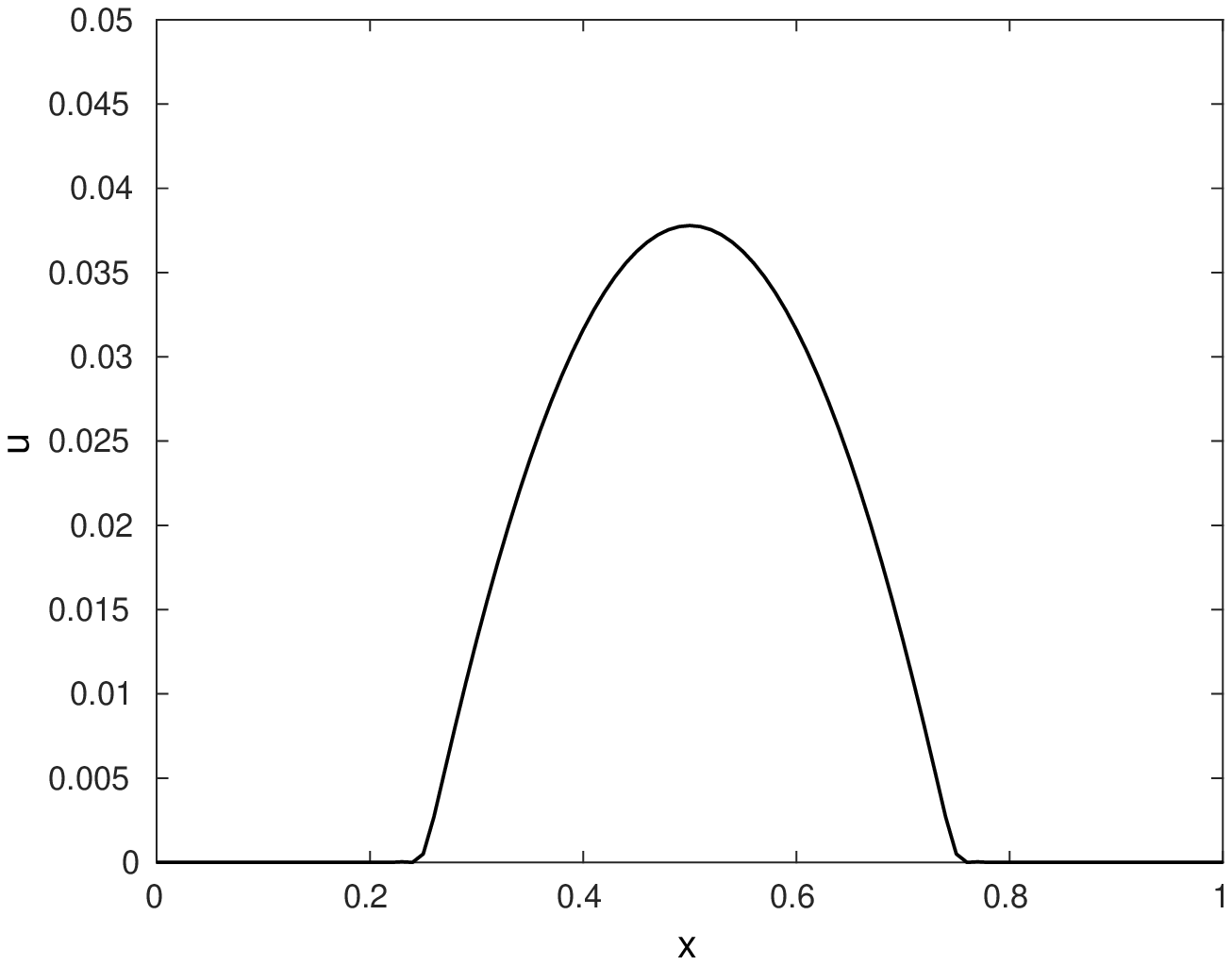}\\
$t=0$&$t=2$&$t=20$&$t=300$
\end{tabular}
\caption{No spreading of support and convergence to a parabolic profile as $t\rightarrow\infty$. Top: 3D view of the numerical solution. Bottom: Plot of the solution along the $x$-axis. $p=4$ and $\Delta t=0.01$.}
\label{Fig:NoSpreadingOfSupport}
\end{figure}

\subsubsection{Dead core phenomenon}
Theoretical studies only guarantee that the support of a solution to \eqref{ThinFilmEquation}-\eqref{MobFunc}, where $f_0(u)\equiv1$ and $\gamma=1$, cannot shrink for values of $p\geqslant \frac{3}{2}$ as time evolves \cite{Bereta1995}. It is known that solutions with initial data bounded away from zero may attain this value on a set of positive measure \cite{Bereta1995}, and this, for a finite interval of time and become again strictly positive \cite{Bertozzi1995}. Such phenomenon is reported in the literature as the so called ``dead core phenomenon" \cite{Grun2000}. A one-dimensional test case illustrating such phenomenon is presented in \cite{Grun2000} (see Figure 5). We further test the robustness of \textit{Scheme3} using this more critical test case in the two-dimensional rectangular domain $[0,1]\times[0,0.3]$. The initial data is given by
\begin{equation}
u(\bm{x},0)=(x-0.5)^4 + 0.001.
\label{InitSolDeadCore}
\end{equation} 
We use a mesh of size $h=0.0047$ ($300\times100$ elements) and a time step $\Delta t =10^{-5}$. Homogeneous Neumann boundary conditions are used. Figure~\ref{Fig:DeadCore} (Top) shows the plot along the $x$-axis of the numerical solution for $p=0.5$ at different times. We can see that a film rupture occurs at a critical time $t^*\approx 0.0005 $ and the solution becomes strictly positive beyond this time. This phenemenon has been described in \cite{Bertozzi1992}. We repeat again the test case, but for $p=2$ and $\Delta t=0.01$. Numerical results are shown in Figure \ref{Fig:DeadCore} (Bottom). The solution remains strictly positive, as predicted by the theory. No dead core phenomenon occurs.
\begin{figure}[!h]
\centering
\begin{tabular}{cccc}
\includegraphics[scale=0.27]{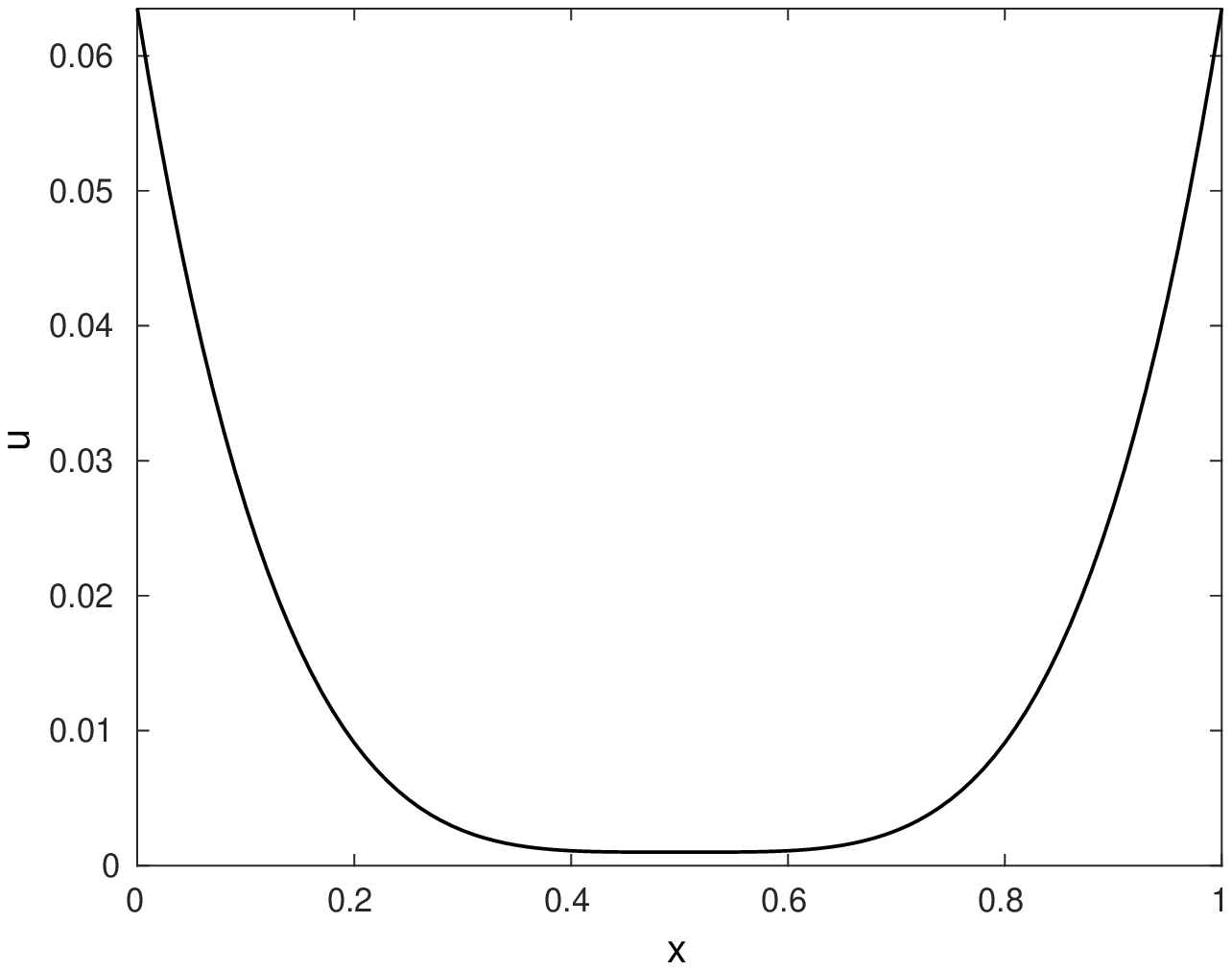}&\includegraphics[scale=0.27]{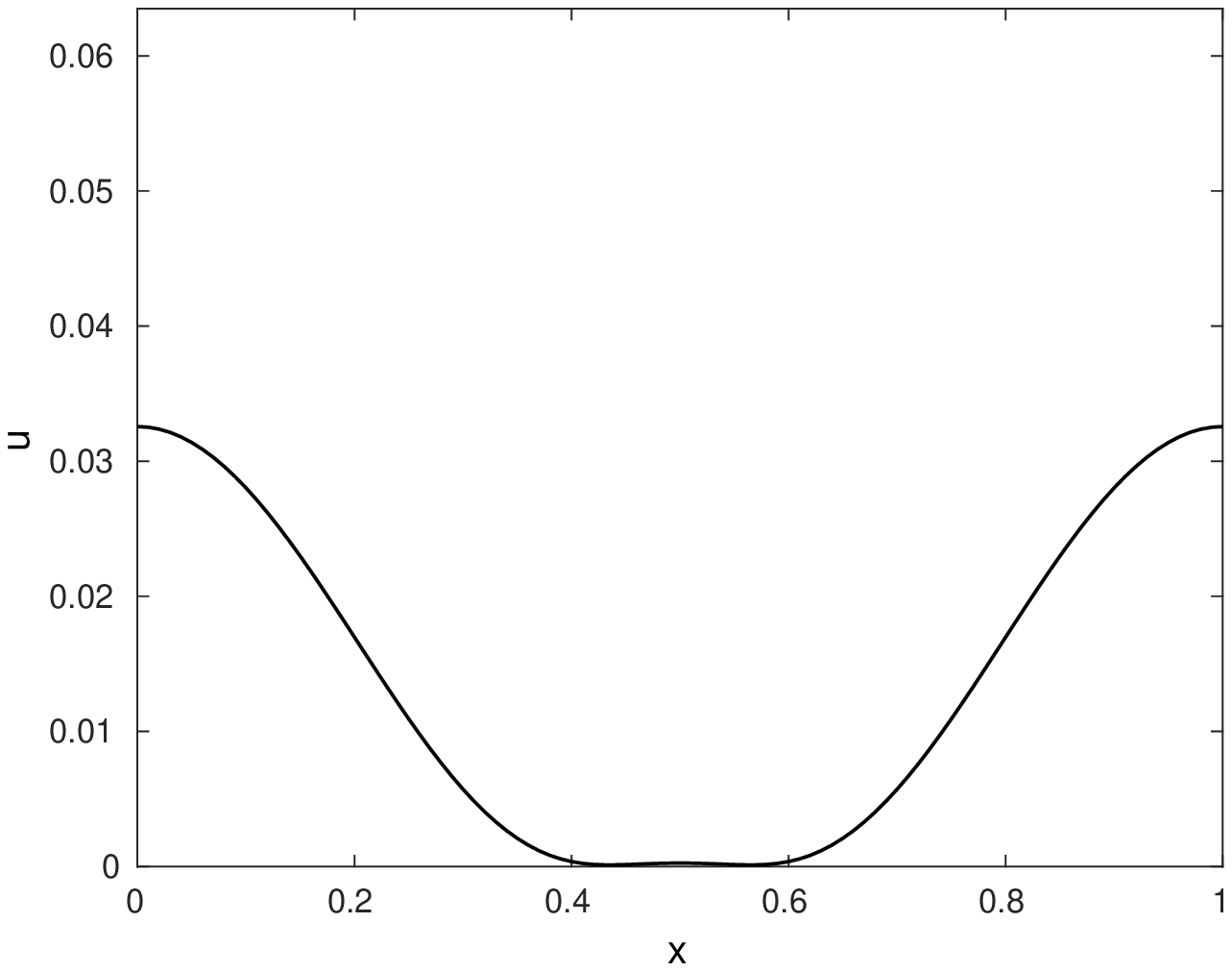}&\includegraphics[scale=0.27]{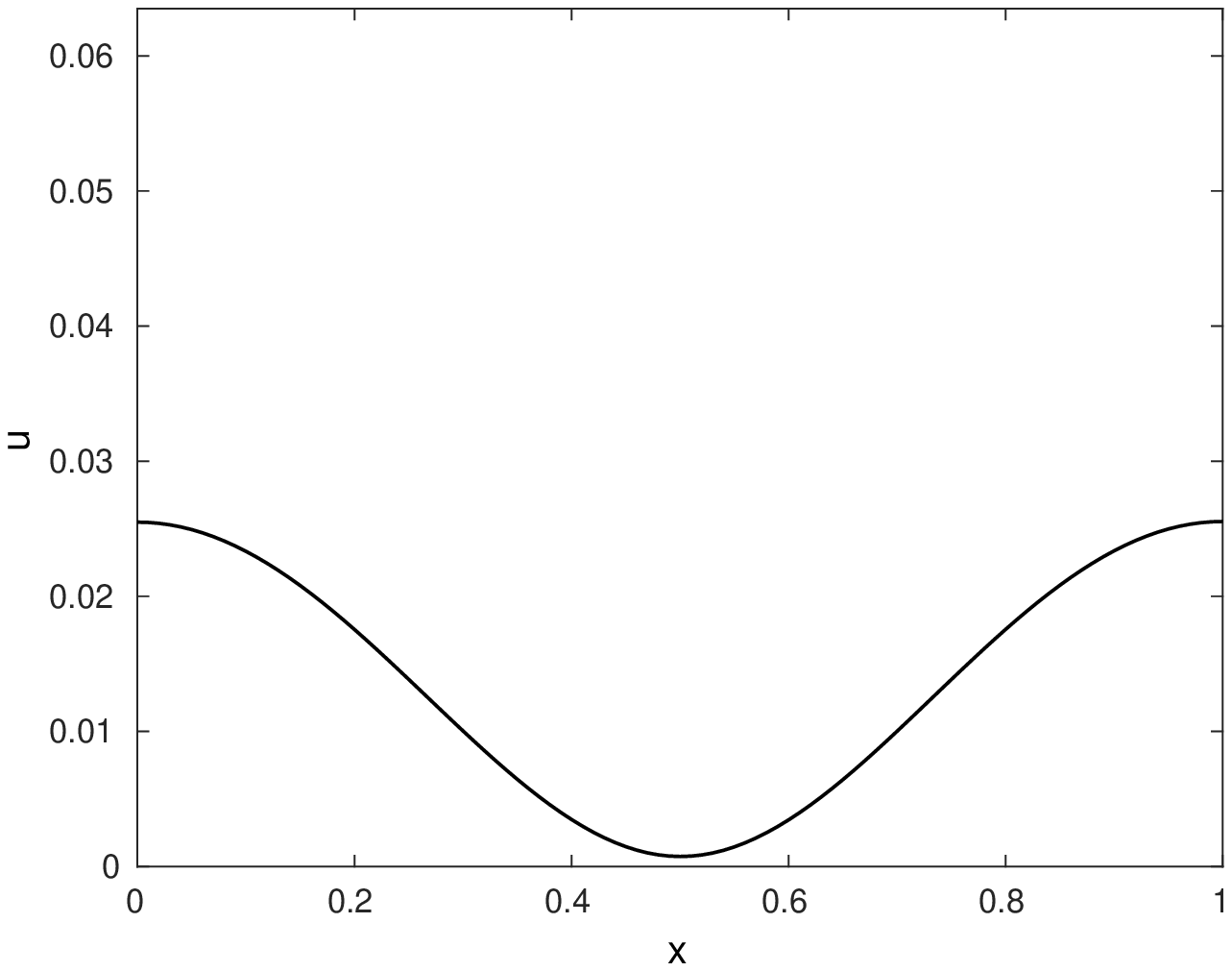}&\includegraphics[scale=0.27]{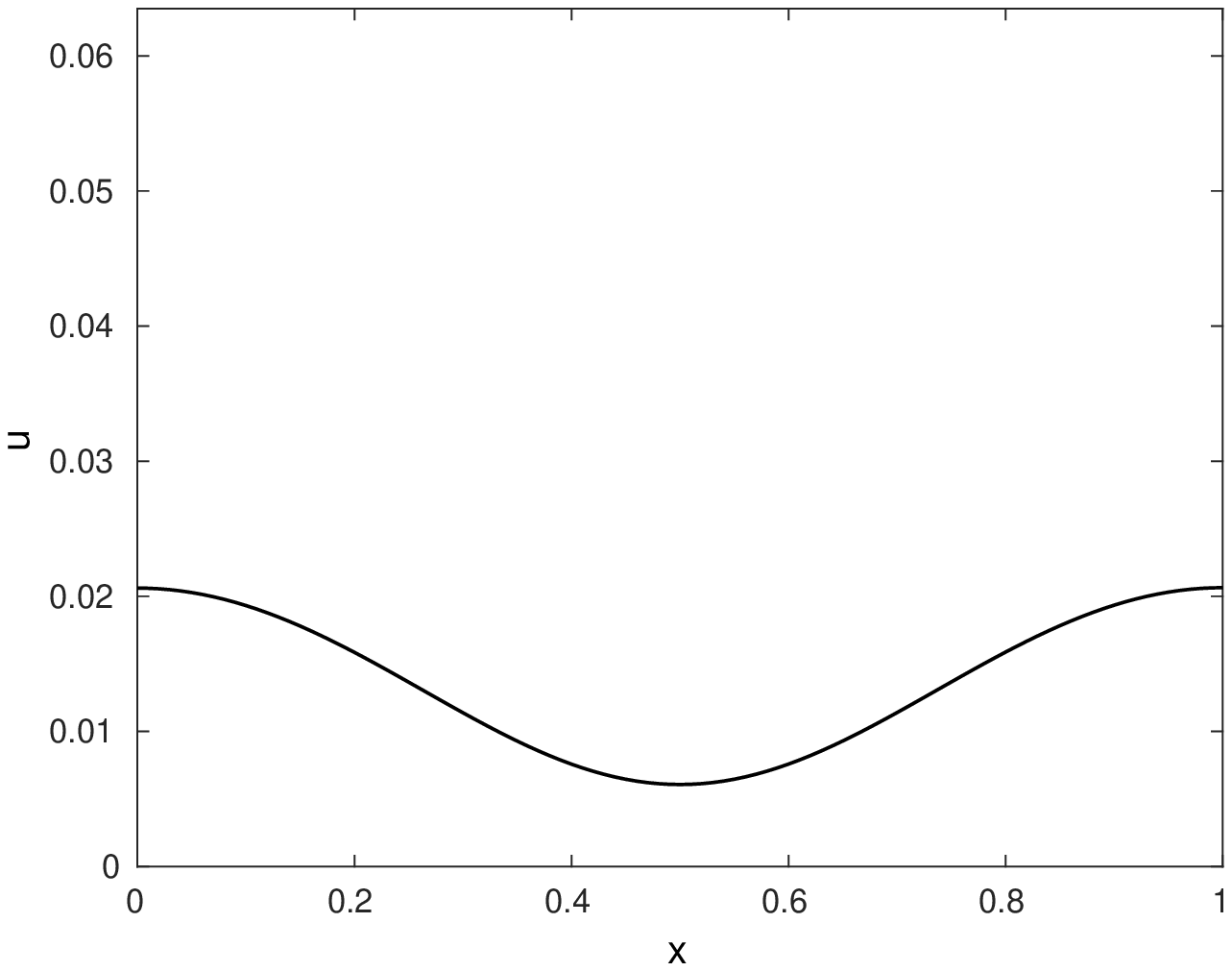}\\
$t=0$&$t=0.0005$&$t=0.0019$&$t=0.005$\\
\includegraphics[scale=0.27]{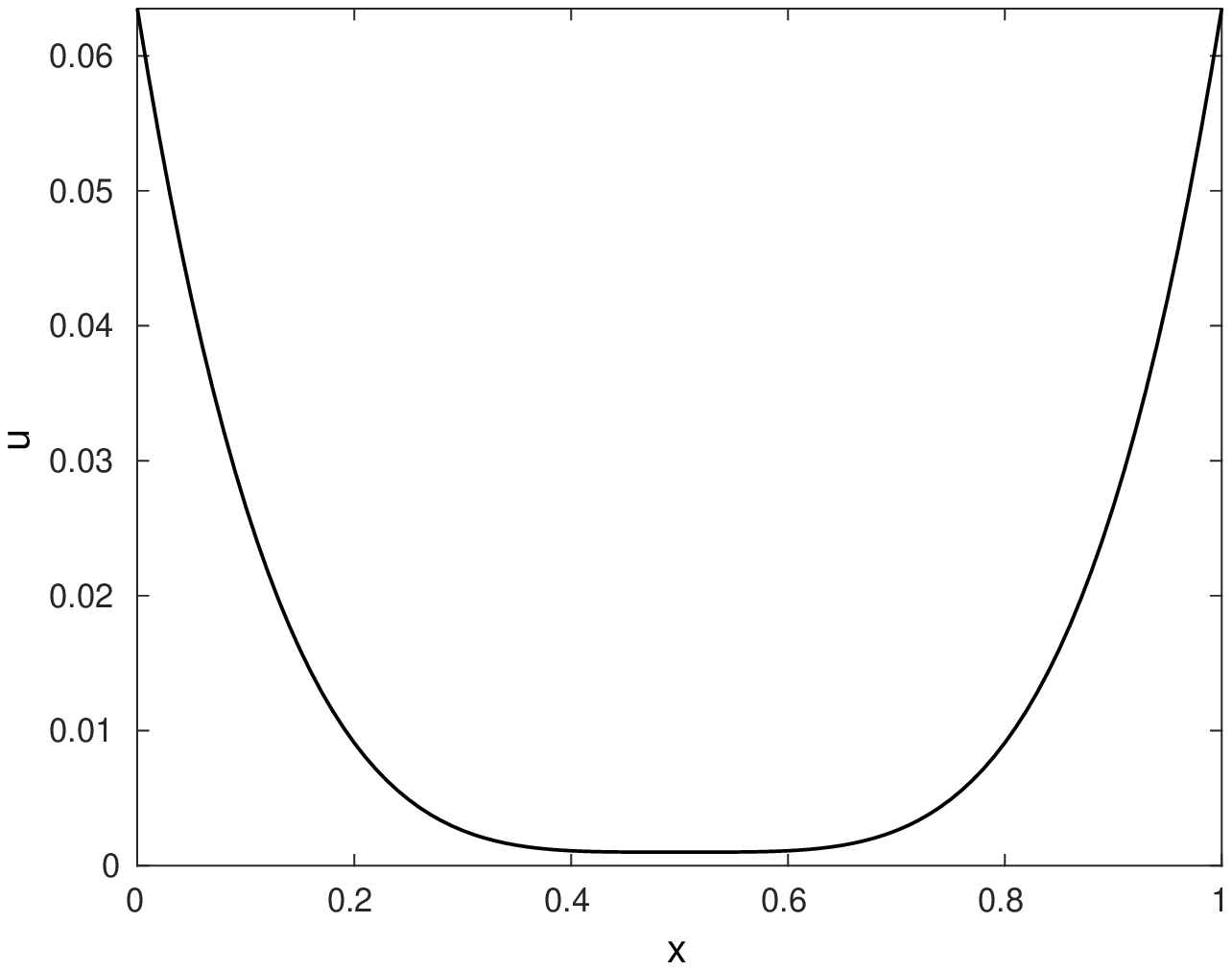}&\includegraphics[scale=0.27]{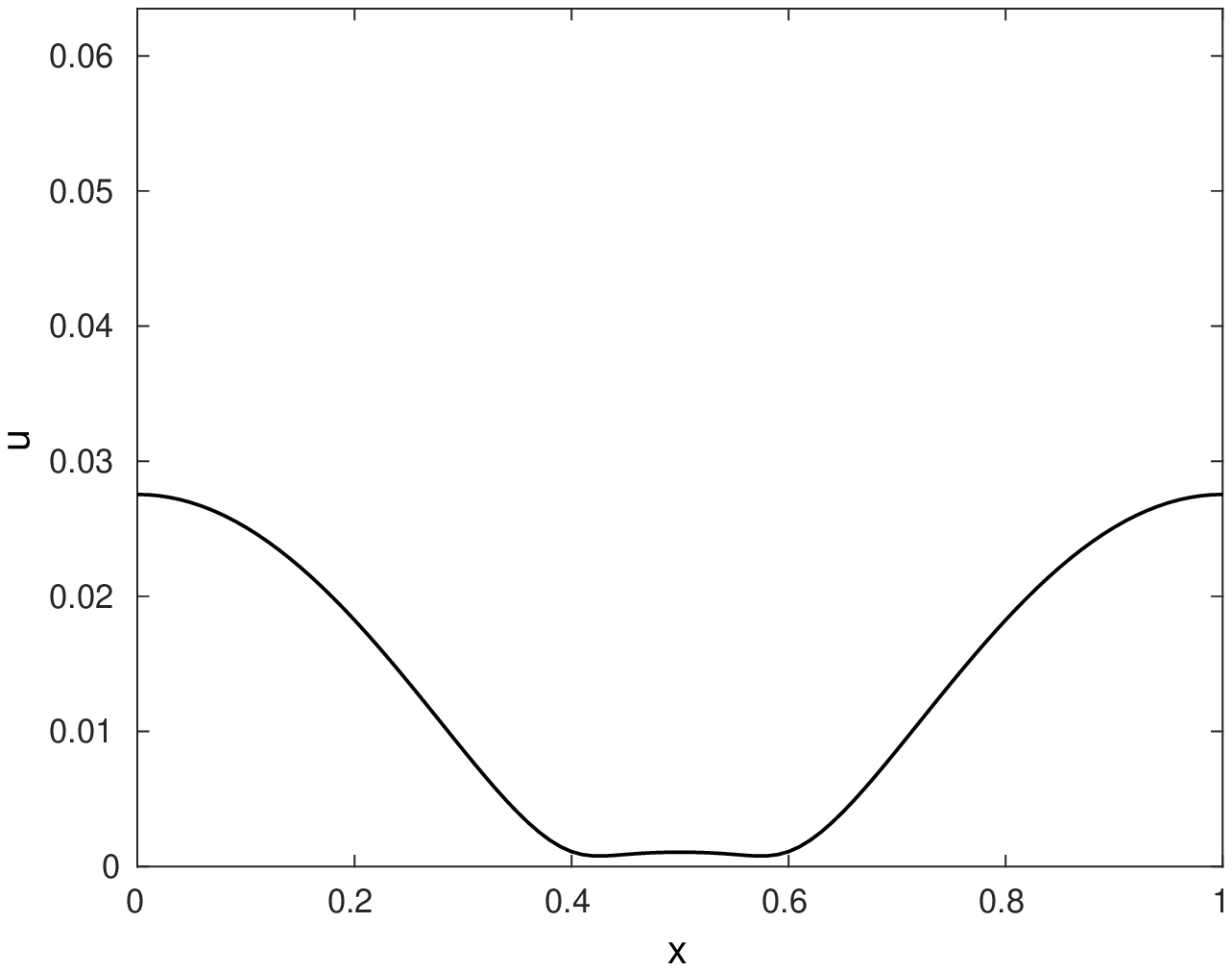}&\includegraphics[scale=0.27]{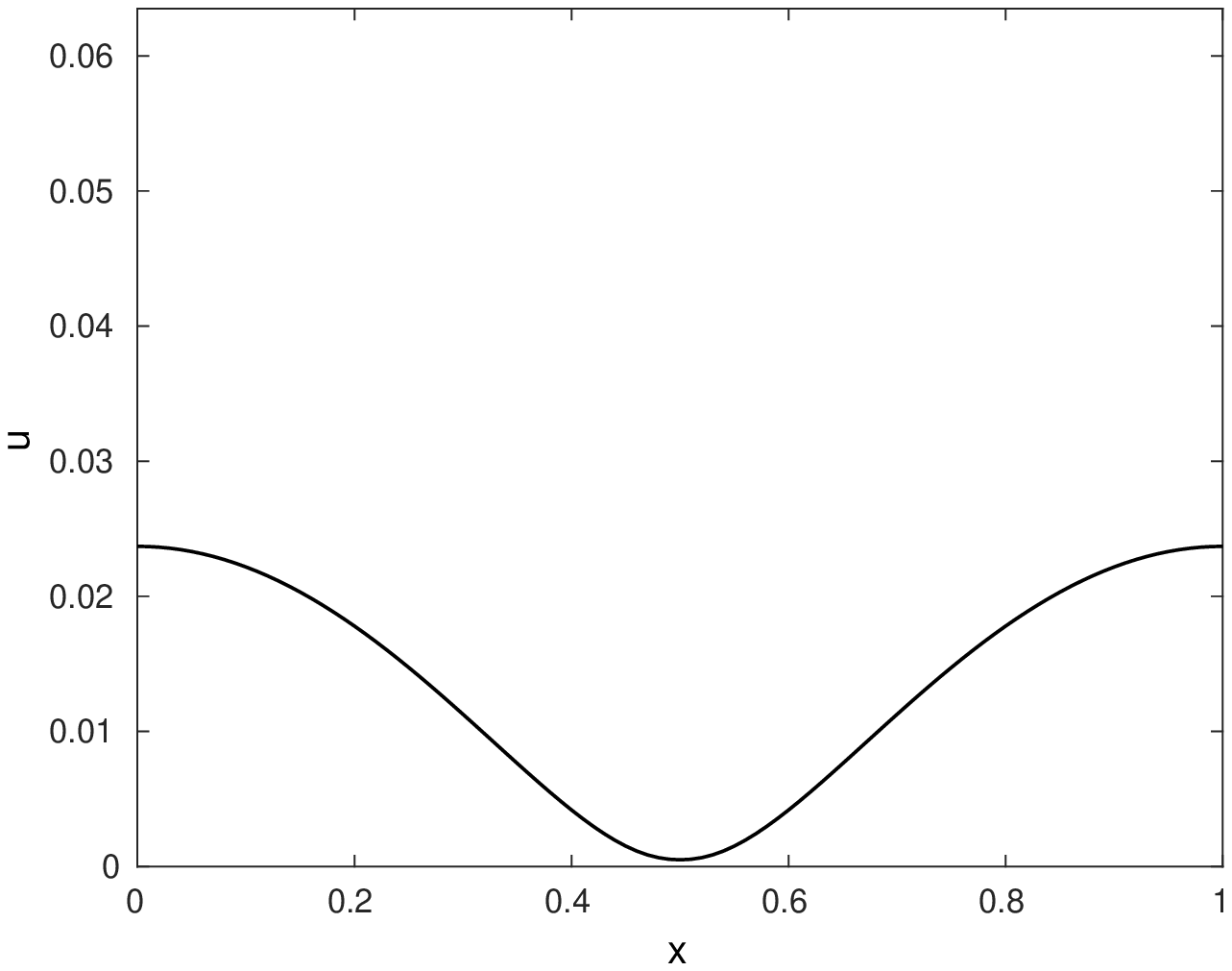}&\includegraphics[scale=0.27]{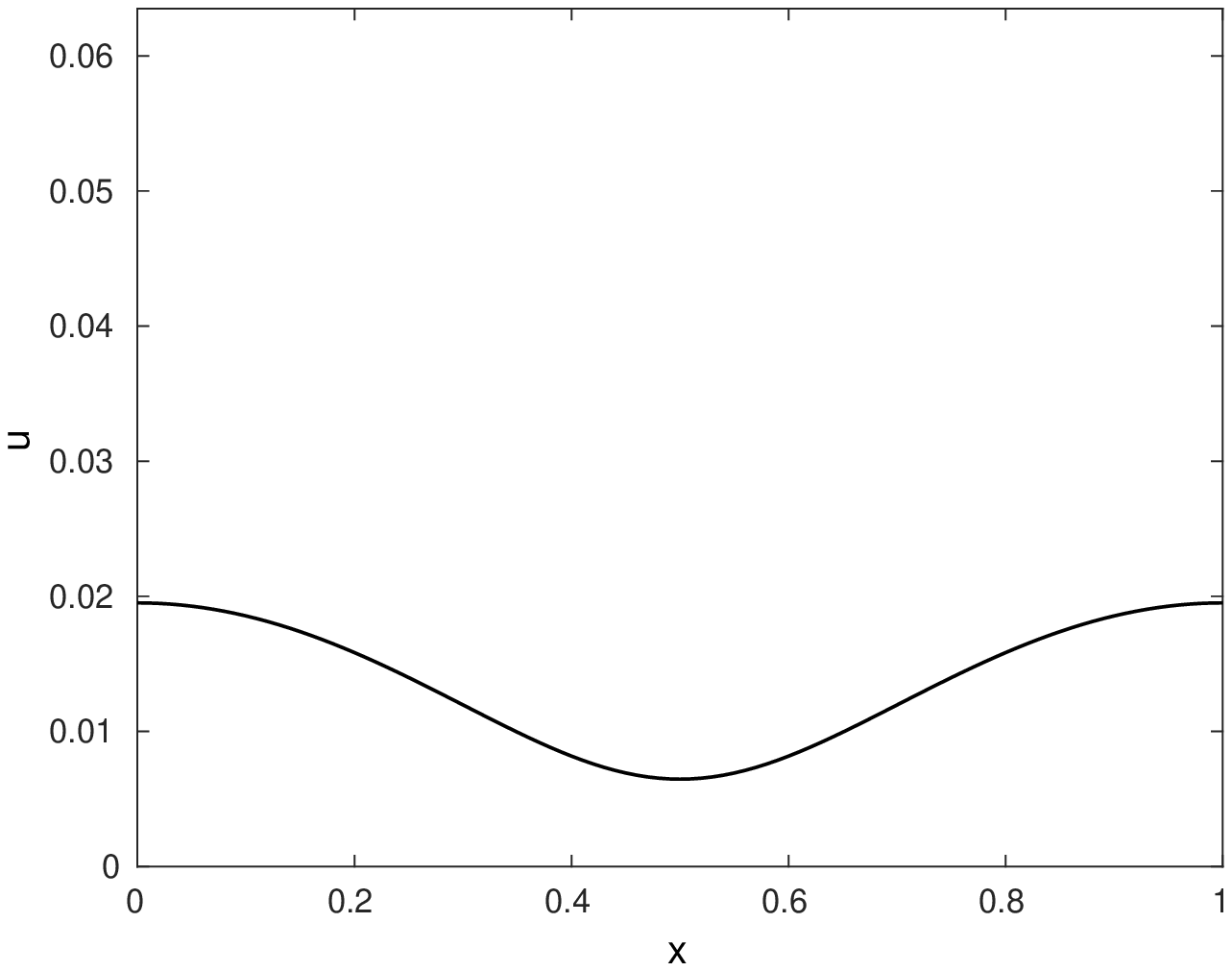}\\
$t=0$&$t=0.6$&$t=1.7$&$t=5$
\end{tabular}
\caption{Top: Illustration of a dead core phenomenon using the initial data \eqref{InitSolDeadCore} and $p=0.5$. Bottom: No dead core phenomenon occurs for $p=2$.}
\label{Fig:DeadCore}
\end{figure}

\subsection{Phase separation with a logarithmic free energy function}
Recall that phase separation in a binary mixture is well-described by \eqref{GenThinFilmEquation} which can be written as 
\begin{equation}
\begin{aligned}
&\partial_tu-\nabla\cdot\big(f(u)\nabla w\big)-\nabla\cdot\big(g(u)\nabla u\big)=0 \quad \text{in } \Omega_T,\\
&w=-\gamma\Delta u \quad \text{in } \Omega_T.
\end{aligned}
\label{ThinFilmSystMod2}
\end{equation} 
In such applications, $u\in[-1,1]$ represents a relative concentration of the binary mixture. Here, we consider the mobility function \eqref{BoundedMob} and the logarithmic free energy \eqref{FreeEnergyLog}-\eqref{F0function}. In this case, $g(u)$ is given by (see \eqref{GfunctionExpl})
\begin{equation}
g(u)=\theta-\theta_c(1-u^2).
\label{GfunctionFinal}
\end{equation}
As opposed to \eqref{ThinFilmSyst}, the formulation \eqref{ThinFilmSystMod2} has no singularity at the limit $\vert u\vert\rightarrow1$ and allows the concentration solution $u$ to reach the bounds $\pm1$ (see discussion in subsection \ref{SubsectionLogFreeEnergy}). Recall also that a solution of \eqref{ThinFilmSystMod2}-\eqref{GfunctionFinal} should satisfy $\vert u(\bm{x},t)\vert\leqslant1$ for all $t>0$ if $\vert u_0(\bm{x})\vert\leqslant1$ \cite{Elliott96}.

We want to use \textit{Scheme1} to solve \eqref{ThinFilmSystMod2}-\eqref{GfunctionFinal}. For that,  a suitable temporal discretization of the nonlinear terms $f(u)$ and $g(u)$ is required for the approximation of the variational problem. As we did in \eqref{VariationalEqualityScheme}, one uses an extrapolation formula of second order for $f(u)$ and $g(u)$ to preserve the linearity and a second order time-stepping method. This leads to the following scheme: Given a suitable approximation of the initial solution $u_h^{0}=\Pi_hu_0\in \mathcal{S}_h^a$ and a proper initialization for $u_h^1\in \mathcal{S}_h^a$, find $\big(\widehat{u}_h^{n+1},w_h^{n+1}\big)\in \mathcal{V}_h \times \mathcal{V}_h$ such that
\begin{eqnarray}
\begin{aligned}
&\int_{\Omega}\left(\dfrac{3\widehat{u}_h^{n+1}-4u_h^{n}+u_h^{n-1}}{2\Delta t}\right)v_h\,d\bm{x}+\int_{\Omega}[2f(u_h^{n})-f(u_h^{n-1})]\nabla w_h^{n+1}\cdot\nabla v_h\,d\bm{x}\\
&\qquad\qquad\qquad\qquad\qquad\qquad\quad +\int_{\Omega}[2g(u_h^{n})-g(u_h^{n-1})]\nabla\widehat{u}_h^{n+1}\cdot\nabla v_h\,d\bm{x}=0,\quad\forall v_h \in {\mathcal{V}_h},\\
&\int_{\Omega} w_h^{n+1}q_h \,d\bm{x}-\int_{\Omega}\gamma\nabla \widehat{u}_h^{n+1}\cdot\nabla q_h\,d\bm{x}=0,\quad\forall q_h \in {\mathcal{V}_h}.
\end{aligned}
\label{VariationalEqualitySchemeSystMod}
\end{eqnarray}
Then, the problem \textbf{(P)} is modified to take into account the boundedness in magnitude by $1$ of the solution. There is no difficulty to show the existence and uniqueness of solution to the modified problem. Finally, the Uzawa algorithm \eqref{lambdak}-\eqref{uk} is then adapted to the modified optimization problem.

A particular regime of phase separation is when one phase is significantly more abundant in the initial mixture. In such case, the minority phase emerges and forms droplets. Larger particles grow at the expense of the smaller ones which shrink and disappear. This form of competitive growth is known as the Ostwald ripening \cite{Pismen2006}. To illustrate such regime, we choose an initial data  as
\begin{equation}
u_0(\bm{x})=b+\sum_{i=1}^Q s_i\omega(\bm{x}-\bm{p}_i),
\label{InitialDataRipening}
\end{equation}
where the constant $b$ belongs to the spinodal interval $\big(-\sqrt{1-\theta/\theta_c},\sqrt{1-\theta/\theta_c}\big)$, the points $\bm{p}_i\in\Omega$, the  $s_i$ are constant scalars, the function $\omega(\bm{x}-\bm{p}_i):=e^{-10^4\vert\bm{x}-\bm{p}_i\vert^2}$ and the integer $Q$ is the number of points. We choose a computational domain $\Omega=[0,1]\times[0,1]$ with an uniform unstructured with elements of average size $h=0.007$. We set $\theta=0.05$, $\theta_c=0.1$, $b=-0.4$ and randomly generate $Q=250$ points and scalars such that the initial data \eqref{InitialDataRipening} satisfies $|u_0(\bm{x})|\leqslant 1$.

The time-evolution of the numerical solution  is displayed in Figure~\ref{Fig:OstwaldTestCase}. From its initial concentration, the system separates quickly into two phases. Small regions of the minority phase combine together to form droplets. Smaller droplets are absorbed into larger ones through diffusion, exhibiting clearly the phenomenon of Ostwald ripening. The computed solution $u_h^{n}$ satisfies $|u_h^n|\leqslant 1$ for all time $t_n \geqslant 0$ with $u_h^n=\pm1$ in some regions. We also notice in practice that for $\theta>0$ big enough, $\widehat{u}_h^n$ computed from \eqref{VariationalEqualitySchemeSystMod} satisfies $|\widehat{u}_h^n|\leqslant 1$. In this case, no projection is needed.  However, this property on $\widehat{u}_h^n$ could be lost as $\theta \rightarrow 0$, but solving the additional optimization problem enforces the solution $u_h^{n}$ to remains in the interval $[-1,1]$ and allows to preserve the mass. The discrete total free energy of \eqref{GenThinFilmEquation} is given by
\begin{equation}
J(u_h^n)=\int_{\Omega}\Big(\frac{\gamma}{2}\vert\nabla u_h^n\vert^2+\varphi(u_h^n)\Big)\,d\bm{x}.
\label{LogDiscreteEnergy}
\end{equation}
A plot of the discrete mass and a log-log scale plot of the total energy \eqref{LogDiscreteEnergy} over time are shown in Figure~\ref{Fig:OstwaldMassCons}, which illustrate that accurate mass-conservation is obtained and the discrete total free energy of the system decreases monotonically.
\begin{figure}[!h]
\centering
\begin{tabular}{cccc}
\includegraphics[scale=0.19]{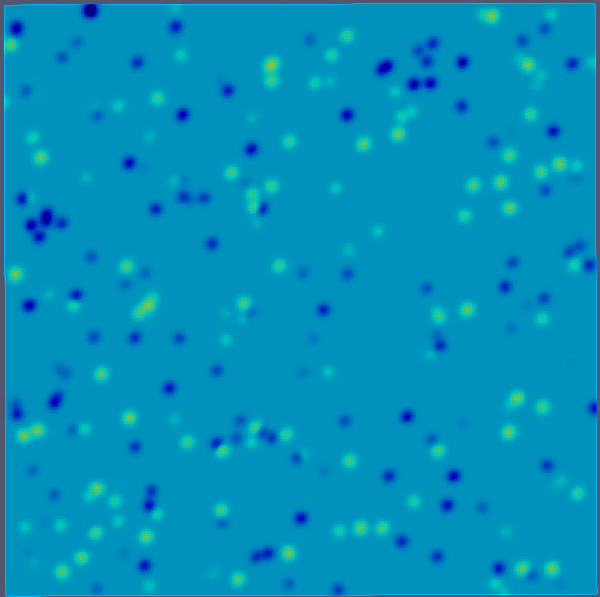}&\includegraphics[scale=0.19]{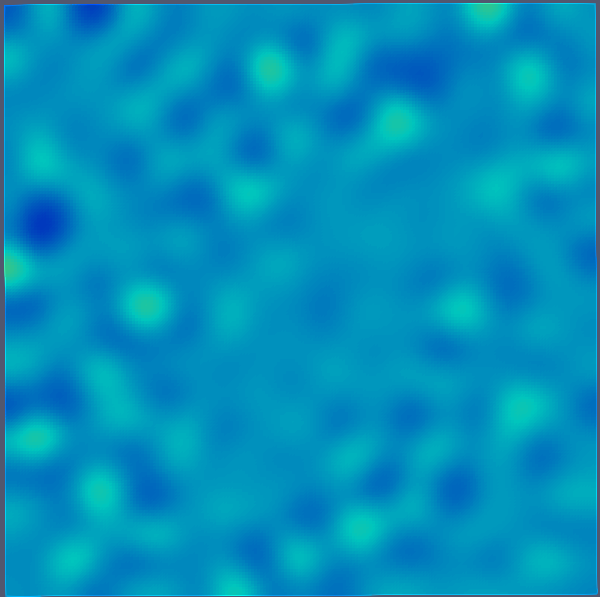}&\includegraphics[scale=0.19]{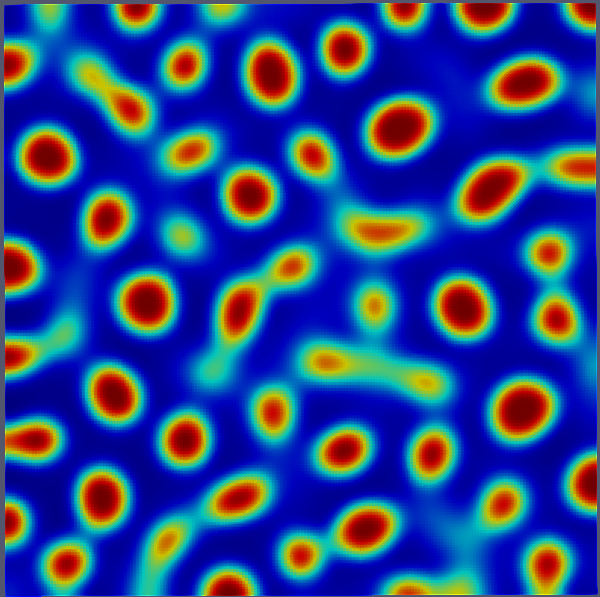}&\includegraphics[scale=0.19]{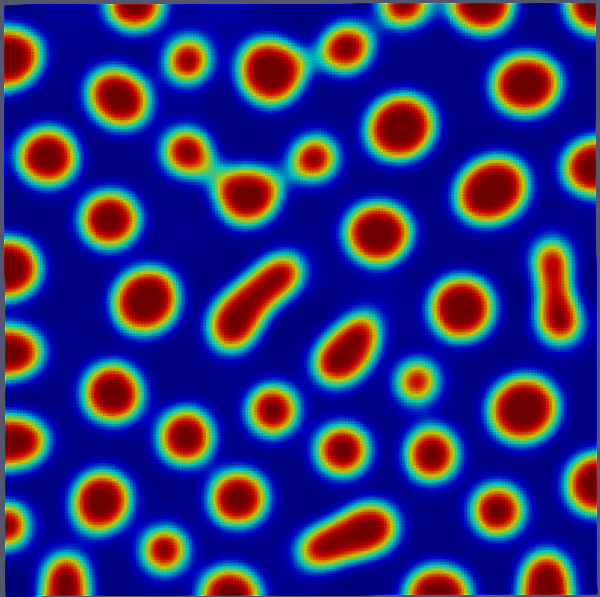}\\
$t=0$&$t=0.025$&$t=0.125$&$t=0.25$\\
\includegraphics[scale=0.19]{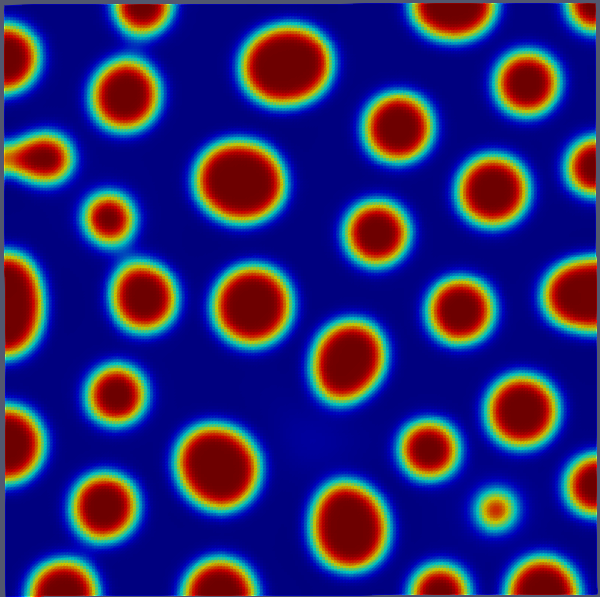}&\includegraphics[scale=0.19]{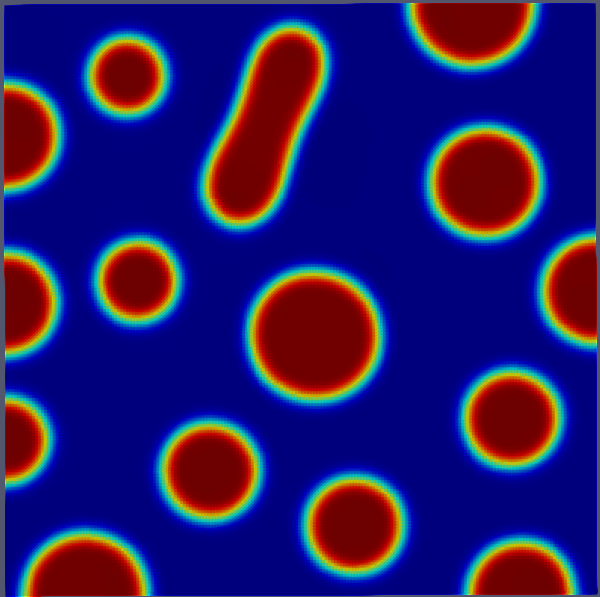}&\includegraphics[scale=0.19]{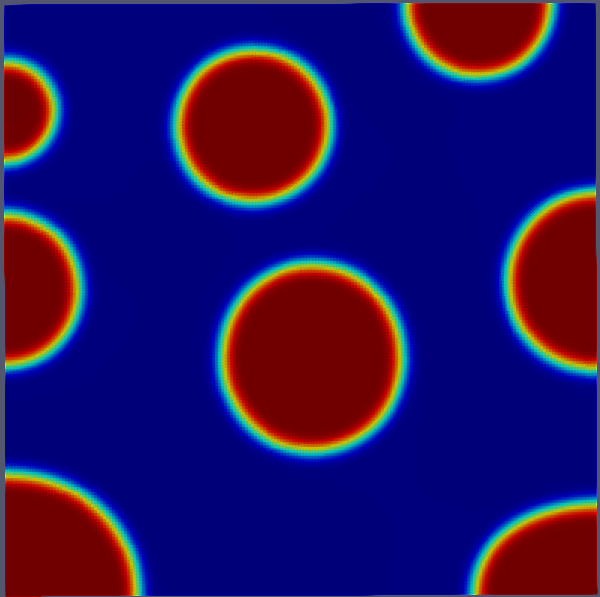}&\includegraphics[scale=0.19]{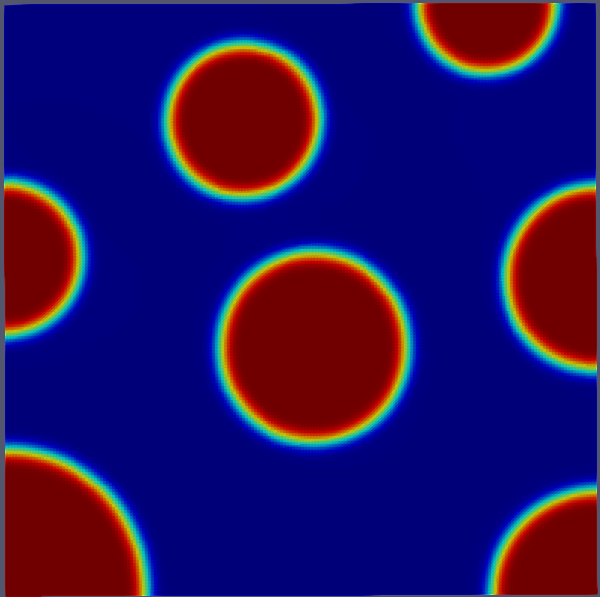}\\
$t=1$&$t=10$&$t=50$&$t=100$
\end{tabular}
\caption{Simulation of a Ostwald ripening process using a logarithmic energy function. $\gamma=10^{-5}$ and $\Delta t=0.001$.}
\label{Fig:OstwaldTestCase}
\end{figure}

\begin{figure}[!h]
\centering
\begin{tabular}{cc}
\includegraphics[scale=0.5]{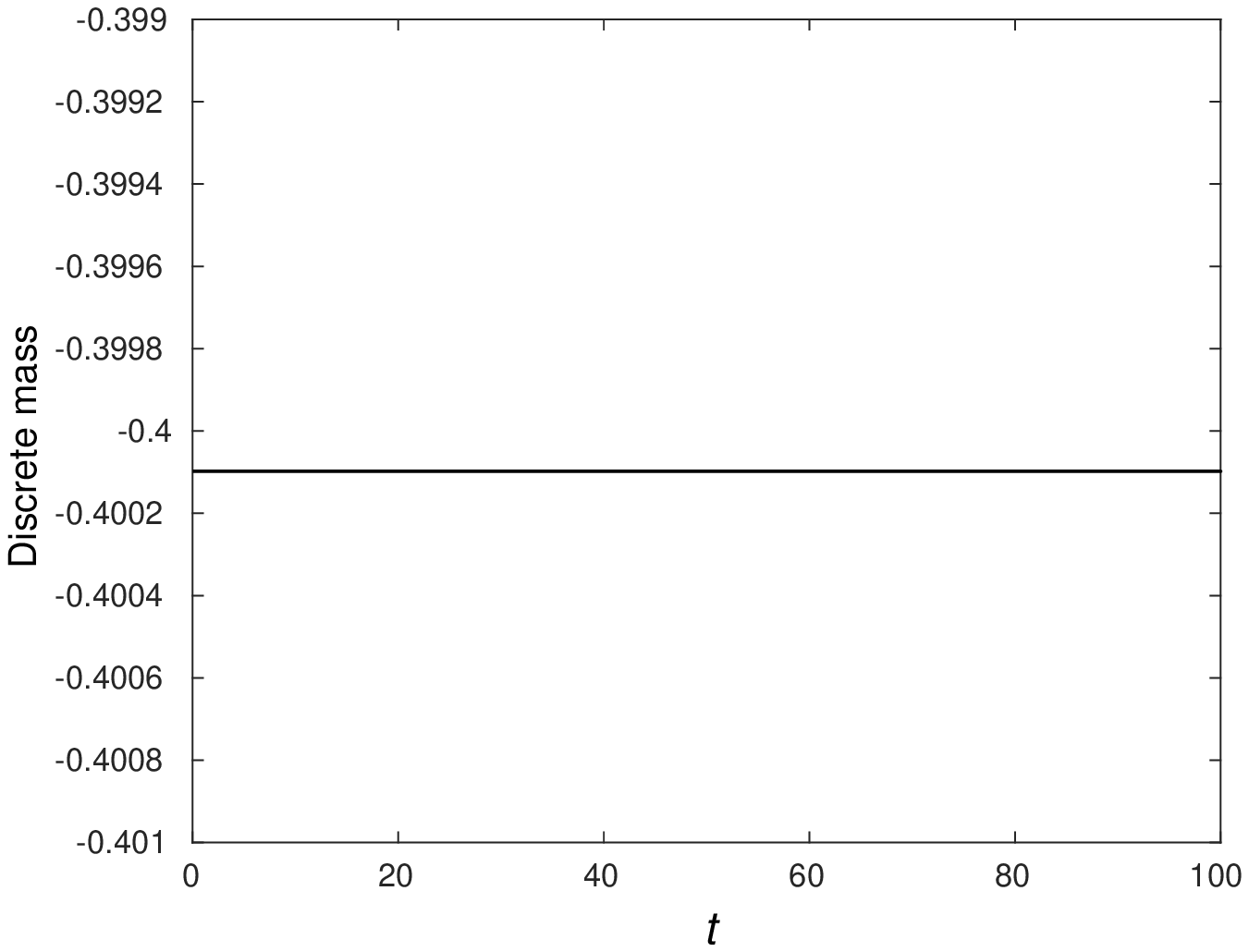}&\includegraphics[scale=0.5]{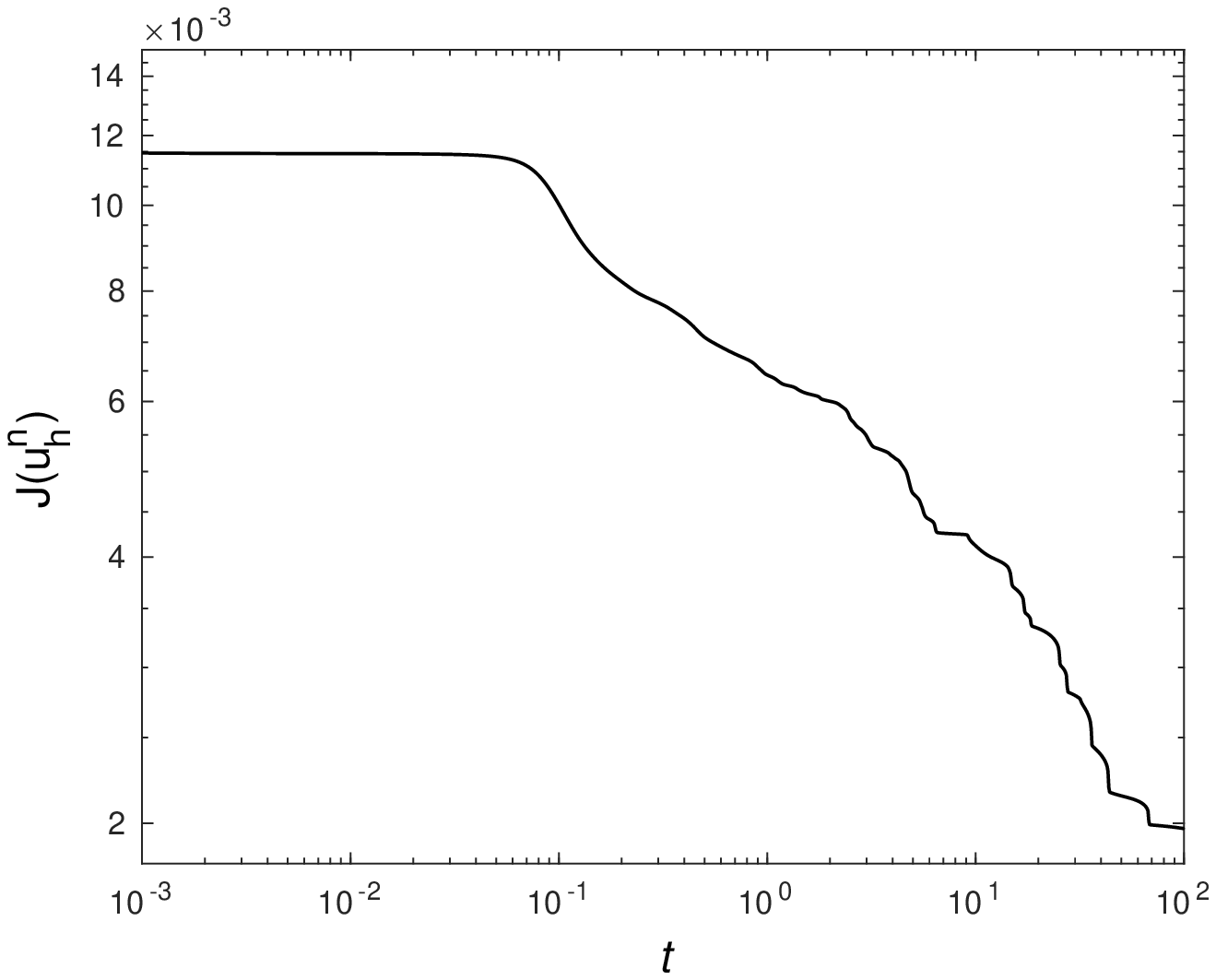}\\
$(a)$& $(b)$
\end{tabular}
\caption{ $(a)$ Time-evolution of the total mass during the Ostwald ripening process. $(b)$ Total free energy of the system over time.}
\label{Fig:OstwaldMassCons}
\end{figure}

\subsection{Discussion on the methods}
While \textit{Scheme4} appears to be more accurate for non regular solutions, the iterative procedure does not always converge, and when convergence occurs, it takes a large number of iterations. This is a weakness of \textit{Scheme4} in terms of efficiency. In terms of computational cost, \textit{Scheme 2} is the winner since no iterative procedure is needed. However, there is a defect of conservation of mass with this numerical method. \textit{Scheme3} is as accurate as the other schemes for regular solutions, preserves all the desired properties of the solutions,  and requires less computational effort than \textit{Scheme1} and \textit{Scheme4}. Only 2-3 iterations are required for the convergence of \textit{Scheme3}. This makes this scheme very competitive in terms of computational cost with respect to the schemes using Newton iteration, such as the schemes based on entropy \cite{Grun2000,Bertozzi3}, which require a large number of iterations for convergence \cite{Bertozzi1}. Moreover, the solution computed from \textit{Scheme3} is always positive, as opposed to \textit{Scheme1} and \textit{Scheme4}, for which the positivity of the computed solution is only guaranteed at the limit $\epsilon\rightarrow0$, a condition that one cannot reach in practice. Hence, \textit{Scheme3} appears to be the most competitive in terms of accuracy and efficiency.

\section{Conclusion}
\label{SectConclusion}
The present study provides several mixed finite element methods to approximate high-order nonlinear diffusion equations. We propose new techniques that allow the computed solution to satisfy all desired physical properties such as mass-conservation, positivity or boundedness of the solution  and dissipation of energy, by solving an optimization problem after getting a first approximate solution from an appropriate finite element method for the variational problem. We also develop new techniques to build truncation schemes preserving the total of mass. Accuracy and efficiency of the proposed schemes are tested and compared using relevant numerical examples. The results clearly demonstrated the ability of the proposed schemes to resolve fourth-order nonlinear diffusion equations and inherit all the physical properties from the associated physical equations.  Some applications are presented including the spreading of droplet on a solid surface and  the modeling of phase separation in a binary mixture. In terms of accuracy and efficiency, the proposed method that uses the conservative truncation techniques appears to be the most promising scheme among all the studied methods.

\section*{Acknowledgements}
The first author was supported through a Postdoctoral Fellowship of the UM6P/OCP group of Morocco and a Postdoctoral Fellowship of the Fields Institute. The second author acknowledges funding from  UM6P/OCP group of Morocco. The research of the third author is funded through a Discovery Grant of the Natural Sciences and Engineering Research Council of Canada.

\bibliographystyle{plain} 
\bibliography{biblio}
\end{document}